\documentclass[11pt]{amsart}

\usepackage[bbgreekl]{mathbbol}
\usepackage{pb-diagram,amssymb,epic,eepic,verbatim,transparent}
\usepackage{tikz-cd,braket}
\usepackage{hyperref,graphics,epsfig,psfrag}
\hypersetup{colorlinks}
\usepackage{color}
\definecolor{darkred}{rgb}{0.5,0,0}
\definecolor{darkgreen}{rgb}{0,0.5,0}
\definecolor{darkblue}{rgb}{0,0,0.5}
\hypersetup{ colorlinks,
linkcolor=darkblue,
filecolor=darkgreen,
urlcolor=darkred,
citecolor=darkblue }
\newtheorem{theorem}{Theorem}[section]

\newtheorem{corollary}[theorem]{Corollary}

\newtheorem{proposition}[theorem]{Proposition}
\newtheorem{lemma}[theorem]{Lemma}
\newtheorem{lem}[theorem]{}
\theoremstyle{definition}
\newtheorem{definition}[theorem]{Definition}
\theoremstyle{remark}
\newtheorem{remark}[theorem]{Remark}
\newtheorem{example}[theorem]{Example}
\newcommand{\blem}{\begin{lem} \rm}
\newcommand{\elem}{\end{lem}}

\newcommand\cO{\mathcal{O}}
\newcommand\M{\mathcal{M}}

\renewcommand\M{\mathcal{M}}

\newcommand\fC{\mathfrak{C}}

\newcommand\LL{{\mathbb{L}}}
\newcommand\XX{\mathbb{X}}
\newcommand\DD{\mathbb{D}}
\newcommand\YY{\mathbb{Y}}

\newcommand{\R}{\mathbb{R}}
\renewcommand{\H}{\mathbb{H}}
\newcommand{\RR}{\mathcal{R}}

\newcommand{\C}{\mathbb{C}}

\newcommand{\Z}{\mathbb{Z}}
\newcommand{\Q}{\mathbb{Q}}

\newcommand{\ddt}{\frac{d}{dt}}

\renewcommand{\P}{\mathbb{P}}

\newcommand{\PP}{\mathcal{P}}

\newcommand\lie[1]{\mathfrak{#1}}

\renewcommand{\t}{\lie{t}}

\newcommand{\on}{\operatorname}

\newcommand{\trop}{\on{trop}}
\newcommand{\foc}{\on{foc}}
\newcommand{\Crit}{\on{Crit}}

\newcommand{\black}{\bullet} 
\newcommand{\white}{\circ} 

\newcommand{\fr}{{\on{fr}}}

\newcommand{\dual}{\vee}

\newcommand{\Edge}{\on{Edge}}

\newcommand{\Ver}{\on{Vert}}

\newcommand{\Aut}{ \on{Aut} } 
 
\newcommand{\aut}{ \on{aut} }

\newcommand{\Spec}{  \on{Spec}}

\newcommand{\codim}{\on{codim}}

\newcommand\dirac{/\kern-1.2ex\partial} 
\newcommand\qu{/\kern-.7ex/} 
\newcommand\hqu{/\kern-.7ex/\kern-.7ex/\kern-.7ex/}
\newcommand\lqu{\backslash \kern-.7ex \backslash} 

\newcommand\dr{r_+ \kern-.7ex - \kern-.7ex r_-}
 
\newcommand{\labell}\label

\renewcommand{\d}{{\on{d}}}
\newcommand{\ol}{\overline}

\newcommand\Phinv{\Phi^{-1}}

\newcommand\eps{\epsilon}

\newcommand{\f}{\frac}
\newcommand{\lan}{\langle}
\newcommand{\ran}{\rangle}
\newcommand{\hh}{{\f{1}{2}}}

\newcommand{\ti}{\tilde}

\newcommand\pt{\on{pt}}

\newcommand\cT{\mathcal{T}}

\newcommand\cP{\mathcal{P}}

\renewcommand{\ss}{{\on{ss}}}

\newcommand\ev{\on{ev}}

\newcommand\ul{\underline}
\newcommand\mO{\mathcal{O}}

\renewcommand\Im{\on{Im}}

\newcommand\bdefn{\begin{definition}}
\newcommand\edefn{\end{definition}}
\newcommand\bea{\begin{eqnarray*}}
\newcommand\eea{\end{eqnarray*}}
\newcommand\bcv{\left[ \begin{array}{r} }
\newcommand\ecv{\end{array} \right] }

\newcommand\bma{\left[ \begin{array}{l} }
\newcommand\ema{\end{array} \right]}
\newcommand\ben{\begin{enumerate}}
\newcommand\een{\end{enumerate}}
\newcommand\beq{\begin{equation}}
\newcommand\eeq{\end{equation}}
\newcommand\bex{\begin{example}}
\newcommand\bsj{\left\{ \begin{array}{rrr} }
\newcommand\esj{\end{array} \right\}}

\newcommand\Fuk{\on{Fuk}}

\newcommand\eex{\end{example}}

\newcommand\sx{*\kern-.5ex_X}

\newcommand{\Bl}{\on{Bl}}

\newcommand{\bGamma}{\mathbb{\Gamma}}
\newcommand{\bGam}{\bGamma}

\def\mathunderaccent#1{\let\theaccent#1\mathpalette\putaccentunder}
\def\putaccentunder#1#2{\oalign{$#1#2$\crcr\hidewidth \vbox
to.2ex{\hbox{$#1\theaccent{}$}\vss}\hidewidth}}

\usepackage{xr}
\begin{document}

\title[Holomorphic disks and tropical Lagrangians]{Holomorphic disks and tropical Lagrangians}

\author{Chris T. Woodward}

\thanks{This work was partially supported by NSF grant DMS 2105417.  Any
  opinions, findings, and conclusions or recommendations expressed in
  this material are those of the author(s) and do not necessarily
  reflect the views of the National Science Foundation.}

\begin{abstract}  We develop a calculus for counting pseudoholomorphic disks with boundary in tropical Lagrangians submanifolds of almost toric manifolds, using our previous work with Venugopalan \cite{vw:trop},
\cite{vw:at}.  The results are mostly in dimension four under monotonicity assumptions although in principle the same technique works in any dimension and without monotonicity.   The calculus is given as a sum over tropical graphs that interact with the tropical graph of the Lagrangian, generalizing results of 
Mikhalkin \cite{mikhalkin} and Nishinou-Siebert \cite{sn}
for holomorphic spheres in toric varieties, and our previous result with Venugopalan \cite{vw:at} which dealt with disks bounding almost toric moment fibers.   The main contribution of this paper is the calculation of 
several multiplicities of vertices corresponding to disks, such as the holomorphic pant (half of the holomorphic pair of pants) and various univalent vertices occuring at trivalent vertices of the graph of the Lagrangian; a key tool is a Lagrangian isotopy from the Lagrangian pair of pants in the del Pezzo of degree seven to the inverse image of a diagonal, which is a special case of results of Hind \cite{hind:spheres} and Evans \cite{evans:dp}.   We show that every integer eigenvalue of non-maximal modulus  for quantum multiplication by the first Chern class is realized by such a sphere. 
\end{abstract}

\maketitle

\tableofcontents

\section{Introduction}

Tropical geometry describes the limits of various geometric objects
under adiabatic limits, such as the limit of a Lagrangian torus
fibration when the volume of the fibers tends to zero.  By a {\em
  tropical Lagrangian} we mean a Lagrangian in a family that approximates some
piecewise linear subset of the base of such a fibration in the sense of Definition \ref{def:realize}.  We give in Theorem \ref{thm:realize} below a slight generalization of work  of Hicks \cite{hicks:realizability} and Mikhalkin
\cite{mikhalkin:examples} in the toric case showing  that such families can be associated to tropical graphs in the bases of almost toric fibrations.

In particular, del Pezzo surfaces equipped
with monotone symplectic forms (that is, symplectic forms $\omega$ for which $[\omega] = \lambda c_1(X)$ 
for some $\lambda > 0$) have systems of embedded Lagrangian
two-spheres whose tropical graphs have edges and nodes corresponding to the nodes and edges of the affine Coxeter-Dynkin diagrams introduced by Manin \cite{manin:cubic} at the level of cohomology. We call these {\em
  Manin configurations} of Lagrangian spheres.  Examples of tropical graphs associated to these configurations are shown
in Figure \ref{triang2}, where the relevant Lagrangians have tropical
graphs connecting two or three adjacent critical values in the almost
tori diagrams introduced by Vianna \cite{vianna:dp}.

\begin{figure}[ht]\begin{center} 
\scalebox{.7}{
\begingroup%
  \makeatletter%
  \providecommand\color[2][]{%
    \errmessage{(Inkscape) Color is used for the text in Inkscape, but the package 'color.sty' is not loaded}%
    \renewcommand\color[2][]{}%
  }%
  \providecommand\transparent[1]{%
    \errmessage{(Inkscape) Transparency is used (non-zero) for the text in Inkscape, but the package 'transparent.sty' is not loaded}%
    \renewcommand\transparent[1]{}%
  }%
  \providecommand\rotatebox[2]{#2}%
  \newcommand*\fsize{\dimexpr\f@size pt\relax}%
  \newcommand*\lineheight[1]{\fontsize{\fsize}{#1\fsize}\selectfont}%
  \ifx\svgwidth\undefined%
    \setlength{\unitlength}{378bp}%
    \ifx\svgscale\undefined%
      \relax%
    \else%
      \setlength{\unitlength}{\unitlength * \real{\svgscale}}%
    \fi%
  \else%
    \setlength{\unitlength}{\svgwidth}%
  \fi%
  \global\let\svgwidth\undefined%
  \global\let\svgscale\undefined%
  \makeatother%
  \begin{picture}(1,0.72222222)%
    \lineheight{1}%
    \setlength\tabcolsep{0pt}%
    \put(0,0){\includegraphics[width=\unitlength,page=1]{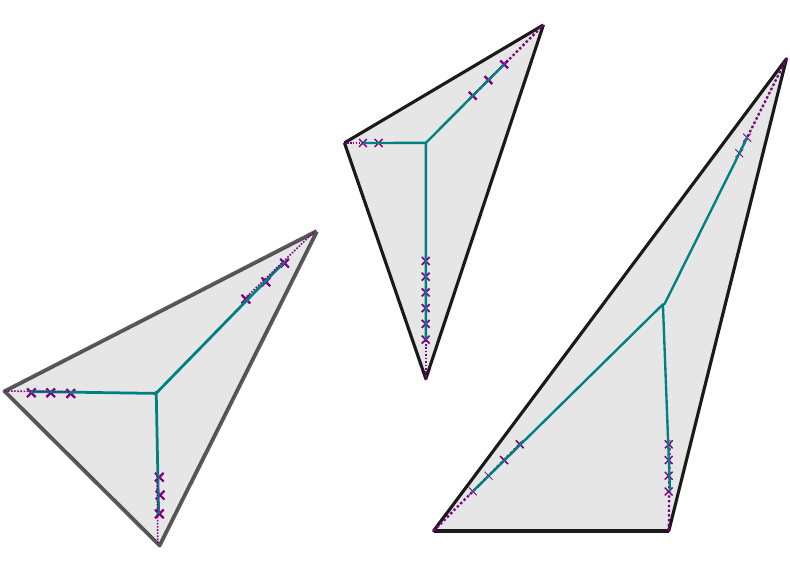}}%
    \put(0.05181848,0.36812922){\color[rgb]{0,0,0}\makebox(0,0)[lt]{\lineheight{1.25}\smash{\begin{tabular}[t]{l}$E_6$\end{tabular}}}}%
    \put(0.77123182,0.54776666){\color[rgb]{0,0,0}\makebox(0,0)[lt]{\lineheight{1.25}\smash{\begin{tabular}[t]{l}$E_7$\end{tabular}}}}%
    \put(0.47845738,0.6738583){\color[rgb]{0,0,0}\makebox(0,0)[lt]{\lineheight{1.25}\smash{\begin{tabular}[t]{l}$E_8$\end{tabular}}}}%
  \end{picture}%
\endgroup%
}
\end{center} 
\caption{Graphs of Manin configurations in del Pezzo's of exceptional type}
\label{triang2}
\end{figure}
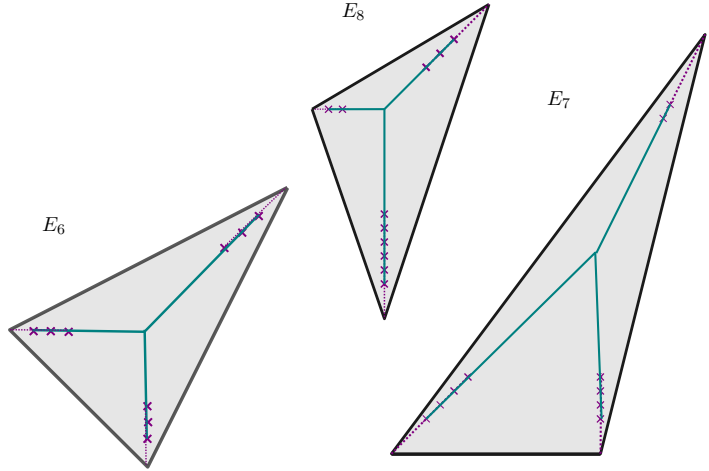

 In this paper, we develop tropical techniques for counting disks with boundary in tropical Lagrangians, and in particular, computing the disk potentials of the Manin Lagrangians in del Pezzo surfaces.    By the disk potential, we mean the count $W_L \in \Z$ of rigid holomorphic disks passing through a generic point as in Oh \cite{oh:floer}, where the moduli spaces are oriented using the unique relative spin structure on these Lagrangian spheres as in \cite{fooo:part1}.  By Sheridan
 \cite[Corollary 2.10]{sheridan:hypersurface}, each such integer is an eigenvalue of quantum multiplication by the first Chern class.  We
 prove the following:

\begin{theorem} \label{thm:dp} Let $X$ be a del Pezzo surface equipped with a monotone symplectic form.  
Every integer eigenvalue $w \in \Z$ of quantum multiplication by the first Chern class
$c_1(X)$ on $QH(X)$ whose absolute value  $|w|$ is not maximal is  
the potential $W_L$ of a Lagrangian sphere $L \subset X$,  whose homology class
represents a simple root 
$\alpha \in S$ for some semisimple root system of $ADE$ type in a Manin configuration in $X$.
\end{theorem} 

The eigenvalues of quantum multiplication by the first Chern class are shown in
Table \ref{tab:eigen} (with critical values computed numerically with Mathematica).
The notation $\Bl^k \P^2$ means $\P^2$ blown-up $k$ times at generic points.
\begin{table}
\[ 
\begin{array}{|l|l|l|}X   &  \text{Manin Root System} &   \approx \Spec(c_1 \star) \\
\hline 
\P^2                 & & 3\alpha, \alpha^3 = 1 \\
\P^1 \times \P^1     & A_1 & 4, 0^{\oplus 2}, -4 \\ 
\Bl^1 \P^2           & & -0.33, 3.8, -2.23 \pm 1.94 \sqrt{-1} \\
\Bl^2 \P^2            & A_1 & (-1)^{\oplus 2}, 4.73, -2.86 \pm 0.94 \sqrt{-1} \\
\Bl^3 \P^2            & A_1 \oplus A_2  &  (-2)^{\oplus 3}, (-3)^{\oplus 2} ,6 \\
\Bl^4 \P^2             & A_4 & (-3)^{\oplus 5}, 8.09, -3.09 \\
\Bl^5 \P^2            & D_5 & (-4)^{\oplus 7}, 12 \\
\Bl^6 \P^2             & E_6 & (-6)^{\oplus 8}, 21 \\
\Bl^7 \P^2            & E_7 &  (-12)^{\oplus 9}, 52\\
\Bl^8 \P^2 & E_8 &  (-60)^{\oplus 10} , 372. 
\end{array} \]
\caption{Eigenvalues of the first Chern class on $QH(X)$}
\label{tab:eigen}
\end{table}

There is another way of phrasing Theorem \ref{thm:dp}.  Given an integer eigenvalue $w$, one can ask whether  it is the potential of some Lagrangian sphere.  Since it is quite obvious that the open-closed map on the Floer cohomology of such a sphere is dimension at least two (the inclusion map on classical homology is injective), the corresponding eigenspace must have multiplicity at least two.  Since del Pezzo surfaces satisfy conjecture $\cO$, by Hu-Ke-Li-Yang \cite{hu:gammaconj}, only the integer eigenvalues with non-maximal modulus have a chance of corresponding to Lagrangian two-spheres.   What we show is that  every such eigenvalue is in fact represented by a sphere, in the case of del Pezzo surfaces.

The proof of Theorem \ref{thm:dp} is not quite given in this paper; what we show here is
that each of the eigenvalues $w$ in Table \ref{tab:eigen} with $|w| \in \Z$ non-maximal is realized by a Lagrangian sphere. 
To see that the numbers in the last column of Table \ref{tab:eigen} are all the eigenvalues of quantum multiplication by $c_1(X)$, one either needs a presentation of the quantum cohomology with the first Chern class identified in the presentation (as is available for the cubic surface, as in Sheridan \cite[Appendix]{sheridan:hypersurface}) or know that 
the Lagrangians split-generate the Fukaya category (as we show in a sequel to this paper.)
Using extensions of these results to immersed Lagrangians with transverse self-intersection, we compute  in a sequel \cite{wo:man} the open-closed maps for Manin configurations of Lagrangians in  almost toric manifolds, and show split-generation of the corresponding Fukaya eigencategories.

We remark that tropical formulas for disk counts were already considered in Mikhalkin \cite{mikh:real}, while tropical formulas for disk counts in the case of almost toric fibers were studied in \cite{vw:at}.   In the former case, the image of the Lagrangian is dimension two (being the real locus) while in the latter case, the image of the Lagrangian is one; therefore, the results of this paper treat the intermediate case. 

\subsection{Disk counts for tropical Lagrangians}

We describe in more detail the disk counts that we wish to compute.  Let $X$ be a compact connected almost toric four-manifold.  
Thus $X$ admits a singular Lagrangian fibration
over base $B$ with projection $\Phi:X \to B$, as defined in Leung-Symington \cite{leungsym}.  
We call $\Phi$ the {\em almost toric moment map} and 
\[ \Delta := \Phi(X) \subset B \] 
the {\em moment polytope}. 
The union of nodal singularities in the fibers, called {\em focus-focus singularities}, is denoted $X^{\foc} \subset X$.    Let 
\[ B^{\on{foc}}  = \Phi(X^{\foc}) \subset B \]
denote the images of the focus-focus
singularities under the projection $\Phi$, called the {\em focus-focus values}.  We focus on the case $\dim(X) =4$, which we now assume. 
We assume that $\Phi$ is proper with connected fibers, so that $\Phi(X)$ is locally polyhedral with respect to the affine structure on the complement of the focus-focus values.   By a {\em tropical Lagrangian} we mean a Lagrangian $L$ whose image under the projection to the base
approximates some tropical graph 
\[ \Lambda = (\Ver(\Lambda),  \Edge(\Lambda)) \] 
in the moment polytope $B$ of $X$; see Definition
\ref{def:realize} for a precise definition.  In the Figures such as Figure \ref{fig:b5p2_sphere}, the tropical graph $\Lambda$ of the Lagrangian is shown in aqua, while the tropical graphs $\Gamma$ of the holomorphic disks are shown in dark blue.    We give a minor generalization of a result of Mikhalkin \cite{mikhalkin:examples} and Matessi  \cite{matessi:pants}, \cite{matessi:trop} we explain  how tropical graphs $\Gamma$ in the base $B$ give rise to Lagrangian submanifolds $L$ in the symplectic manifold $X$. 
Let
\[ \cP = \{ P \subset
B \} \] 
be a polyhedral decomposition of $B$ into simple rational polytopes; that is, 
the intersection of finitely many rational half-spaces so that at most $\dim(X)$
of the boundaries of these half-spaces meet at any point.  The quotient construction for the circle actions generated by the components of  the moment map 
normal to the facets gives a degeneration denoted $\XX$.  We 
choose particularly simple decompositions in order to compute disk potentials:

\begin{definition}  \label{def:elem} 
The decomposition $\cP$ is  {\em elementary} if
  one of the possibilities holds exclusively: Each $P \in \cP$ either
  \begin{enumerate}
    \item  contains at most one
  focus-focus value $b \in B^{\foc}$  and in which case is equivalent
  to the polytope shown after item \eqref{mv:multcov}, up to the action of
  $GL(2,\Z)$;
  \item  contains the projection $\ell = \Phi(L)$ of the Lagrangian
    $L$;
    \item intersects the boundary of
  $\Phi(X)$ and contains at most one vertex of $\Phi(X)$; in this case $P \cap \Phi(X)$
  is a trapezoid and, if it contains a vertex, $P \cap \Phi(X)$ is a parallelogram.
\end{enumerate}
\end{definition}

Any polyhedral decomposition of the type above induces a decomposition and degeneration 
of the symplectic manifold as follows.  First, the manifold $X$ is a disjoint union of the preimages
of the polyhedra in the collection $\cP$:
\[X = \bigcup_{P \in \cP} \Phinv(P)  .\] 
On the other hand, we also have a degeneration (loosely speaking) of $X$ 
of obtained by identifying the points in $X$ by the action of a torus $T_P$ over the elements of $\cP$.  We say that the polyhedral decomposition is {\em admissible} if these
groups act with only finite stabilizers, so that the quotient 
\[ X/_{\sim} = \bigcup_{P \in \cP} \Phinv(P)/T_P \]  
is a union of symplectic orbifolds, as explained in \cite{vw:trop}.  For each such polyhedral decomposition, we choose a set of {\em cutting datum}, consisting of a {\em dual polytope} $P^\dual$ each $P$ in $B$, as described in \cite{vw:trop}.    The {\em dual complex}
associated to the cutting datum is the union of dual polytopes, modulo the face equivalence relation:
\[ B^\dual = \bigcup_{P \in \cP} P^\dual / \sim .\]
By the {\em boundary} of the dual complex we mean the union 
\[ \partial B^\dual := \bigcup_{P \in \cP, P \cap \partial \Delta \neq \emptyset} P^\dual / \sim \]
of duals of polytopes meeting the boundary  $\partial \Delta$ of the moment polytope $\Delta$.
One sees easily from the definitions that $B^\dual$ is a topological manifold with boundary
$\partial B^\dual$, with an affine structure on the complement of the duals $P^\dual$
of polytopes containing the focus-focus values.   We denote by 
\begin{equation} \label{eq:lambdadual}
\Lambda^\dual = \bigcup_{P \in \cP, P \cap \Lambda \neq \emptyset} P^\dual / \sim 
\end{equation}
the part of the dual complex corresponding to pieces meeting the Lagrangian; by our assumption, $\Lambda^\dual$ is a one dimensional complex in $B^\dual$, with the same topology as $\Lambda$.
However, each edge in $\Lambda$ dualizes to a collection of edges in $\Lambda^\dual$ which may not have the same direction, with respect to the dual affine structure.

A cutting datum gives rise to a {\em neck-stretched} family of almost complex structures.  In the neck-stretching limit corresponding to the decomposition, any sequence of holomorphic maps with bounded energy  limits, after passing to a subsequence, to a broken map:
 a collection of holomorphic maps $u_v$ associated to the vertices $v$ of a 
 graph $\Gamma$ with vertices $\Ver(\Gamma)$ and edges $\Edge(\Gamma)$. The graph  $\Gamma$ is equipped with a map
\[ \Ver(\Gamma) \to \cP , \quad v \mapsto P(v) .\]
Each edge $e$ of $\Gamma$ connecting vertices $v_-$ and $v_+$ has an assigned {\em direction}
\[ \cT(e) \in \t_{P(e)} , \quad P(e) := P(v_-) \cap P(v_+) \] 
where $\t_{P(e)} $ is a certain vector space to be defined later, dual to the 
span of the polytope $P(e)$.   
If $e$ connects two vertices $v_-,v_+$, we call $e$ an {\em ordinary edge}.   If $e$ contains a single vertex $v$ we call $e$ a {\em leaf}.    A picture of some of the tropical graphs $\Gamma$ associated to holomorphic disks in a del Pezzo of degree four is shown 
on the right in Figure \ref{fig:b5p2_sphere}; in the depiction we have overlayed the dual polyhedral complex with the original polyhedral complex; otherwise the pictures become difficult to visualize.  The left picture in Figure \ref{fig:b5p2_sphere} shows the almost toric diagram and the center picture shows the tropical graph of the Lagrangian.  The right figures show the four tropical graphs of the holomorphic disks with boundary in the  Lagrangian and passing through the given constraint (shown as a burgundy dot on top of the graph of the Lagrangian.)  The potential is calculated
in Example \ref{ex:b5p2}.  
\begin{figure}[ht]
    \centering
    \scalebox{.4}{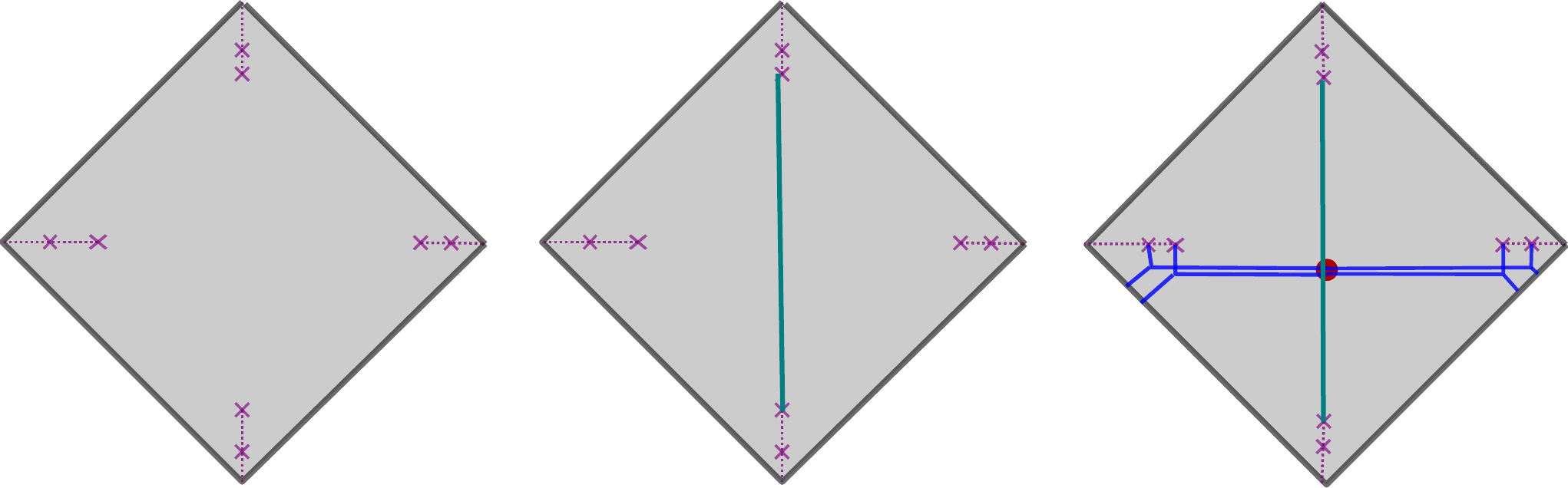}
    \caption{Left:  The base diagram of a fibration  for the degree four del Pezzo, with focus-focus values as crosses (purple).  Center: The tropical graph (green) of a Lagrangian sphere.  Right:  Cartoon diagrams (blue) of Maslov-index-two disks with boundary in the  Lagrangian sphere, showing that the potential of the Lagrangian is $-4$.}
    \label{fig:b5p2_sphere}
\end{figure}
For each vertex $v$ in $\Ver(\Gamma)$ there is a holomorphic disk or sphere 
\[ u_v: C_v \to \XX_{P(v)}  \] 
in some target $\XX_{P(v)}$ obtained by taking the thickening of the inverse image of $P(v)$ under the moment map.     We denote by $\Ver_\white(\Gamma)$ resp. $\Ver_\black(\Gamma)$ the 
{\em open resp. closed} vertices corresponding to disks and spheres respectively.  The edges correspond to long cylinders
or strips with boundary in the Lagrangian connecting the components $u_v$; we denote by $\Edge_\black(\Gamma)$
resp. $\Edge_\white(\Gamma)$ the corresponding sets of {\em closed resp. open} edges. 
For example, the map $u_v$ may be a {\em holomorphic pant} (the singular is intentional here) by which we mean a holomorphic disk with a single cylindrical end and a single strip-like end pictured on the right in Figure \ref{fig:antidiag2}, with the corresponding bivalent vertex of a tropical graph pictured on the left.   This curve is half of the holomorphic pair-of-pants considered in Mikhalkin \cite{mikhalkin} and, perhaps not surprisingly, counts in the formulas with a factor of half of the determinant that appears in \cite{mikhalkin}.

\begin{figure}[ht]\begin{center} \scalebox{1.1}{\includegraphics{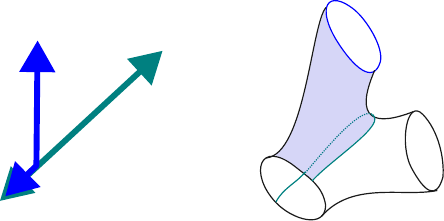}}
\end{center} \caption{The holomorphic pant (right in blue) and its tropical graph (left in blue)}  \label{antidiag3}
\end{figure}

The broken maps are required to satisfy matching conditions at the edges.  In our previous work \cite{vw:trop}, we performed an additional degeneration so that the diagonal condition at the edges was deformed to split form.  We show in Theorem \ref{thm:split} below 
that the moduli space $\M_\bGamma(\LL)$ of broken maps with type $\bGamma$ is a product 
\[ \M_\bGamma(\LL) \cong \prod_{v \in \Ver(\bGamma)} \M_{\bGamma(v)}(\LL_{P(v)},\YY_v) \]
where 
\[ \ul{\YY}_v = (\ul{\YY}_{v,\white}, \ul{\YY}_{v,\black}) \] 
is a collection of  {\em constraints}
attached to the leaves of $\Gamma(v)$ for each vertex $v$; in the current situation, these
constraints will be point constraints.  The counts of broken maps then reduce to counts of tropical graphs $\Gamma$ with multiplicities $m(\Gamma)$.   If the tropical graph $\bGamma$ of the curve is trivalent and satisfies the conditions above,  we assign to $\bGamma$ a {\em multiplicity} $m(\bGamma)$ that is, up to a combinatorial factor, a product of multiplicities assigned to vertices 
\[ m(\bGamma) = (d_\black(\bGamma))^{-1} \prod_{v \in \Ver(\bGamma)}  m(v)  \in \Q ; \]
here $d_\black(\bGamma) \in \Z_{> 0}$ is the number of interior leaves corresponding to intersections with a Donaldson hypersurface $D \subset X$ used for the regularization of pseudoholomorphic curves.  In case $\bGamma$ has vertices $v$ with valence $|v|$ greater than $3$, (or greater than two and at least two of the adjacent edges are of closed type) the multiplicity $m(\bGamma)$ is defined by a desingularization procedure explained in \cite{vw:at}, which perturbs the positions of the closed edges.    Given a collection 
$ \ul{\YY} = (\ul{\YY}_\black, \ul{\YY}_\white)$ of such constraints at the boundary leaves and interior leaves respectively, we consider the moduli space $\M(\LL,\ul{\YY})$ of broken maps bounding $\LL$ with the corresponding constraints $\ul{\YY}$ at the leaves.   We have in mind especially the disk counts that contribute to the disk potentials and open-closed map, and denote
by 
\[ m(\LL,\ul{\YY}) \in \Q \] 
the weighted count of its elements.  In particular, the disk potential of a monotone
Lagrangian is given by the count $W_\LL$ of disks with a single point constraint, independent of the choice of
point $\pt_L \in L$.

We will be particularly interested in the following graphs for which we can compute
the contributions using the main result:

\begin{definition} \label{def:ci} 
A tropical graph $\Gamma$  in $B^\dual$ has collisions {\em in the interior}
  if for any vertex of $\Gamma$ for which the polytope $P(v)$ intersects $\partial \Phi(X)$,
  \begin{enumerate}
  \item the valence $|v|$ of $v$ is one and 
  \item $P(v)$ intersects exactly a single facet $Q$ of $\Phi(X)$ and the direction $\cT(e) 
  \in \t$ of the adjacent edge $e \in \Edge(\Gamma)$ is the primitive normal to $Q$.  
  \end{enumerate}
  A tropical graph $\Gamma$  in $B^\dual$ has collisions only at {\em smooth points} if for any vertex $v \in \Ver(\Gamma)$, the cut space $X_{P(v)}$ contains a focus-focus singularity, then the valence $|v|$ of $v$ is one.  A tropical graph $\Gamma$ is {\em generic} if $\Gamma$ has only collisions in the interior and the valence $|v|$ of each vertex $v \in \Ver(\Gamma)$ is at most three, and $|v|$  is at most two if there is a single open edge adjacent to $v$.
\end{definition}

\begin{theorem} \label{thm:mult}
Let $X$ be a compact monotone almost toric four-manifold with base $B$.   Let $L$ be an embedded Lagrangian sphere that realizes a tropical graph $\Lambda \subset B$.  Let $\cP = \{ P \}$ be a polyhedral admissible decomposition of $B$  equipped with cutting data. 
The count $m(\LL,\ul{\YY})$  of broken rigid  disks in $\XX$ with boundary in the  broken Lagrangian $\LL$ with a point constraint $\ul{\YY} = (\{ \pt \}) \subset \LL$ is a count of types $\bGamma$ in the dual complex $B^\dual$:
\[  m(\LL,\ul{\YY}) = \sum_{\bGamma \in \cT(\ul{\YY})} (\# \Aut(\bGamma))^{-1}  \prod_{v \in \Ver(\bGamma)} m(\LL_v,\ul{\YY}_v) \]
where $\cT(\ul{\YY})$ is the set of rigid graphs in $B^\dual$ with constraints $\ul{\YY}$, and $m(\LL_v,\ul{\YY}_v)$ is a count of holomorphic disks in $\XX_{P(v)}$ with some collection of constraints $\ul{\YY}_v$.  Furthermore, if $\bGamma$ is generic  then each $m(v) = m(\LL_v,\ul{\YY}_v)$ is a  multiplicity given  in Definition \ref{def:mv}.
\end{theorem}

We explain how to obtain this Theorem by modifying the arguments in 
Venugopalan-Woodward \cite{vw:trop} in Section \ref{sec:limit}. 
We remark that in contrast to the case of toric moment fibers in \cite{vw:at}, 
it is not possible to achieve the valence conditions stated in the theorem 
through perturbation.  Indeed, the positions of the open edges of the tropical graph $\bGamma$ cannot be perturbed, since they coincide with open edges of the tropical graph $\Lambda$ of the Lagrangian.  It follows that the condition that the valence of each vertex is at most three cannot be achieved
by perturbation.  As a result, the list in Theorem \ref{thm:mult} is only partial
as we do not compute the higher valence multiplicities.   Nevertheless, the case of vertices of valence at most three will suffice  to compute the disk potential for the Lagrangian spheres considered here.

\subsection{Multiplicities of tropical vertices} 

The multiplicities of certain vertices are given as follows, and will be shown to be counts of holomorphic disks bounding
the Lagrangian: (The meaning of the dots in the pictures will be explained later.)   The multiplicities are non-canonical; they depend on the choices of relative spin structures on certain extensions of the boundary value problems over the punctures, and in the case of the holomorphic pant, on a particular choice of complex structure and perturbation.

\begin{definition} \label{def:mv} Define multiplicities $m(v) \in \Q$ associated to vertices 
$v \in \Ver(\Gamma)$ with valence $|v|$ at most three
as follows:  (Each item heading describes the corresponding holomorphic curves.)
\begin{enumerate}

    \item \label{mv:spheres} {\rm (Holomorphic cylinders)} The multiplicity 
\[ m(v)= 1 \] 
for bivalent vertices $v$ with 
     adjacent edges  $e_1,e_2 \in \Edge_\black(\Gamma)$ in the same directions
    \[ e_1,e_2 \in \Edge_\black(\Gamma), \quad \cT(e_1) = \cT(e_2). \]

    \begin{figure}[ht]
    \begin{center} 
\scalebox{.9}{\includegraphics{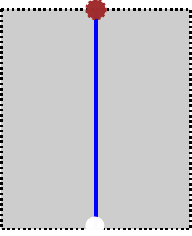}}
\end{center}
        \caption{Tropical graph for a holomorphic cylinder}
        \label{fig:mvspheres}
    \end{figure}

    \item \label{mv:pants} {\rm (Holomorphic pairs of pants)}  
    As in Mikhalkin \cite{mikhalkin}, the multiplicity 
\[ m(v)= |\det(\cT(e_1) \cT(e_2))| \] 
for trivalent vertices $v \in \Ver_\black(\Gamma)$ with adjacent edges $e_1,e_2, e_3 \in \Edge(\Gamma)$ whose edge directions (exactly two of which are incoming) satisfying the balancing condition
    \begin{equation} \label{eq:balance}
    \quad \cT(e_1) + \cT(e_2) + \cT(e_3) = 0 ; \end{equation}

    \begin{figure}[ht]
  \begin{center} 
\scalebox{.9}{
\begingroup%
  \makeatletter%
  \providecommand\color[2][]{%
    \errmessage{(Inkscape) Color is used for the text in Inkscape, but the package 'color.sty' is not loaded}%
    \renewcommand\color[2][]{}%
  }%
  \providecommand\transparent[1]{%
    \errmessage{(Inkscape) Transparency is used (non-zero) for the text in Inkscape, but the package 'transparent.sty' is not loaded}%
    \renewcommand\transparent[1]{}%
  }%
  \providecommand\rotatebox[2]{#2}%
  \newcommand*\fsize{\dimexpr\f@size pt\relax}%
  \newcommand*\lineheight[1]{\fontsize{\fsize}{#1\fsize}\selectfont}%
  \ifx\svgwidth\undefined%
    \setlength{\unitlength}{138.12511826bp}%
    \ifx\svgscale\undefined%
      \relax%
    \else%
      \setlength{\unitlength}{\unitlength * \real{\svgscale}}%
    \fi%
  \else%
    \setlength{\unitlength}{\svgwidth}%
  \fi%
  \global\let\svgwidth\undefined%
  \global\let\svgscale\undefined%
  \makeatother%
  \begin{picture}(1,0.83121261)%
    \lineheight{1}%
    \setlength\tabcolsep{0pt}%
    \put(0,0){\includegraphics[width=\unitlength,page=1]{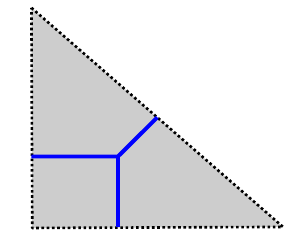}}%
     \put(0,0){\includegraphics[width=\unitlength,page=2]{mv1.pdf}}%
    \put(0.16503584,0.34118367){\color[rgb]{0,0,0}\makebox(0,0)[lt]{\lineheight{1.25}\smash{\begin{tabular}[t]{l}$e_1$\end{tabular}}}}%
    \put(0.46137506,0.16085866){\color[rgb]{0,0,0}\makebox(0,0)[lt]{\lineheight{1.25}\smash{\begin{tabular}[t]{l}$e_2$\end{tabular}}}}%
    \put(0.40937953,0.42263157){\color[rgb]{0,0,0}\makebox(0,0)[lt]{\lineheight{1.25}\smash{\begin{tabular}[t]{l}$e_3$\end{tabular}}}}%
  \end{picture}%
\endgroup%
}
\end{center}
        \centering
        \caption{Tropical graph for a holomorphic pair of pants}
        \label{fig:mvpants}
    \end{figure}

\item \label{mv:bound} {\rm (Disks meeting the toric boundary)} As in Cho-Oh \cite{chooh:fano}, the multiplicity 
\[ m(v) = 1 \] 
if $v \in \Ver_\black(\Gamma)$ is a univalent vertex at a facet $Q$ of the toric boundary $\partial \Delta$ whose adjacent edge 
$e$ has direction $\cT(e)  \in \t_{P(v),\Z}$ is the primitive generator of the normal direction to $TQ$.

    \begin{figure}[ht]
\begin{center} 
\scalebox{.5}{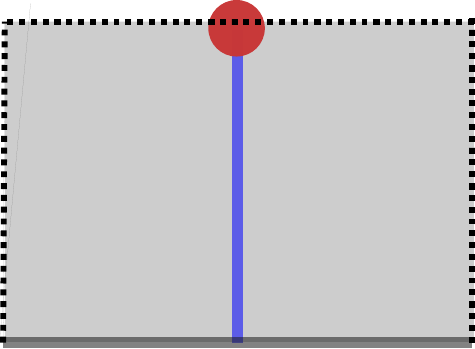}
\end{center}
        \centering
        \caption{Tropical graph for collisions with the boundary}
    \end{figure}

\item \label{mv:multcov} {\rm (Multiple covers near focus-focus singularities)} 
As in Bryan-Pandharipande \cite{bryan:lgw}, the multiplicity 
\[ m(v) = (-1)^{\ell(\mu)-1}/\ell(\mu)^2 \] 
for the univalent closed vertices
$v \in \Ver_\black(\Gamma)$ at the focus-focus singularities, where 
\[ \ell(\mu) \in \Z_{> 0} \] 
is the lattice length of the 
direction $\mu = \cT(e) \in \t_{P(v),\Z}$ of the adjacent edge $e \in \Edge(\Gamma)$;
that is, $\mu = \ell(\mu) \mu_0$ where $\mu_0$ is the primitive lattice vector in the same direction as $\mu$;

    \begin{figure}[ht]
\begin{center} 
\scalebox{1}{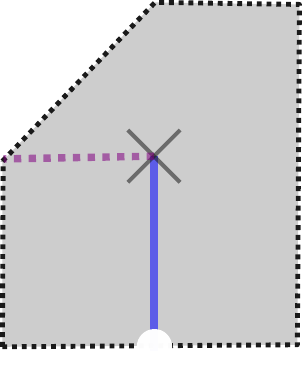}
\end{center}
        \caption{Tropical graph for collisions with the focus-focus values}
    \end{figure}
    
\item \label{mv:cylhit}{\rm (Cylinders colliding with the Lagrangian perpendicularly)}  
The multiplicity 
\[ m(v) = (-1)^{\ell(e)}  \] 
if $v \in \Ver_\white(\Gamma)$ is a univalent vertex 
whose polytope $P(v)$ meets edge $\eps \in \Edge(\Lambda)$ of the Lagrangian $L$ with
the adjacent edge $e \in \Edge(\Gamma)$ having direction $\cT(e) \in \t_{P(v),\Z}$ 
with lattice length 
\[ \ell(e) \in \Z_{> 0} \] 
perpendicular to the direction
$\cT(\eps)$ of $\eps \in \Edge(\Lambda)$.

    \begin{figure}[ht]
\begin{center} 
\scalebox{.5}{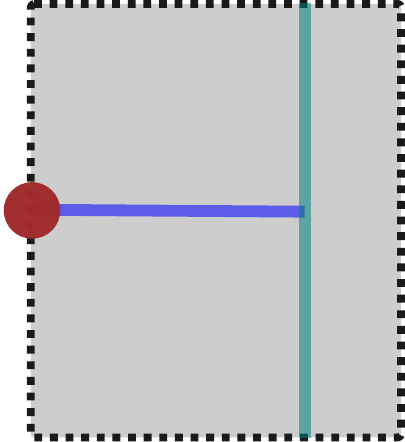}
\end{center}
        \caption{Tropical graph for collisions with Lagrangian}
    \end{figure}

\item \label{mv:pant} {\rm (Holomorphic pant, Lagrangian with direction $(\pm 1,1)$)}  
The multiplicity 
\begin{eqnarray*}   
m(v) &=& (-1)^{    | \det(\cT(e_\white) \cT(e_\black))|  - 1/2  }  | \det(\cT(e_\white) \cT(e_\black))|   \\ &=&\hh (-1)^{
| 2 \det(\cT(e_\black) r(\cT(e_\black))|  - 1 } 
| \det(\cT(e_\black) r(\cT(e_\black))|   
\end{eqnarray*}
for bivalent vertices $v \in \Ver_\white(\Gamma)$ mapping
to an  edge $\eps$ of $L$ with open edge  $e_\white  = (1,1) \subset \eps$ and closed
edges $ e_\black, e_\black'$ with directions 
\[ \cT(e_\black) = (x,y),r(\cT(e_\black)) = \cT(e_\black') = (y,x) \in \Z^2 \] 
for some integers $x,y$, with directions 
$\cT(e_\white) \in \t_{P(v)}, \cT(e_\black) \in \t_{P(v),\Z} $
satisfying the balancing condition 
\[  2 \cT(e_\white) + \cT(e_\black) + r(\cT(e_\black)) = (1 + x + y, 1 + x + y ) =  0  \]
where $r: \R^2 \to \R^2$ is the reflection over the span of $\eps$.\footnote{Note that we insist that the direction of the edge of the graph of the Lagrangian is the diagonal or antidiagonal.  Of course, any edge may be put in this position by some element $A$ of $SL(2,\Z).$   However, the same transformation acts by the transpose $A^t$ on the directions of the edges of the graph of the broken map, so there seems to be only a clean statement in the diagonal or anti-diagonal cases. } 

 \begin{figure}[ht]
\begin{center} 
\scalebox{1.1}{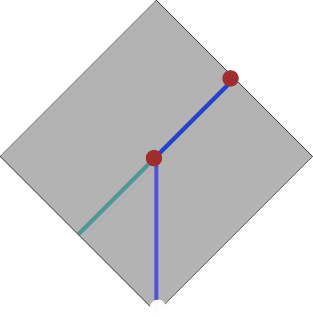}
\end{center}
        \caption{Tropical graph for a holomorphic pant}
    \end{figure}

\item {\rm (Disks dividing Lagrangian pants in half)}  \label{mv:pantseam}  The multiplicity 
\[ m(v) = 1 \]
for univalent vertices $v \in \Ver_\white(\Gamma)$ mapping 
to a trivalent vertex $\ell \in \Ver(\Lambda)$ with adjacent edge $e \in \Edge_\white(\Gamma)$ 
having direction $\cT(e)$ the same as the direction $\cT(\eps)$ of one of the adjacent edges $\eps \in \Edge(\Lambda)$ of the tropical graph $\Lambda$ of $L$.

 \begin{figure}[ht]
\begin{center} 
\scalebox{.5}{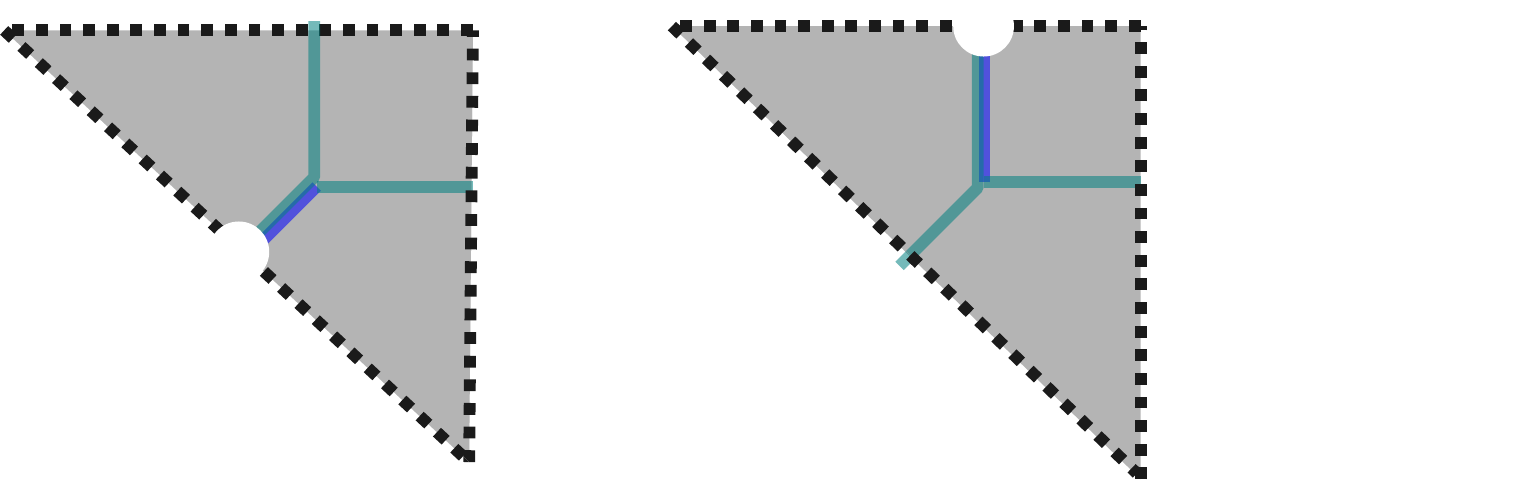}
\end{center}
        \caption{Cartoon diagrams for disks with boundary in a Lagrangian pair of pants with a single strip-like end}
    \end{figure}

\item  \label{mv:twoend}  {\rm (Two-ended strips in Lagrangian pants)}
The multiplicity 
\[ m(v) = 1 \] 
for bivalent vertices $v \in \Ver_\white(\Gamma)$ mapping 
to trivalent vertices $\ell \in \Ver(\Lambda)$ with edge directions $\cT(e_1), \cT(e_2)$ of adjacent edges $e_1 \in \Edge_\white(\Gamma),  e_2 \in \Edge_\white(\Gamma)$ with
directions 
\[ \cT(e_1) = \cT(\eps_1) ,\cT(e_2) = \cT(\eps_2) \in \t_{P(v)} \]
equal to the directions $\cT(\eps_1),\cT(\eps_2)$ of adjacent edges of the graph $\Lambda$ of $L$, and either one point boundary constraint $p \in L$ or one point constraint $p \in \RR_{\white}(P(e_i)), i \in \{ 1,2 \} $  on the Reeb orbit  at a strip-like end $e_1$ or $e_2$.  The situation is depicted in the figure at left below:

 \begin{figure}[ht]
\begin{center} 
\scalebox{.5}{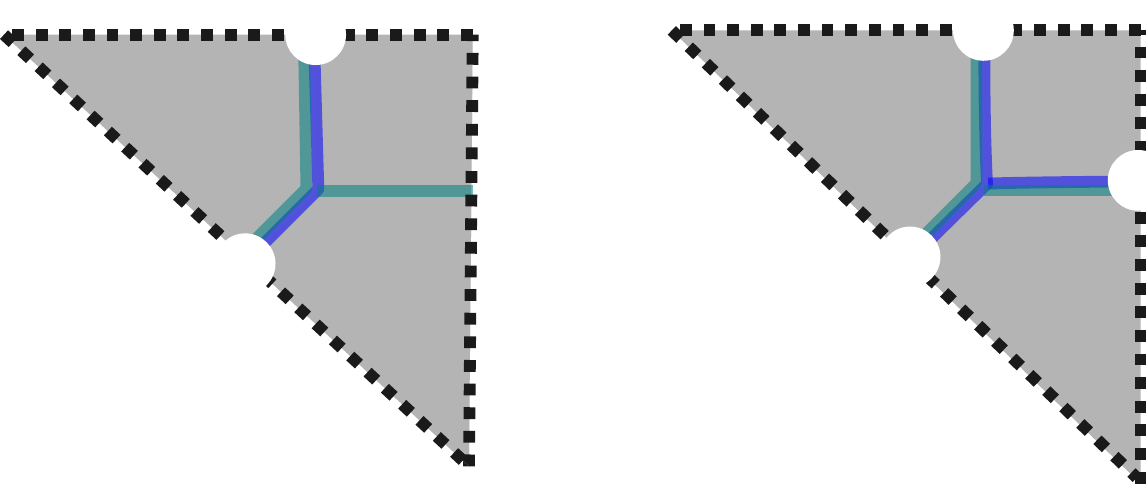}
\end{center}
        \caption{Cartoon diagrams for disks with boundary in a Lagrangian pair of pants with two or three strip-like ends}
    \end{figure}

\item \label{mv:threeend}  {\rm (Three-ended strips in Lagrangian pants)}  The multiplicity 
\[ m(v) = 1 \] 
for trivalent vertices $v \in \Ver_\white(\Gamma)$ with adjacent
edges $e_1,e_2,e_3$ having directions $\cT(e_1), \cT(e_2), \cT(e_3) \in \t_{P(v)}$
equal to the directions of the adjacent edges $\eps_1,\eps_2,\eps_3 \in \Edge(\Lambda)$ of $L$ with a point constraint
$p \in L$ or a point constraint at a Reeb chord at one of the strip-like ends. 
(Recall that the Mikhalkin multiplicity $m(v) = 1$ of $v \in \Ver(\Lambda)$ by assumption.)  The figure is included at right in the previous item.
%
%



\end{enumerate}
\end{definition}

The appearance of the holomorphic pant, i.e., half-pair-of-pants seems to be new in the literature on disks bounding tropical Lagrangians, which was somewhat surprising to us.

\subsection{Sample computations of potentials}

As examples, we compute the disk potential for some tropical Lagrangians in del Pezzo surfaces.

\begin{example}\label{ex:3p2}  We compute the disk potential for a tropical Lagrangian in the blow-up $X$ of the product of projective lines $\P^1 \times \P^1$ at $(0,0)$ and $(\infty,\infty)$.
The moment polytope may be taken to be the hexagon that is the convex hull 
\[ \Phi(X) = \on{hull} \{ (-1,0), (0,-1), (1,0), (0,1), (1,-1), (-1,1) .\} \] 
Consider the Lagrangian 
\[ L_0 = \{ ([z_1,z_2], [\ol{z}_2,\ol{z}_1] ) | (z_1,z_2) \in \C^2 - \{ 0 \} \} \subset \P^1 \times \P^1 \] 
which is the graph of the anti-symplectic involution
\[ \tau: X \to X, \quad [z_1,z_2] \mapsto [\ol{z}_2, \ol{z}_1]. \] 
The Lagrangian $L_0$ is disjoint from the cut locus in $\P^1 \times \P^1$ and lifts to a Lagrangian $L$ in the blow-up $X$. 

Since $X$ is toric, the codomain of the moment map is $B = \R^2$.  
 The moment image of $L$ is the segment $\Lambda$ between the vertices $(1,-1)$ and $(-1,1)$, as shown on the left in Figure \ref{fig:3p2wdual}.    In particular, $L$ is a realization of $\Lambda$
in the sense of Mikhalkin \cite{mikhalkin:examples}; see Theorem \ref{thm:realize}. 

Let $p \in L $ be a point mapping to $(0,0)$ under the moment map.   Let $\cP = \{ P \} $ be an elementary polyhedral decomposition of $\Phi(X)$ in the sense of Definition \ref{def:elem}, with the property that there exists an element $P_0 \in \cP$ so that $P_0$ contains the image $\Phi(p)$ but no part of the boundary $\partial \Phi(X)$.   The dual complex $ B^\dual $ shown on the right in Figure \ref{fig:3p2wdual}; the dual graph $\Lambda^\dual$ is shown in aqua.   The dual complex $B^\dual$ may be identified with a subset of the dual of $\R^2$, which we may identify with $\R^2$ via the standard Euclidean metric.    The dual complex $B^\dual$ in the middle and right parts of Figure \ref{fig:3p2wdual} is colored in peach to distinguish it from the moment polytope $\Phi(X)$ shown in 
grey on the left. 

\begin{figure}[ht]
    \centering
    \scalebox{.2}{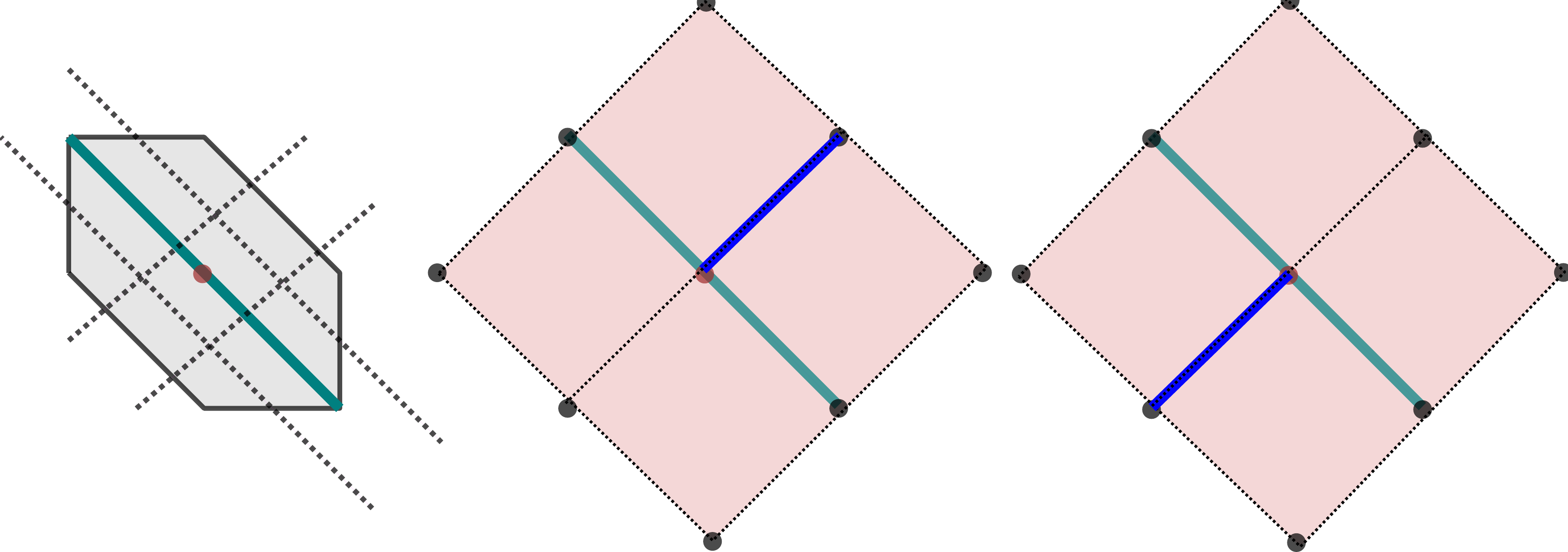}
      \caption{Cartoon diagrams of Maslov-index-two disks in the degree six del Pezzo}
    \label{fig:3p2wdual}
\end{figure}

Indeed, there are two rigid tropical graphs $\Gamma_1,\Gamma_2 \subset \Phi(X)$ representing Maslov-index-two disks shown on the right-most two figures in blue in 
Figure \ref{fig:3p2wdual}. 
\begin{itemize} 
\item The graph $\Gamma_1$ is a single segment connecting $(0,0)$ to $(1/2,1/2)$. 
\item The graph $\Gamma_2$ is a single segment connecting $(0,0)$ to $(-1/2,-1/2)$. 
\end{itemize}
Each of the graphs $\Gamma_1,\Gamma_2$ has two vertices $v_1,v_2 \in \Ver(\Gamma)$, one at the Lagrangian with $m(v_1) = -1$ as in Definition \ref{def:mv} \eqref{mv:cylhit} and one at the boundary $\partial B^\dual$ of the dual complex with $m(v_2) = 1$ as in Definition \ref{def:mv} \eqref{mv:bound}. 

There can be no tropical graphs $\Gamma$ contributing with 
a vertex $v \in \Ver(\Gamma)$ with valence$ |v| \ge  3$.  Indeed, any such graph $\Gamma$ cannot be rigid as the position $\cT(v)$ of the vertex $v$ is deformable.   Alternatively, note using Remark \ref{rem:index} below that any such graph $\Gamma$ must represent a disk $u$ having at least two intersections with the boundary divisor $Y = \Phinv(\partial \Phi(X)) $.  Thus $u$ is a disk of Maslov index at least four. 

Combining the previous two paragraphs we obtain the value of the disk potential as
\[ W_L = m(\Gamma_1) + m(\Gamma_2) = (-1 ) + (-1) = -2 . \]
\end{example}

\begin{example} \label{ex:invim} 
We compute the disk potential for the tropical Lagrangian in the blow-up of the product of projective lines at a point whose graph has a single trivalent vertex.  Let
$X$ be the blow-up of $\P^1 \times \P^1$ with moment polytope the pentagon that is
the convex hull 
\[ \Delta = \on{hull} ( (-1,1), (1,0), (0,1), (1,-1), (-1,1)) . \] 
Let $\Lambda$ denote the tropical graph with a unique trivalent vertex at $\nu_0 = (0,0)$ and univalent vertices at $\nu_1 = (1,0), \nu_2 = (0,1), \nu_3 =  (-1,-1)$, with three edges $\eps_1,\eps_2,\eps_3$ containing a univalent vertex $\nu_1,\nu_2,\nu_3$, respectively.   By, for example, Mikhalkin's realization theorem \cite{mikhalkin:examples} the graph $\Lambda$ is realized by a Lagrangian sphere denoted $L \subset X$.  

We  compute the number of disks passing through a particular point.
 Let $\cP$ be a polyhedral decomposition 
 so that $P_0$ is a polytope of top dimension that contains the vertex of $\Lambda$, $P_i$ is a polytope of top dimension that contains $\Phi(p_i)$ for $i = 1,2,3$ but no point in the boundary $\partial \Phi(X)$.   

\begin{figure}[ht]
    \centering
    \scalebox{.5}{
\begingroup%
  \makeatletter%
  \providecommand\color[2][]{%
    \errmessage{(Inkscape) Color is used for the text in Inkscape, but the package 'color.sty' is not loaded}%
    \renewcommand\color[2][]{}%
  }%
  \providecommand\transparent[1]{%
    \errmessage{(Inkscape) Transparency is used (non-zero) for the text in Inkscape, but the package 'transparent.sty' is not loaded}%
    \renewcommand\transparent[1]{}%
  }%
  \providecommand\rotatebox[2]{#2}%
  \newcommand*\fsize{\dimexpr\f@size pt\relax}%
  \newcommand*\lineheight[1]{\fontsize{\fsize}{#1\fsize}\selectfont}%
  \ifx\svgwidth\undefined%
    \setlength{\unitlength}{565.43529601bp}%
    \ifx\svgscale\undefined%
      \relax%
    \else%
      \setlength{\unitlength}{\unitlength * \real{\svgscale}}%
    \fi%
  \else%
    \setlength{\unitlength}{\svgwidth}%
  \fi%
  \global\let\svgwidth\undefined%
  \global\let\svgscale\undefined%
  \makeatother%
  \begin{picture}(1,0.25909096)%
    \lineheight{1}%
    \setlength\tabcolsep{0pt}%
    \put(0,0){\includegraphics[width=\unitlength,page=1]{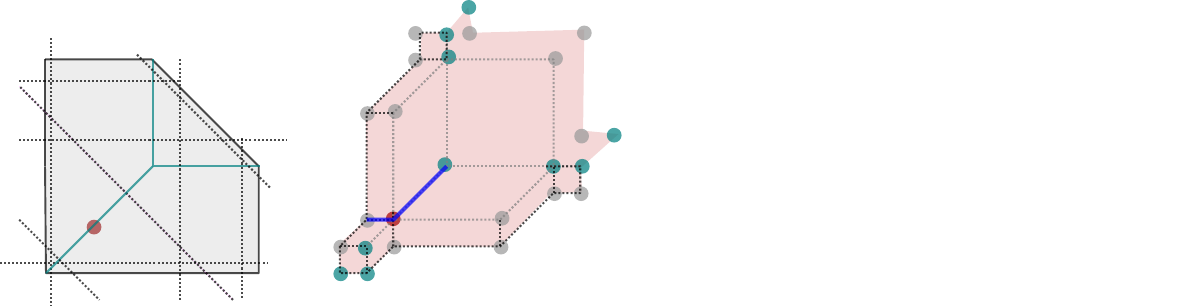}}%
    \put(0.08754422,0.04786274){\color[rgb]{0.87058824,0.52941176,0.52941176}\makebox(0,0)[lt]{\lineheight{1.25}\smash{\begin{tabular}[t]{l}$p_1$\end{tabular}}}}%
    \put(0.82723158,0.1422042){\color[rgb]{0.10196078,0.10196078,0.10196078}\makebox(0,0)[lt]{\lineheight{1.25}\smash{\begin{tabular}[t]{l}$w = (-1/2) + (-1/2) = -1 $\end{tabular}}}}%
    \put(0,0){\includegraphics[width=\unitlength,page=2]{b2p2wdual.pdf}}%
  \end{picture}%
\endgroup%
}
    \caption{Tropical graphs of Maslov-index-two disks in the degree seven del Pezzo}
    \label{fig:b2p2wdual}
\end{figure}

Let $p_3$ be a constraint with moment image $\Phi(p_3) = (0,\eps), \eps > 0 $.  There are two tropical graphs $\Gamma_1,\Gamma_2$
(shown in the upper part of Figure \ref{fig:b2p2wdual}) that have a bivalent vertex at $\Phi(p_3)$, and edges with directions $(1,1)$
and $(0,-1)$ resp. $(-1,0)$.   

We claim that there can be no graphs $\Gamma$ corresponding to Maslov-index-two disks with boundary in $L$  that have  trivalent (or higher) vertices  at $(0,0)$.  Indeed, each outgoing edge of such $\Gamma$ must connect to a leaf $e$ of $\Gamma$ that ends at a vertex $v$ with $P(v)$ meeting the boundary of the moment polytope.  

The Maslov index of such disks is at least four, by Remark \ref{rem:arearem}.     Similarly, for a bivalent vertex at $(0,0)$ the Maslov index of the corresponding disk would be at least four.  Each of the graphs $\Gamma_1,\Gamma_2$ has two vertices, with one trivalent vertex 
with multiplicity $m(v_1) = 1$, and one bivalent vertex with multiplicity 
\[ m(v_2) = (-1)^{ | \det(I)|} |\det(I)|/2 = -1/2 . \]
The total contribution of this graph to the potential is  
\[ m(\Gamma_k) = m(v_2) m(v_1) =  (-1/2)(1) = -1/2 .\]
So the total number of Maslov-index-two disks is 
\[ W_L = m(\Gamma_1) + m(\Gamma_2) =  -1/2 - 1/2 = - 1. \]  
\end{example} 

\begin{example} \label{ex:b5p2}  We compute the potential of a Lagrangian in a del Pezzo of degree four.
Let $X$ be the del Pezzo of degree four equipped with a monotone symplectic form. 
We claim that there exists a Lagrangian sphere $L$ so that the number of Maslov-index-two holomorphic disks with boundary in $L$  and passing through a generic point in $L$ is 
\[ W_L = -4 .\]
The almost toric diagram for the degree four del Pezzo is shown on the left
in Figure \ref{fig:b5p2_sphere}.  The diagram is a slight modification of the diagram $B_1$ in Vianna \cite[Figure 11]{vianna:dp}. 
A tropical graph $\Lambda$ for a Lagrangian sphere $L$ is shown at center, given by Theorem \ref{thm:realize}.
There are four rigid tropical graphs 
\[ \Gamma_1,\Gamma_2,\Gamma_3,\Gamma_4 \subset \Phi(X) \] 
representing Maslov-index-two disks passing through a generic point, shown as cartoon diagrams on the right in Figure 
\ref{fig:b5p2_sphere}.   Each tropical graph 
$\Gamma_1,\Gamma_2,\Gamma_3,\Gamma_4$ has four vertices.  For any $i = 1,2,3,4$, all multiplicities of vertices $v$ of $\Gamma_i$
are equal to $m(v) = 1$ except the vertex at the Lagrangian, for which $m(v)= -1$.  Therefore, 
the contribution of each graph $\Gamma_i$ is $m(\Gamma_i) = -1$, and the potential is 
\[ W_L = m(\Gamma_1) + m(\Gamma_2) + m(\Gamma_3) + m(\Gamma_4)  = -1 + -1 + -1 + -1 = -4 . \]
We are unaware of a simple presentation in the literature for all del Pezzos, generalizing that of Givental described in Sheridan \cite[Appendix]{sheridan:hypersurface}, that would give eigenvalues of quantum multiplication by the first Chern class, 
although presumably a Hori-Vafa-style presentation is known to experts.
\end{example}

\begin{example} \label{ex:degsix} We compute the potential of the Lagrangian 
with trivalent graph in the del Pezzo of degree six shown in Figure 
\ref{fig:b3p2all1}, for two different choices of constraint.  

\begin{figure}[ht]
    \centering
    \scalebox{.4}{
\begingroup%
  \makeatletter%
  \providecommand\color[2][]{%
    \errmessage{(Inkscape) Color is used for the text in Inkscape, but the package 'color.sty' is not loaded}%
    \renewcommand\color[2][]{}%
  }%
  \providecommand\transparent[1]{%
    \errmessage{(Inkscape) Transparency is used (non-zero) for the text in Inkscape, but the package 'transparent.sty' is not loaded}%
    \renewcommand\transparent[1]{}%
  }%
  \providecommand\rotatebox[2]{#2}%
  \newcommand*\fsize{\dimexpr\f@size pt\relax}%
  \newcommand*\lineheight[1]{\fontsize{\fsize}{#1\fsize}\selectfont}%
  \ifx\svgwidth\undefined%
    \setlength{\unitlength}{394.44044495bp}%
    \ifx\svgscale\undefined%
      \relax%
    \else%
      \setlength{\unitlength}{\unitlength * \real{\svgscale}}%
    \fi%
  \else%
    \setlength{\unitlength}{\svgwidth}%
  \fi%
  \global\let\svgwidth\undefined%
  \global\let\svgscale\undefined%
  \makeatother%
  \begin{picture}(1,1.25850643)%
    \lineheight{1}%
    \setlength\tabcolsep{0pt}%
    \put(0,0){\includegraphics[width=\unitlength,page=1]{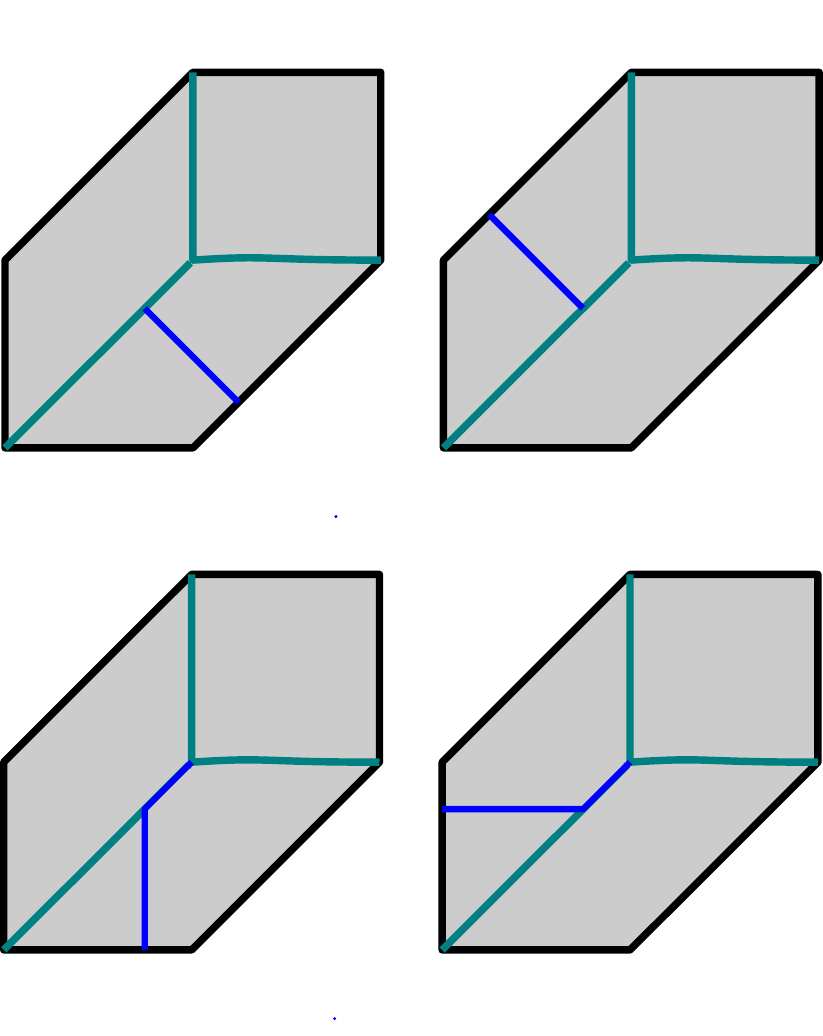}}%
    \put(0.13352924,1.2333917){\color[rgb]{0,0,0}\makebox(0,0)[lt]{\lineheight{1.25}\smash{\begin{tabular}[t]{l}$m(\Gamma) = -1$\end{tabular}}}}%
    \put(0.7271185,1.24289884){\color[rgb]{0,0,0}\makebox(0,0)[lt]{\lineheight{1.25}\smash{\begin{tabular}[t]{l}$m(\Gamma) = -1$\end{tabular}}}}%
    \put(0.05211166,0.00408009){\color[rgb]{0,0,0}\makebox(0,0)[lt]{\lineheight{1.25}\smash{\begin{tabular}[t]{l}$m(\Gamma) = -1$/2\end{tabular}}}}%
    \put(0.64383457,0.00689361){\color[rgb]{0,0,0}\makebox(0,0)[lt]{\lineheight{1.25}\smash{\begin{tabular}[t]{l}$m(\Gamma) = -1$/2\end{tabular}}}}%
    \put(0,0){\includegraphics[width=\unitlength,page=2]{b3p2all1.pdf}}%
  \end{picture}%
\endgroup%
}
    \caption{Cartoon diagrams of Maslov-index-two disks in the degree six del Pezzo}
    \label{fig:b3p2all1}
\end{figure}

Consider first the case that the point constraint is on the lower-left leg with direction $(1,1)$.  In this case, there are two tropical graphs
$\Gamma_1,\Gamma_2$ shown with two univalent  vertices, each with contribution $m(\Gamma_1) = m(\Gamma_2) = -1$, and two tropical graphs $\Gamma_3,\Gamma_4$ shown with two univalent vertices and one bivalent vertex. The direction of the open edge is $(1/2,1/2)$.  The associated determinant is 
$1/2$, and the sign is negative.  Thus the contribution of each is 
\[ m(\Gamma_3) = m(\Gamma_4) = -1/2 \]
for a total potential of $W_L  = -3$. 
\end{example}

\begin{example} \label{ex:trive6} 
We compute the potential of the tropical Lagrangian with a trivalent graph in the Manin $E_6$ system of 
Lagrangian spheres in the cubic surface, shown in Figure \ref{fig:b6p2trop1}. 

\begin{figure}[ht]
    \centering
    \scalebox{.4}{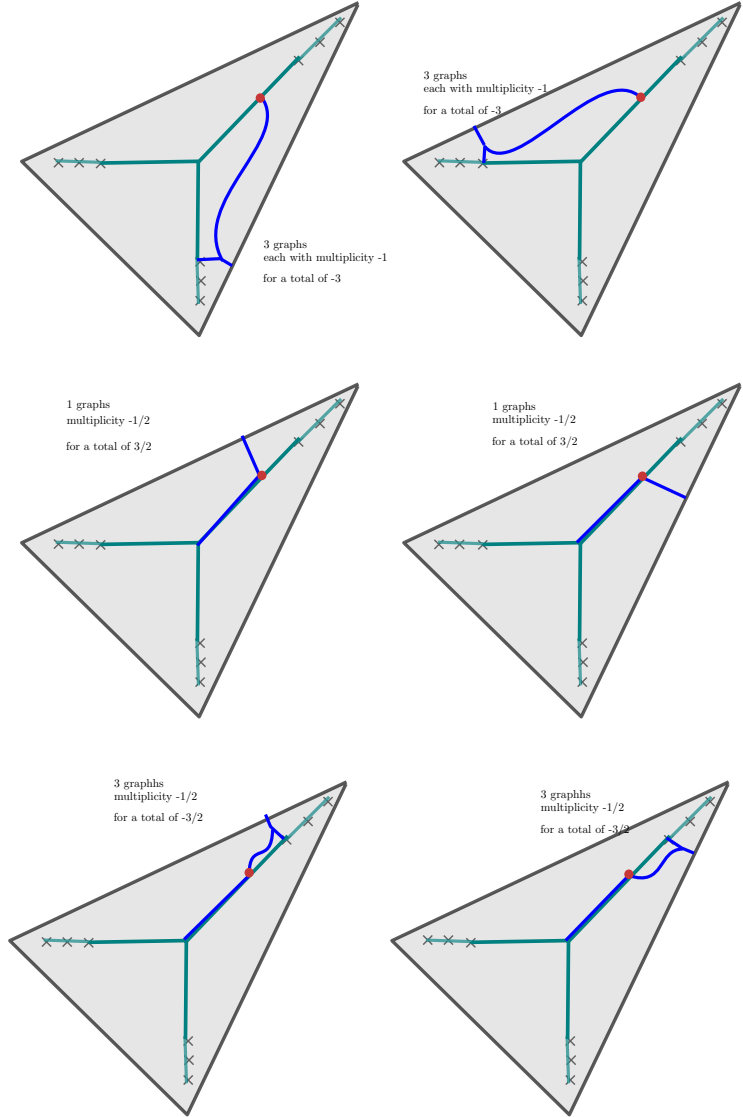}
    \caption{Cartoon diagrams of Maslov-index-two disks in the cubic surface}
    \label{fig:b6p2trop1}
\end{figure}

We assume that the point constraint is on the leg of the Lagrangian with direction $(1/2,1/2)$.   There are two tropical graphs $\Gamma_1,\Gamma_2$ with a single bivalent vertex, and determinant $3$, contributing 
\[ m(\Gamma_1) = m(\Gamma_2) = 3/2 \] 
each according to the sign rule in item 
\eqref{mv:pant} above.    There are six tropical graphs $\Gamma_3,\ldots, \Gamma_8$ with a single trivalent vertex with all edges of closed type, and a negative sign arising from the vertex at the Lagrangian, contributing %
\[ m(\Gamma_3) + \ldots m(\Gamma_8) = -6 \]
in total.  Finally, there are six tropical graphs $\Gamma_9,\ldots, \Gamma_{14}$ with a single trivalent vertex representing a holomorphic pant, each with determinant $1/2$ and negative sign, with a total contribution of 
\[ m(\Gamma_9) + \ldots \ldots m(\Gamma_{14}) = -3 . \] 
Thus the total potential is $W_L = -6$.
\end{example}

We denote references to our earlier works \cite{vw:trop} with the prefix T-and the work \cite{vw:at} with the prefix P-.

\section{Tropical Lagrangians in almost toric manifolds}
\label{troplag}

In this section is to extend the construction of 
Lagrangians associated to tropical graphs in 
Hicks \cite{hicks:realizability} and Mikhalkin
\cite{mikhalkin:examples} to the case of almost toric manifolds, 
and describe the Manin collections associated to simply-laced systems in monotone del Pezzo surfaces. 
The tropical graphs corresponding to the Lagrangians are required to have certain directions at the focus-focus singularities.  

\subsection{Almost toric manifolds}

An almost toric manifold is equipped with a singular Lagrangian fibration so that the fibers are either tori or nodal tori:

\begin{definition}  \label{locmod} An almost toric structure on a symplectic manifold  $(X,\omega)$ is a smooth map $\Phi: X \to B$ to a manifold $B$ so that the map $\Phi$ is given locally as follows:  For any $x \in X$, there exist Darboux coordinates $(q_1,\ldots, q_n,p_1,\ldots,p_n)$ on an open neighborhood $U$ of $x$ so that the restriction of $\Phi$ to $U$ is a product $\Phi_1 \times \ldots \times \Phi_k$ of maps $\Phi_j, j = 1,\dots, k$, each of which is of the following form:
\begin{enumerate}
\item {\rm (Linear)} $\Phi_j(q_j,p_j) = p_j$;
\item {\rm (Elliptic)} $\Phi_j(q_j,p_j) = \hh (q_j^2 + p_j^2) $;
\item {\rm (Focus-focus)}  $\Phi_j(q_j,p_j, q_{j+1},p_{j+1}) = (q_j p_j  + q_{j+1} p_{j+1}, q_j p_{j+1} - q_{j+1} p_j )$.
\end{enumerate}
\end{definition}

\begin{figure}[ht]\begin{center} 
\scalebox{.7}{\includegraphics{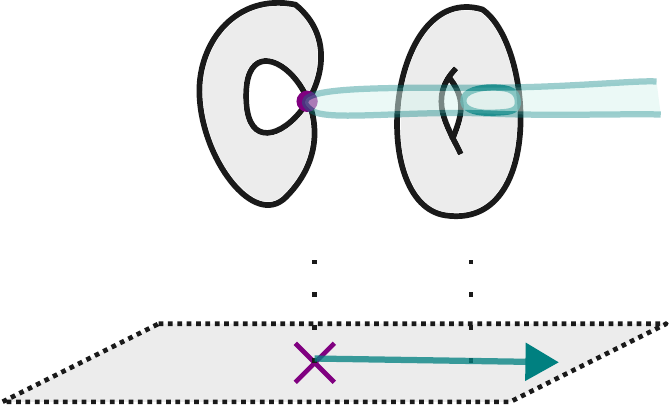}}
\end{center} 
\caption{A focus-focus singularity and one of the Lagrangian vanishing thimbles} 
\label{ffsing}
\end{figure}

We are mostly concerned with the case of almost toric four-manifolds. Near any focus-focus singularity there is a pair of holomorphic disks, given as the unstable and stable manifolds for the first componen of $\Phi_j$, as well as as a pair of Lagrangian {\em vanishing thimbles}, the anti-diagonals in the local model, whose images under the projection are shown in Figure \ref{rotate}.
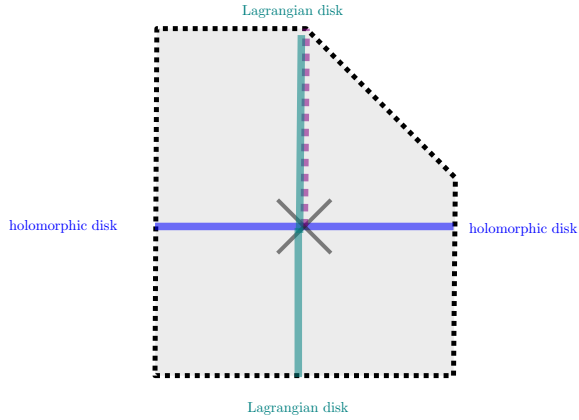
\begin{figure}[ht]\begin{center} \scalebox{.5}{
\begingroup%
  \makeatletter%
  \providecommand\color[2][]{%
    \errmessage{(Inkscape) Color is used for the text in Inkscape, but the package 'color.sty' is not loaded}%
    \renewcommand\color[2][]{}%
  }%
  \providecommand\transparent[1]{%
    \errmessage{(Inkscape) Transparency is used (non-zero) for the text in Inkscape, but the package 'transparent.sty' is not loaded}%
    \renewcommand\transparent[1]{}%
  }%
  \providecommand\rotatebox[2]{#2}%
  \newcommand*\fsize{\dimexpr\f@size pt\relax}%
  \newcommand*\lineheight[1]{\fontsize{\fsize}{#1\fsize}\selectfont}%
  \ifx\svgwidth\undefined%
    \setlength{\unitlength}{448.02073081bp}%
    \ifx\svgscale\undefined%
      \relax%
    \else%
      \setlength{\unitlength}{\unitlength * \real{\svgscale}}%
    \fi%
  \else%
    \setlength{\unitlength}{\svgwidth}%
  \fi%
  \global\let\svgwidth\undefined%
  \global\let\svgscale\undefined%
  \makeatother%
  \begin{picture}(1,0.69424025)%
    \lineheight{1}%
    \setlength\tabcolsep{0pt}%
    \put(0,0){\includegraphics[width=\unitlength,page=1]{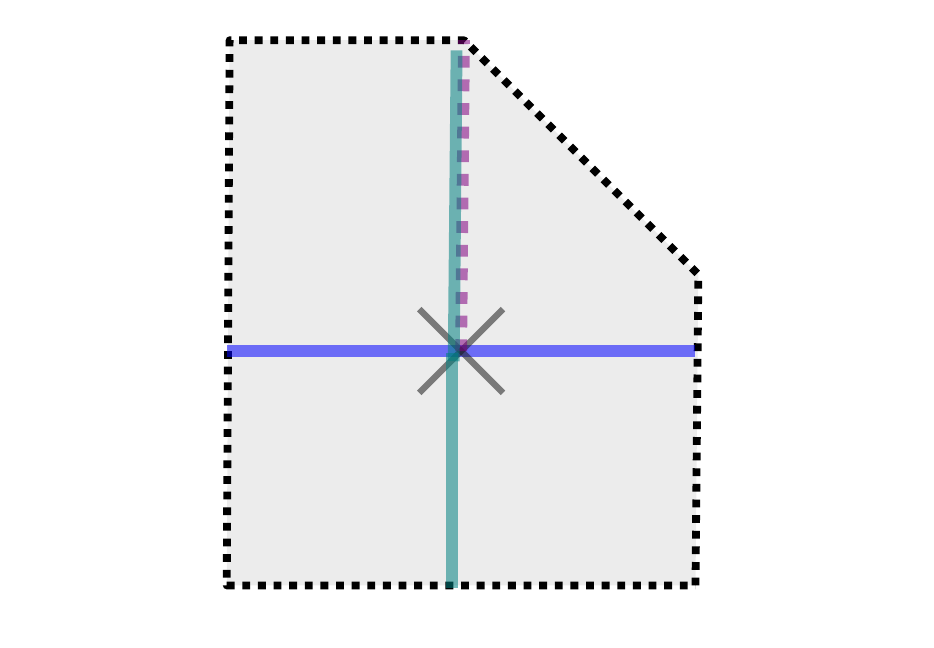}}%
    \put(0.77141664,0.30704588){\color[rgb]{0,0,1}\makebox(0,0)[lt]{\lineheight{1.25}\smash{\begin{tabular}[t]{l}holomorphic disk\end{tabular}}}}%
    \put(-0.0024195,0.3120511){\color[rgb]{0,0,1}\makebox(0,0)[lt]{\lineheight{1.25}\smash{\begin{tabular}[t]{l}holomorphic disk\end{tabular}}}}%
    \put(0.38949634,0.67389034){\color[rgb]{0,0.50196078,0.50196078}\makebox(0,0)[lt]{\lineheight{1.25}\smash{\begin{tabular}[t]{l}Lagrangian disk\end{tabular}}}}%
    \put(0.39846543,0.00553213){\color[rgb]{0,0.50196078,0.50196078}\makebox(0,0)[lt]{\lineheight{1.25}\smash{\begin{tabular}[t]{l}Lagrangian disk\end{tabular}}}}%
  \end{picture}%
\endgroup%
}
\end{center} \caption{Holomorphic and Lagrangian disks near a focus-focus singularity}  \label{rotate}
\end{figure}
The holomorphic and Lagrangian disks are related by a hyper-K\"ahler rotation.  Recall \cite{hitchin:hk} that a hyper-K\"ahler manifold is a datum $(X,g,I,J,K)$ consisting of a smooth manifold $X$, necessarily of some dimension $\dim(X)$ that is a multiple of $4$, and compatible K\"ahler structures 
\[ I,J,K: TX \to TX \] 
satisfying the relations 
\[ I^2 = J^2 = K^2 = - \on{Id}, \quad IJ = K, JK = I, KI = J .\]
 The two-forms 
\[ \omega_I = g( \cdot, I \cdot), \omega_J = g(\cdot, J \cdot), 
\omega_K = g(\cdot, K \cdot) \in \Omega^2(X) \] 
are symplectic, and the hyper-K\"ahler rotation 
\[ (X,g,I,J,K) \mapsto (X,g,J,I,-K) \] 
interchanges special Lagrangians with respect to the form $\omega_I$
with holomorphic submanifolds with respect to $I$ as in Strominger-Yau-Zaslow \cite{strominger}.  Explicitly, let
\[ X = \C^2 = \{ (q_1 + i p_1, q_2 + i p_2 ) \}  \] 
and let $I,J$ be the complex structures defined by 
\[ I(x_1,x_2) = (ix_1,i x_2), \quad J(x_1,x_2) = (-i \ol{x_2},ix_1) .\]
The almost complex manifold $(X,J)$ may be identified with $(X,I)$ by the map 
\begin{equation} \label{eq:hk} \phi: \C^2 \mapsto \C^2, \quad ( q_1 + ip_1, q_2+ i p_2) \mapsto ( q_1 - ip_2, q_2 + ip_1)  .\end{equation} 
Since $\omega(D\phi v,D \phi Jv) = 0$, 
for any vector $v \in TX $, the rotation maps any embedded holomorphic curve to a Lagrangian  submanifold.  

We now turn to the global picture for almost toric four-manifolds.    Via the Arnold-Liouville theorem as explained in Duistermaat \cite{duist:global}, the complement $B - B^{\on{foc}}$ of the focus-focus values has an affine
 structure so that $\Phi$ is locally the moment map for a Hamiltonian torus action over
 $B - B^{\on{foc}}$.     The image $\Phi(X)$ in $B$ is the {\em moment polytope} of $X$.
 The set $\Phi(X)$ is locally polyhedral on $B - B^{\on{foc}}$
 with respect to the given affine structure.   The monodromy of a small loop around each  critical value $b \in B^{\on{foc}}$ is an automorphism of the fiber given by 
\[ \alpha_b \in \Aut(\Phi^{-1}(b)) \cong \Aut(TB_\Z), \quad \alpha_b (\nu) = \nu + \lan \Sigma_b, \nu \ran 
\Pi_b
\]
for some vectors 
\begin{equation} \label{eq:dirshear}
\Sigma_b \in T^\dual_b B , \quad \Pi_b \in T_b B \end{equation}
defining the hyperplane and direction of the shear, respectively.
We choose a collection of hyperplanes 
\[ \ul{C} = (C_i \subset B, i \in I) \] 
called {\em branch cuts} (these are rays in the two-dimensional case)  associated to the components of $B^{\on{foc}}$.  Each branch cut $C_i$
connects a focus-focus value $b \in B^{\foc}$ to the boundary $\partial \Phi(X)$ of the moment polytope.   
The complement $B - B^{\on{foc}} - \cup_i C_i  $ has trivial monodromy and so, if simply connected, may be embedded in a vector space.  
%
%
%
The codomain $B$ equipped with the moment polytope $\Phi(X)$, the  focus-focus values $B^{\on{foc}}$ and branch cuts $\ul{C}$ is the {\em almost toric base diagram} of the almost toric manifold $X$. 

By a {\em del Pezzo surface} we mean a projective Fano surface.  As del Pezzo showed (see Kollar-Smith-Corti \cite{kollarsmithcorti}) any such surface is either a quadric surface or obtained by blowing up at most eight points in general position on the projective plane. By a result of McDuff \cite{mcduff:blowup}, the monotone symplectic form in the first Chern class $c_1(X)$ is unique up to symplectomorphism.  The {\em degree} of a del Pezzo surface $X = \Bl^k \P^2$ is the self-intersection number of the canonical class
\[ c_1(\Bl^k \P^2)^2 = c_1(\P^2)^2 - k  = 9 - k .\]
In the case of  monotone del Pezzo's, we may take $\Phi(X)$ to be a convex polytope in $\R^2$ so that each shear has the direction of a ray connecting the barycenter with a vertex, as in Vianna \cite{vianna:dp}.    Some of the almost toric diagrams for 
del Pezzo surfaces studied in \cite{vianna:dp} are shown in Figure \ref{triang2}.

\subsection{Realizability for tropical Lagrangians}

We introduce tropical Lagrangians in 
almost toric manifolds, following Hicks \cite[Section 4]{hicks:dp} \cite{hicks:realizability} \cite[Section 7]{hicks:thesis} and Mikhalkin \cite{mikhalkin:examples}
in the toric case;  note that Matessi \cite{matessi:trop} has generalized the constructions to tropical hypersurfaces but here we are restricting to tropical Lagrangians corresponding to graphs. 

\begin{definition}   Let $n > 0 $ be an integer.  An {\em affine manifold} is a smooth $n$-manifold $B$ equipped with a torsion-free flat connection on $TB$, so that the transition maps with respect to flat coordinates are constant maps with values in $\R^n \ltimes GL(n,\R)$.   An {\em integral affine manifold} is an affine manifold with the additional structure of a covariant-constant family of lattices $T_{b,\Z} B \subset T_b B, b \in B$ so that the transition maps may be taken to lie in $\R^n \ltimes GL(n,\Z)$.  An {\em integral affine two-manifold with singularities} is a topological two-manifold 
$B$, a finite set $B_0 \subset B$ and an integral affine structure on the complement $B - B_0$. 
\end{definition}

We have in mind the case that $B_0$ is the set of focus-focus  values for an almost toric moment map.

\begin{definition}  \label{def:tropgraph} Let $B$ be an integral affine two-manifold with singularities.  A {\em tropical graph} in $B$ is a graph 
\[ \Gamma = (\Ver(\Gamma),\Edge(\Gamma)) \] 
equipped with 
\begin{enumerate}
\item $\iota: \Ver(\Gamma) \to B$ and
\item for each $e \in \Edge_\black(\Gamma)$ joining vertices, a geodesic 
\[ \gamma_e: [0,\ell(e)] \to B \] 
joining $\iota(v_-)$ and $\iota(v_+)$ given by a path $\gamma_e$ with some covariant constant direction 
\[ \cT(e) := \ddt \gamma_e(t) \in T_{\gamma(t),\Z} B \]
in the integral lattice $T_{\gamma(t),\Z} B, t \in (0,\ell(e))$.
\end{enumerate}
\end{definition}

\begin{definition}  A tropical graph $\Gamma$ is {\em balanced} 
if for any vertex $\nu$ with at least two adjacent
edges $\eps \ni v$  the {\em balancing condition}
\[ \sum_{\eps \ni \nu} \cT(\eps) = 0 .\]
holds.
\end{definition}

We associate to certain tropical graphs Lagrangian submanifolds following Mikhalkin \cite{mikhalkin:examples}.  Let $\Lambda$ be a tropical graph in the moment polytope $\Phi(X) \subset B$ of an almost toric diagram with image of the focus-focus singularities denoted $B^{\on{foc}}$.   For a face $Q$ of the polytope $\Phi(X)$, denote by $T^\dual_{Q_i} B$ the cotangent space of $B$ at any point $q \in Q$; these are identified via the natural affine structure.
Therefore, the tangent spaces to $Q_i$ are canonically isomorphic.  Our first examples will be embedded graphs; however in the sequel \cite{wo:man} we consider immersed graphs $\Lambda$.

\begin{definition} A direction $\cT(\eps)$ of an
edge $\eps \in \Edge(\Lambda)$ at a codimension 
two face $Q$ of the moment polytope $\Phi(X)$ contained in facets $Q_i,Q_j \subset \Phi(X)$ with primitive normal vectors 
\[ \lambda_i  \in T^\dual_{q_i} B, \ q_i \in Q_i ,\quad \lambda_j \in T_{q_j}^\dual B, q_j \in Q_j \] 
is a {\em bisectrice} if the pairings 
\[ \lan \cT(\eps),\lambda_i \ran = \lan \cT(\eps),\lambda_j \ran = 1 \] 
of the direction $\cT(\eps)$ with the normal vectors
to the corresponding facets $\nu_i,\nu_j$ are equal to 1.
\end{definition}

\begin{definition} \label{def:compat} 
A tropical graph $\Lambda \subset \Phi(X)$ is {\em allowable} if 
\begin{enumerate}
\item all univalent vertices are at the boundary  $\partial \Phi(X)$ or at 
the images $B^{\on{foc}}$ of focus-focus singularities;
\item all vertices $\nu \in \Ver(\Gamma)$ have valence at most $3$; 
\item if $\nu \in \Ver(\Gamma)$ has valence $3$ then the corresponding Mikhalkin determinant $m(\nu)$ is equal to $1$;
\item the edges $\eps \in \Edge(\Lambda)$ do not contain any focus-focus values
$b \in B^{\foc}$ in their interiors $\eps^\circ \subset \eps$;
\item  if $\nu \in \Ver(\Lambda)$ is univalent then either $\Phi(\nu)$ in a codimension two boundary face $Q$ of $\partial \Phi(X)$ in which case the adjacent directions $\cT(\eps)$ are a {\em bisectrice} as defined above; and
\item if $\nu \in \Ver(\Lambda)$ is univalent and 
equal to a focus-focus value $b \in B^{\foc}$, in which 
case the direction $\pm \cT(\eps)$ of the adjacent edge $\eps \in \Edge(\Lambda)$ is equal to the direction
$\Pi_b \in TB$ of the associated shear from \eqref{eq:dirshear}; 
\footnote{We differ from Mikhalkin \cite{mikhalkin:examples} in that we do not allow univalent vertices on facets of the moment polytope; 
these require non-orientable local models.} 
\end{enumerate}
A tropical graph $\Lambda$ is {\em primitive} if
\begin{enumerate}
    \item each edge direction $\cT(\eps), \eps \in \Edge(\Lambda)$ is a primitive lattice vector in $T_\Z B$, 
    \item there are no interior vertices $v \in \Ver(\Lambda)$ of valence $|v|$ more than three, and 
    \item each boundary vertex $\nu \in \Ver(\Lambda)$ has valence 
one if  $\nu$ lies in a codimension one or two face $Q$ of $\Phi(X)$.   
\end{enumerate}
\end{definition}

\begin{figure}[ht]\begin{center} \scalebox{.5}{\includegraphics{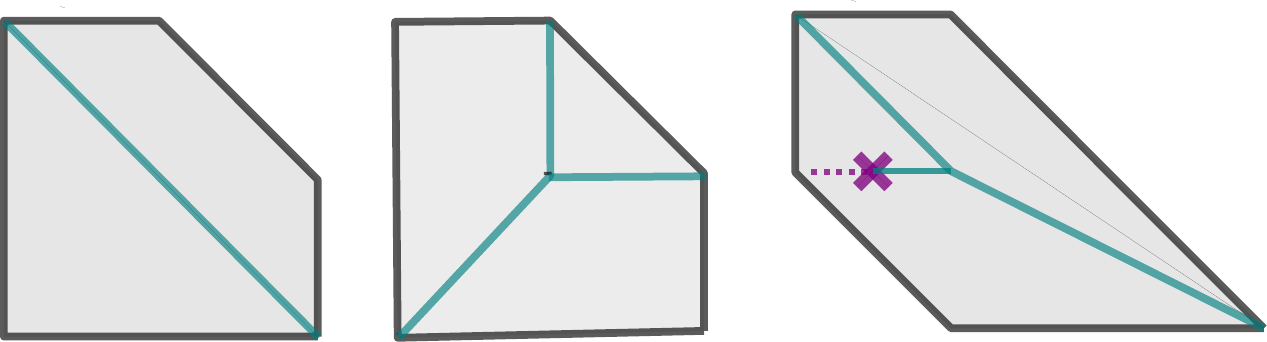}}
\end{center} \caption{Tropical graphs of tropical Lagrangian spheres in $\Bl^2 \P^2$}  \label{fig:prehind}
\end{figure}

\begin{definition} For any tropical graph $\Lambda$ with a trivalent vertex $\nu \in \Ver(\Lambda)$, define the \em{local self-intersection number}
\[ \delta(\nu) = \hh (m(\nu) - 1) \]
where $m(\nu)$ is the determinant from Definition 
\ref{def:mv} \eqref{mv:pants}.
\end{definition}

Choose a metric on $B$ and for $\delta > 0$ let 
\[ U_\delta(\Ver(\Lambda)) \subset B \] 
denote an $\delta$-neighborhood around the vertices $\Ver(\Lambda)$.

\begin{definition} \label{def:realize}
A tropical graph $\Lambda \subset \Phi(X)$ is {\em Lagrangian realizable} if there
exists a family of proper Lagrangian immersions, called {\em Lagrangian realizations}
\[ \iota_\rho:  L \to X,  \ \rho \in (0,\rho_0) \] 
smoothly dependent $\rho$ with the
following properties:   For any $\delta > 0 $ there exists $\kappa_0$ so that  for $\kappa < \kappa_0$ 
\begin{enumerate}
\item $\Phi(\iota_\kappa(L))$ is contained in the union of $\Lambda$
with $U_\delta(\Ver(\Lambda))$;
\item for each $\mu$ in $\Lambda - U_\delta(\Ver(\Lambda))$, 
the intersection of $L$ with $\Phinv(\mu) \cong T^n$ is an affine 
subtorus of the form $a(\mu) A$ for some codimension one torus $A \subset T^n$ and $a(\mu) \in A$; and 
\item for every vertex $v \in \Ver(\Lambda)$,
the inverse image $\Phinv(U_\delta(v))$ is diffeomorphic to the product of a
two-sphere $S^2$ with some number of holes removed 
\[ S^{2} - \cup_{i=1}^k B_{\kappa_i}(v_i) \] 
times a torus $T^{n-2}$, if $v$ does not map to $B^{\on{foc}}$;
or a transverse self-intersection $(\R^2 \cup i \R^2) \times (S^1)^{n-2}$, if $v$
maps to $B^{\on{foc}}$.
\end{enumerate}
We call any realization of a tropical graph a {\em tropical Lagrangian}.
\end{definition}


\begin{theorem} \label{thm:realize}
(following Mikhalkin \cite{mikhalkin:examples} in the toric case, Hicks \cite[Definition 7.1.0.1]{hicks:thesis}) Any allowable embedded primitive tropical graph  $\Lambda$ in the image $\Phi(X)$ of an almost toric fibration $X$ is tropically realizable by a Lagrangian $L$ whose genus is equal to the first Betti number of $\Lambda$, and whose self-intersection number is the sum of the local self-intersections
\[ [L].[L] = \sum_{v \in \Ver(\Lambda)} \delta(v) .\]
\end{theorem}

 Mikhalkin \cite{mikhalkin:examples} proves the special case of Theorem \ref{thm:realize} for toric varieties.   The proof in both this case and in the toric case relies on \cite[Lemma 4.4]{mikhalkin:examples} which describes how to patch together the local models for the Lagrangian near the vertices:  Let $\C^\times := \C - \{ 0 \}.$

\begin{lemma} \label{lem:44} Let $\Phi:(\C^\times)^n \to \R^n$ be the standard moment map arising from the identification
$(\C^\times)^n \cong (T^*S^1)^n$.  Let $\eps_j \subset \R^n, j = 1, 2$ be two disjoint open intervals parallel to the same integer vector $\cT(\eps_j) \in  \Z^n$ and $L_j \subset (\C^\times)^n$  be two smooth connected Lagrangian submanifolds such that $\Phi(L_j ) \subset E_j$ , and $L_j$ is relatively closed in $\Phinv(E_j)$. There exists a Lagrangian submanifold $L \subset (\C^\times)^n$  diffeomorphic to $\R \times (S^1 )^{n-1}$
such that $L_1$ and $L_2$ are its submanifolds if and only if $E_1$ and $E_2$ belong to the same line $E \subset \R^n$.  Furthermore, if $E_1, E_2 \subset E $ then the Lagrangian submanifold $L$ can be chosen so that $\Phi(L) \subset E$.
\end{lemma}

 \begin{example}  \label{ex:locmod} The Lemma is proved using various local models. 
 \begin{enumerate}
     \item {\rm (Lagrangian pair of pants)}  \label{item:lpants}
 The model for trivalent vertices
is the Lagrangian pair of pants from \cite[Lemma 4.2, 4.3]{mikhalkin:examples}.    In coordinates $z_j = {q_j + i p_j} \in \C^\times$, $j = 1,2$  the {\em holomorphic pair of pants} given by the hyperplane  
\[ H = \{ z_1 + z_2 - 1 = 0 \}  \subset (\C^\times)^2 \] 
has hyper-K\"ahler rotation in the coordinates $q_1,p_1,q_2,p_2$ (see \eqref{eq:hk}) the {\em Lagrangian pair of pants} 
\[ L = \{ 1 - (q_1- ip_2) + (q_2 + ip_1) = 0  \}  \subset \C^{\times,2} .\]
See Hicks \cite[Section 3]{hicks:realizability}, Matessi \cite{matessi:pants}, and also the definition of the {\em surgery trace} considered in Biran-Cornea \cite{biran:lcob}.   The Lagrangian $L$ has three cylindrical ends, corresponding to the intersections with $z_1  = 0, z_2 = 0$, and $z_1 = z_2 = \infty$.
  \item {\rm (Diagonals)}   The diagonal embedding of $L = \R^2$ in $X = \R^4$
  is a Lagrangian with a single cylindrical end, corresponding to a univalent vertex with a bisectrice edge.
  \item {\rm (Vanishing thimbles)}   Consider the diagonal 
\[ \{ p_1 = q_1, p_2 = q_2 \} \subset \R^4 .
\]
The image under the projection to the base
\[ \{  \Phi(q_1,p_1, q_{2},p_{2}) = (q_1 p_1  + q_{2} p_{2}, q_1 p_2 - q_2 p_1 )  \} \subset \R^2 \] 
is a ray emanating out of the focus-focus value in the shear direction $(1,0)$.  In the case of univalent or bivalent
vertices mapping to the singular set, one can take the Lagrangian to be the vanishing thimble 
$\R^2 \times (S^1)^{n-2}$ resp.  $(\R^2 \cup i \R^2) \times (S^1)^{n-2}$ in the local model 
for the focus-focus singularities.      
 \end{enumerate}
\end{example}

\begin{proof}[Proof of Theorem \ref{thm:realize}]   Given the local models the statement of the first part of the Theorem now follows from Lemma \ref{lem:44}.
The proof of the self-intersection formula is the same as in Mikhalkin \cite{mikhalkin}.
\end{proof}

\begin{remark}  By the construction in the proof of 
Lemma \ref{lem:44}, the Lagrangian $L$ in Theorem \ref{thm:realize} realizing a graph $\Lambda$ is unique up to isotopy. If $L$ is a sphere then $L$ is unique up to Hamiltonian isotopy since the first homology vanishes. 
\end{remark}

\begin{remark}
There are other tropical graphs $\Gamma$ besides those described in Theorem \ref{thm:realize} that are  realizable in a similar sense:  as pointed out to us by Annie Wei, the balancing condition $\pm \cT(e_1) \pm \cT(e_2) \pm \cT(e_3) = 0 $  for any triple of edges $e_1,e_2,e_3$ meeting at a trivalent vertex $v$ of $\Gamma$ is sufficient.  See also the work of Hicks \cite{hicks:realizability}.
\end{remark}

The following seems to be a folklore theorem known among experts:

\begin{proposition} \label{prop:contains} Each monotone del Pezzo surface $X$ contains a Manin configuration of tropical Lagrangian spheres $L_\alpha, \alpha \in S$ of the corresponding simple root system $S$.
\end{proposition}

\begin{proof}  The configurations may be read off  from the almost toric diagrams constructed in Vianna \cite{vianna:dp}, 
as in Figure \ref{dpall2}.   For example, in the case of $X = \Bl^3(\P^2)$
\begin{itemize} 
\item the graphs $\Gamma_1,\Gamma_2,\Gamma_3$ with no trivalent vertices, each consisting of a single segment, correspond to Lagrangian spheres $L_1,L_2,L_3$ with intersection numbers 
\[ L_i \cdot L_{i+1} = 1 , \quad \forall i \in \Z/3\Z \]
and give an affine $A_2$ root system, while;
\item a graph $\Gamma_0$ with a single trivalent vertex gives an $A_1$ root system.  
\end{itemize}
In the case of $X = \Bl^4(\P^2)$, the obvious graphs $\Gamma_1,\ldots, \Gamma_4$ connecting the focus-focus values and two of the vertices give a collection of Lagrangian sphere $L_1,\ldots, L_4$ forming $A_4$ root system; adding the surgery $L_5$ at each intersection point would give Lagrangians
$L_1,\ldots,L_5$ forming an affine $A_4$ root system.  In the remaining cases, the ``obvious'' tropical Lagrangians give an affine root system of the required type.  
\end{proof}

\begin{figure}[ht]\begin{center} 
\scalebox{.4}{\includegraphics{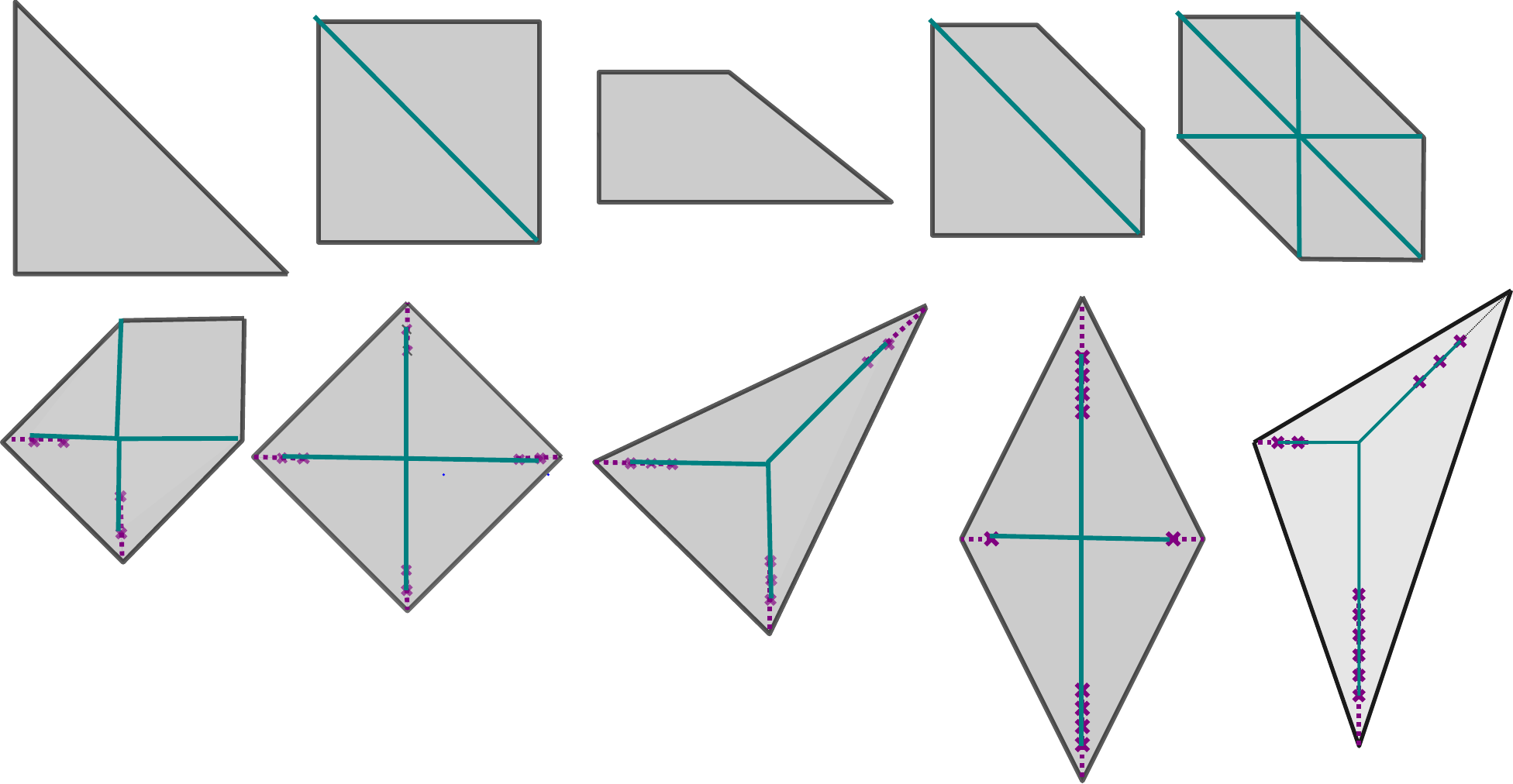}}
\end{center} 
\caption{Graphs of Manin configurations of tropical Lagrangians in del Pezzo surfaces}
\label{dpall2}
\end{figure}

\begin{example}  \label{ex:hind}
Let $X$ be the blow-up of $\P^1 \times \P^1$ at a torus-fixed point corresponding to a vertex of the moment polytope $\Phi(X) = [0,1]^2$.   There
is a Lagrangian sphere $L \subset X$ whose tropical graph $\Lambda \subset \R^2$
is the trivalent graph shown on the right in Figure \ref{fig:prehind}.   In fact,
the Mikhalkin multiplicity at the unique vertex is equal to one, and so the self-intersection computed in Mikhalkin \cite{mikhalkin:examples} is equal to zero. 
\end{example}




\subsection{Isotopies of tropical Lagrangians}
\label{mutatesec}

For the purposes of counting disks in the next section, we wish to study
Hamiltonian isotopies of Lagrangians associated to tropical graphs that change the vertices in the tropical graph:

\begin{theorem} \label{thm:mut} Suppose an almost toric base diagram $B$ is modified by a nodal slide in which a critical value $b \in B^{\on{foc}}$ passes through an edge $\eps$ of the graph $\Lambda$ of a tropical Lagrangian $L$ in $B$, resulting in a new tropical base diagram $B'$.  Then there exists a tropical graph $\Lambda'$ in $B'$ realized by a Lagrangian $L'$  isotopic to $L$, with $\Lambda'$ obtained by adding an edge in the direction of the shear at $b$.\end{theorem}

\begin{proof}   We describe how to obtain the new Lagrangian by a handle attachment.  Let $x \in X$ be the focus-focus singularity corresponding to $b$, and let
 $\phi:U \to \R^2$ be a coordinate chart containing $b$.  By assumption, the projection to the base 
$\Phi |_U  $ in the coordinates $\phi$ is given by a pair of functions $  (\Phi_1, \Phi_2)$ so that  the second component $\Phi_2$ is Morse. 
 Let $ \{ \kappa \} \times I \subset U$ be disjoint from the critical value.  By assumption, $L$ is a realization of $\Lambda$, so we may assume $L$ 
is a circle bundle over an interval $ \{ \kappa \} \times I$.  Let $L'$ denote the tropical Lagrangian corresponding to $\{ - \kappa \} \times I$.  The function $\Phi_2$ is Morse and the focus-focus point represents a singularity of $\Phi_2$ of index $2$; as in Milnor \cite{milnor:morse}, gradient flow of $\Phi_2$ defines a diffeomorphism of the level sets of $\Phi_2$ away from the unstable and unstable manifolds.  
Since $L$ intersects the stable manifold in a single point, $L'$ is isotopic to the manifold obtained from $L$ by adding a handle, that is, removing a small disk from $L$ and gluing in an isotropic disk (the vanishing thimble) passing through the critical point.

In order to construct an isotopy of Lagrangians, we construct an isotopy of holomorphic curves and then perform a hyper-K\"ahler rotation.  Let  $X$ be the blow-up $\P^1 \times \P^1$ at $(z,0)$ for some point $z \neq 0, \infty$.    We view $X$ as a Lefschetz fibration over $\P^1$ with a single nodal fiber.  The union $Y$ of
the proper transforms of the divisors 
\[ P^1 \times \{0 \}, \quad \P^1 \times \{ \infty \}, \quad \{ 0 \} \times \P^1, \quad \{ \infty \} \times \{ 0 \}  \]
is naturally identified  $\C^2 - \{ z_1 z_2 = 1 \}$.   Indeed, we may identify $X$ the the blow-up of $\P^1 \times \P^1$
at $(0,0)$.  Removing three divisors gives the torus-invariant chart $\C^2$ containing the nodal point, and then $X - Y$
is identified with the complement of a single fiber of the map $\C^2 \to \C, (z_1,z_2) \mapsto z_1 z_2$. 
In particular, 
\[ \XX := X - Y \] 
is a cylindrical-end manifold with a hyper-K\"ahler structure.  After hyper-K\"ahler rotation we obtain a singular Lagrangian fibration with a single nodal fiber with a single node, giving a local model for a focus-focus singularity.

\begin{figure}[ht]\begin{center} 
\scalebox{.4}{\includegraphics{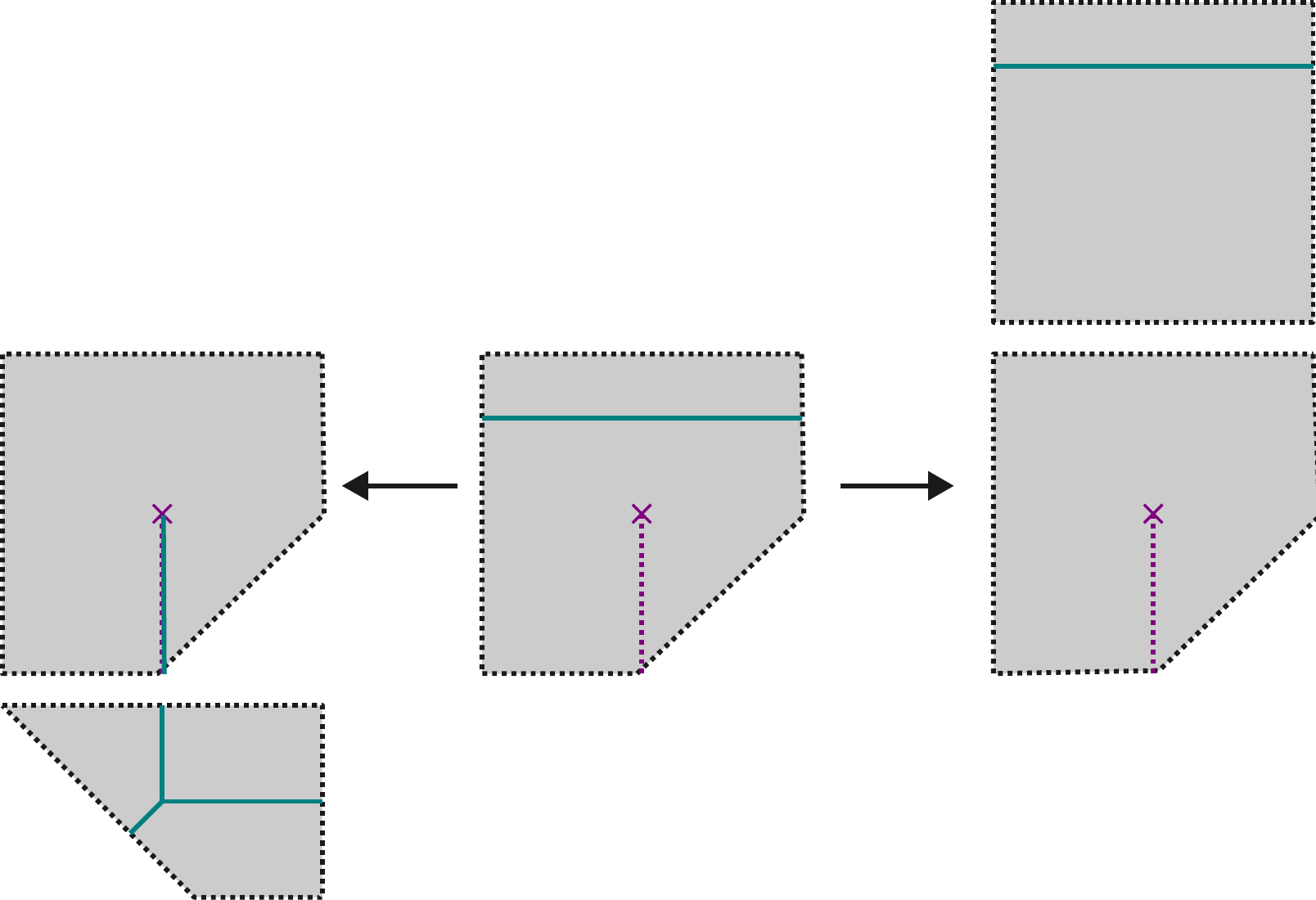}}
\end{center} 
\caption{Behavior of tropical Lagrangians under nodal slides}
\label{fig:mutate}
\end{figure}

We produce the needed family of Lagrangians as a rotation of a family of holomorphic curves.  Let 
\[ u_t: \P^1 \to \XX, \quad z \mapsto (z,t) \] 
be the family of sections taking non-zero values.   The Gromov limit $u_0$ of the family of stable maps $u_t$
is described as follows, and shown in Figure \ref{fig:mutate}.    The image of $u_0$ is the union of the proper transform of the zero section with the exceptional divisor $E$.  Indeed,
 the homology class $[u_0]$ of $u_0$ must equal that of $[u_t]$ by preservation of homology class and the limit must project to $\P^1 \times \{ 0 \}$
in $\P^1 \times \P^1$.  

We now compute the limit in the sense of symplectic field theory, treating the target
as a manifold with cylindrical end.  The symplectic field theory limit of $u_t$ has two levels $(u_+,u_-)$ described as follows.  
\begin{itemize} 
\item The map $u_+$ is the (unique up to isomorphism) relative stable map in the total space of 
the degree-$-1$ line bundle $\mO(-1) \to \P^1$, so that $u_+$ has graph the trivalent graph 
with directions $(-1,-1), (0,1), (1,0)$, and has cylindrical puncture at the fiber 
corresponding to the exceptional curve. 
\item The map  $u_-$ is the map to $\XX$ whose 
image is contained in exceptional curve $E$. 
\end{itemize} 
Indeed, the broken map $(u_+,u_-)$ must map to the stable map $u_0$ of the previous paragraph, under the canonical functor from broken maps to stable maps.  This functor assigns to any neck piece the composition  with the projection $\mO(-1) \to \P^1$,  and assigns to any finite energy map to the complement of the zero section, the map to $X$ obtained by removal of singularities.  In particular, the sft limit of $u_t$ must have at least two components. 
one of which is the map $u_-$.   The map $u_+$ must satisfy a matching condition with $u_-$,
and so meet infinity at the the fiber corresponding the exceptional curve.  By conservation of homology class 
$u_+$ must meet the fibers at $0$ and $\infty$, and have zero intersection with the zero section.   It follows that $u_-$ is the section with a single pole, as claimed.

By hyper-K\"ahler rotation \eqref{eq:hk} in $\XX$, one obtains a family $L_t$ of Lagrangians limiting as $t \to 0$ (in the Gromov sense, pre-rotation) to a gluing $L_\infty$ (in the sense of gluing of pseudoholomorphic curves) of a curve $L_{\infty,-}$ corresponding to a trivalent graph and a punctured disk $L_{\infty,+}$, and as $t \to \infty$ to a cylinder.    That is, as the parameter $t \to 0$,
the curve $L_\infty$ Gromov-converges to the union of $L_{\infty,-}$ and $L_{\infty,+}$.

To check that the Lagrangians are exact isotopic,  it suffices to show that the areas of homotopy classes of disks are unchanged.  Let $L_t$ be an isotopy from $L$ to $L'$.  It suffices to show that for each isotopy class of disks $u:S \to X$ with boundary in $L$ , there exists an isotopy $u_t$ of $u$ bounding $L_t$ so that the areas $A(u_t)$ are constant in $t$. We may assume, by deforming the boundary $u(\partial S) $
away from the attaching sphere, that the disk boundary $u(\partial S)$ lies in the region $V \subset L$ so that
the intersection $V \cap L_t$ is independent of $t$. Indeed the region $V$ is the complement of a small ball in $L$.  Then the constant disk $u_t = u$ extends over the isotopy. 
\end{proof}

\begin{remark} Applying mutations to an Manin configuration $L_\alpha, \alpha \in S$ for a root system $S$ for some base diagram $B$ produces new Manin configurations $L'_\alpha, \alpha \in S$ whose tropical graphs $\Lambda$ lie in the mutated toric base diagram $B'$.  An example in the case of the del Pezzo of degree one shown in Figure \ref{e8mut}. 
\end{remark}

\begin{figure}[ht]\begin{center} 
\scalebox{.8}{\includegraphics{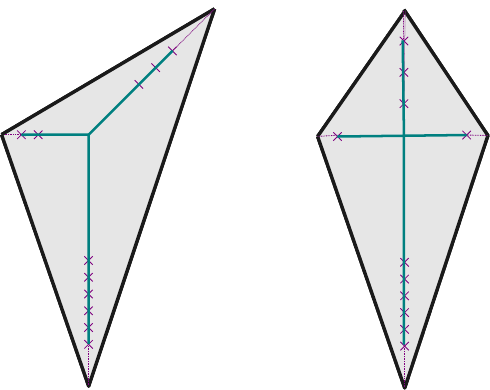}}
\end{center} 
\caption{Mutation of a Manin configuration in the del Pezzo of degree one}
\label{e8mut}
\end{figure}

\begin{proposition} \label{hindprop}  Let $X = \Bl^2 \P^2 = \Bl^1 (\P^1 \times \P^1)$
equipped with a monotone symplectic form.  
There is a unique Hamiltonian isotopy class of embedded Lagrangian spheres in $X$ 
containing the inverse image $L$ of the diagonal in $\P^1 \times \P^1$.  In particular, the Lagrangian $L'$ whose tropical graph $\Lambda$ is the trivalent graph
shown in the last diagram in Figure \ref{hind} (discussed in Example \ref{ex:hind})
is Hamiltonian isotopic to $L$, by a Hamiltonian isotopy $(L_t)_{t \in [0,1]}$ so that every $L_t$ is a tropical realization of some graph in a diagram for $X$, except for a single value of $t$.
\end{proposition}

\begin{example} \label{ex:hind2} The existence of the Hamiltonian isotopy 
is a special case of a generalization of a result of Hind \cite{hind:spheres} due to Evans \cite{evans:dp}: For the del Pezzo surfaces 
$\Bl^k \P^2$ for $k \leq 4$ and any given Lagrangian homology class $[L] \in H_2(X)$, 
the Lagrangian spheres $L$ in $[L]$ are related by Hamiltonian isotopy.  
The isotopy of the Proposition can be constructed explicitly using the remarks in 
Section \ref{mutatesec}.  The second diagram in Figure \ref{hind} is obtained from the first by a nodal trade (see \cite[Section 2]{vw:at}); the third  by the handle attachment in Theorem \ref{thm:mut}; the fourth by a reverse nodal trade; the fifth and sixth by change of basis. 
\end{example}

\begin{figure}[ht]\begin{center} \scalebox{.4}{\includegraphics{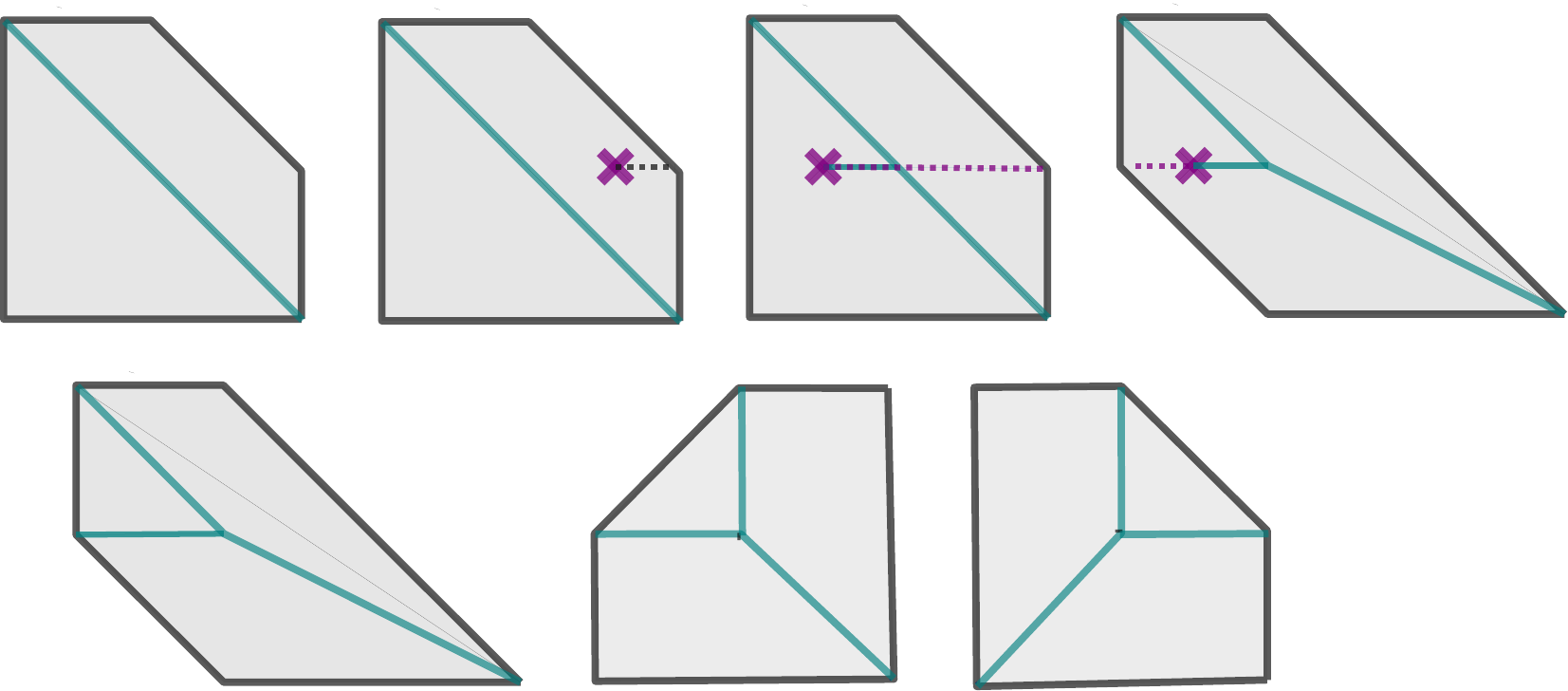}}
\end{center} \caption{An isotopy of tropical Lagrangian spheres in $\Bl^2 \P^2$}  \label{hind}
\end{figure}

\section{Tropical limits of holomorphic disks}

We describe the degeneration techniques from Venugopalan-Woodward \cite{vw:trop}.
The moduli space of disks is cobordant to a moduli space of broken disks, each of which has an associated tropical graph.  The cobordism implies tropical formulas for the disks potential and open-closed map. 

\subsection{Broken Maps} 

Broken maps arise as a limit of neck stretching associated to some polyhedral decomposition of an almost toric fibration. Suppose that  $B$ is the base of an almost toric fibration, so that
$B - B^{\foc}$ has an affine structure.

\begin{definition} A polyhedral decomposition is a collection $\cP = \{ P \} $
of subsets $P \subset B$ so 
\begin{enumerate}
    \item each $P$ is  locally  a simple rational convex polyhedra with respect
to the affine structure on $B - B^{\foc}$
\item each face of $P \in \cP$ is also in $\cP$,
and 
\item  the intersection of any two elements of $\cP$ is either empty or a face of each.
\end{enumerate}
\end{definition}

\begin{definition} \label{def:admissible}
    A polyhedral decomposition $\cP$ is {\em admissible} for $X$ if each element $P \in \cP$ is a simple polytope that is transverse to the moment map  $\Phi: X \to B$, and so the inverse image $\Phinv(P)$ of each $P \in \cP$ is a smooth submanifold of codimension $\codim(P)$. 
\end{definition}

For an admissible decomposition, if $\dim(B) =2 $ then each component of $B^{\foc}$ is contained in the interior of some polyhedron. In higher dimension, the functions cutting out the polytopes are smooth components of the map $\Phi: X \to B$.       Let $\t_P \subset T_b B$ denote the vector space that is the annihilator of the tangent space to $P$, independent up to isomorphism of the choice of $b \in P$. Let $T_{P}$ denote the torus with Lie algebra $\t_P$,  and $T_{P(v),\C}$ its complexification.   By assumption, $T_P$ 
acts with finite stabilizers on $\Phinv(P)$, for each $P \in \cP$.

A polyhedral decomposition  leads to a degeneration of the given symplectic manifold.  By definition, the composition of $\Phi: X \to B$  with the projection near $\Phi^{-1}(P)$ to $\t_P^\dual$  is a moment map   for a $T_P$ action near $\Phi^{-1}(P)$.  The {\em degeneration associated to $\cP$} is the union 
\[ X_\cP = \cup_{P \in \cP} \on{int}(X_P) , \quad \on{int}(X_P) = \Phi^{-1}(\on{int} (P))/T_P .\] 
We denote by 
\[ X_P := \bigcup_{Q \subset P} \on{int} X_Q  \] 
the union over faces $Q$ of $P$.  Because of the simplicity assumption on the elements $P \in \cP$, for admissible decompositions each $X_P$ is a symplectic orbifold (although a precise definition of orbifold is not necessary for this paper.)    Existence of a cutting data for the cuts considered in this paper is proved in \cite[Proposition 3.28, Example 3.30]{vw:trop}.    Associated to each face $P$ we have a {\em thickened piece} 
\begin{equation} \label{eq:xxp} \XX_P = \Phi^{-1}(P) \times \t_P^\dual \end{equation}
which is equipped with an almost complex
structure invariant under translation by the $\t_P$-action.   Denote by 
\[ \XX = \bigcup_{P \in \cP} \XX_P \]
the disjoint union.

The neck-stretching limit for requires the choice of {\em cutting data} describing the rate of neck stretching in each neck: for each polytope 
$P \in \PP$, a {\em dual polyhedron} 
\[ P^\dual \subset \t_{P}^\dual  \] 
of top dimension so that whenever $Q$ is a face of $P$,  the inclusion of $\t_P$ in $\t_Q$ induces an isomorphism of $P^\dual$ with a face of $Q^\dual$; see \cite{vw:trop} for a more detailed description. 
An example is given in Figure \ref{fig:b8p2-dual2} for the degree one del Pezzo, from \cite{vw:at}.

\begin{figure}[ht]\begin{center} \scalebox{.6}{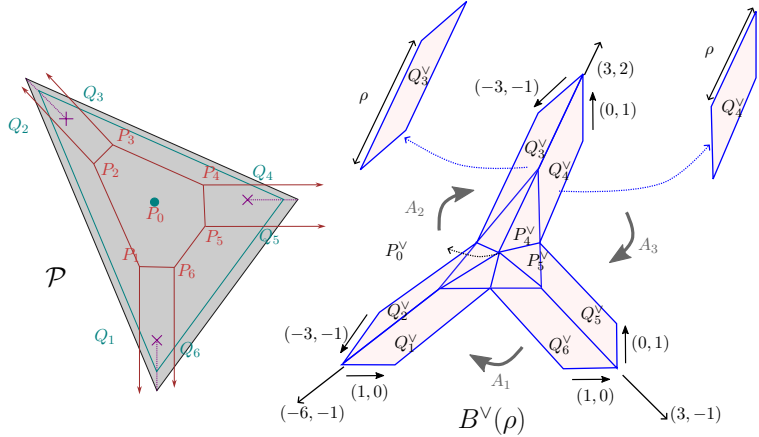}
\end{center} \caption{Dual complex for the degree one del Pezzo}  \label{fig:b8p2-dual2}
\end{figure}

The results in Venugopalan-Woodward \cite{vw:trop} on the limits of holomorphic disks under neck-stretching extend to Lagrangian boundary conditions 
under translation-invariance conditions on the Lagrangians near the facets of the polytope. 

\begin{definition} {\rm (Partial translation actions on the necks)} \label{def:partrans} For each $P \in \cP$, identify a neighborhood of $\Phinv(P) \times \{ 0 \} $ in $\XX_P$  with a neighborhood of $\Phinv(P)$ in $X$; this neighborhood is then equipped
with a partial action of the group $\t_P^\dual$ by translation in the second factor 
in \eqref{eq:xxp} under the identification with 
an open subset of $\Phinv(P) \times \t_P^\dual$, meaning an open subset 
of $\{ 0 \ \times \XX_P$ in $\t_P^\dual \times \XX_P$ so that the {\em partial translation map}
\[ U \to \XX_P, (t,(t_0,l)) \mapsto ( t + t_0,l) \]
is defined and satisfies the usual axioms for an action.   Similarly, for any face $Q$ of $P$, 
the inverse image of an open neighborhood of $Q$ in the thickening $\XX_P$ has a partial action of 
$\t_Q$.
\end{definition}

\begin{definition}  \label{def:trans}
\begin{enumerate}
\item  A subset $L \subset X$ is {\em translation-invariant}
for $P \in \cP$ if $tl \in L$ for all such pairs $(t,l) \ni \t_P^\dual \times L$
for which the action is defined.   
\item  A {\em breakable Lagrangian} with respect to a polyhedral decomposition $\cP$ for $X$ is a Lagrangian submanifold 
$L \subset X$ so that for any $P \in \cP$, the Lagrangian $L$ is translation-invariant in a neighborhood of   $\Phinv(P) \cap L$.
\end{enumerate}
\end{definition} 

\begin{remark}   Breakable Lagrangians give rise to totally-real submanifolds in cylinders as follows.  The product
\[ \LL_{P} := (X_P \cap L) \times \t_P^\dual \]
is a totally real submanifold in the product $\XX_P = \Phinv(P) \times \t_P^\dual$.
Note that while $\XX_P$ has an (at least orbifold) compactification $\ol{\XX}_P$ whose image under
the moment map is the product $P \times P^\dual$, the Lagrangian $\LL_P$ typically does not have a smooth compactification but rather is a closed subset of an immersed Lagrangian submanifold with clean self-intersection; see \cite[Section 4.1]{bcsw:leg1}.
\end{remark}

For our regularization scheme, we choose a collection of Donaldson hypersurfaces as follows. 
A {\em broken Donaldson hypersurface} for $(\XX,\LL)$ is a collection of symplectic submanifolds 
\begin{equation} \label{eq:ddp} \ul{\DD} = (D_P)_{P \in \cP} \end{equation}
satisfying the compatibility condition \cite[Section 5.3]{vw:trop}, so that the projection of $D_P$ to $X_P = \XX_P/T_\C$ is a Donaldson hypersurface, and  the image of $\LL_P$ in $X_P $ is exact in the complement of $D_P$. 

\begin{lemma} \label{lem:donald}  
Let $\LL = (\LL_P)$ be a broken Lagrangian with the property that $\LL_P$ is empty whenever $P \in \cP$ is a vertex. Then a Donaldson hypersurface exists for $(\XX,\LL)$.
\end{lemma}

\begin{proof}  The construction of such hypersurfaces for rational $\LL$ used in this paper given in  \cite[Section 5.3]{vw:trop} (for the case where the Lagrangians do not extend over the neck) and \cite[Proposition 4.7]{bcsw:leg1} (for the case where the Lagrangians do extend over the neck). 
\end{proof}

Broken maps are collections of holomorphic maps corresponding to vertices
of {\em tropical graphs} defined as follows.

\begin{definition} \label{def:reeborbit}
For each $P \in \cP$ denote the space of {\em Reeb orbits} 
\[ \RR_\black(P) = \{ \gamma: S^1 \to \Phinv(P), \  t \mapsto t^\mu x \} \]
where $\mu \in \t_{P,\Z}$ is some integral direction and $x \in \Phinv(P)$.
Similarly, for $L \subset X$ a breakable Lagrangian denote the space of {\em Reeb chords}
\begin{equation} \label{eq:rc} \RR_\white(P) = \{ \gamma: [0,1] \to \Phinv(P) , \quad t \mapsto t^\mu x , \quad  \gamma(\{ 0, 1 \}) \subset L \} . \end{equation}
The {\em set of directions of Reeb chords} is
\[ \t_{P,\white} = \{ \mu \in \t_P \ | \ t^\mu x \ \in \RR_\white(P) \}  .\] 
Analogously the set of directions of Reeb orbits is $\t_{P(e), \black} := \t_{P(e),\Z}$.
\end{definition}

\begin{lemma} \label{lem:discrete}  Let $\cP$ be an admissible polyhedral decomposition for $X$.  For each polytope $P \in \cP$, the set  $\t_{P,\white}$ of directions of Reeb chords is discrete. 
\end{lemma}

\begin{proof}  The statement of the Lemma follows immediately from the fact that  the intersection of $L$ with each $T_P$-orbit in $\Phinv(P)$ is finite. Indeed, the Lagrangian $L$ is $\t_P^\dual$ invariant
by assumption near $\Phinv(P)$.
\end{proof}

\begin{definition} \label{def:openclosed} An {\em open-closed tropical graph} for a polyhedral decomposition $\cP$ is a tropical graph $\Gamma$ in $B^\dual$, as in Definition \ref{def:tropgraph},  for which 
\begin{enumerate} 
\item the edge set $\Edge(\Gamma)$ consists of a collection of edges of {\em finite type}
$\Edge_{< \infty}(\bGamma)$ each containing two vertices, and {\em leaves} $\Edge_\rightarrow(\Gamma)$ each containing a single edge; 
\item the edge and vertex sets  admit partitions
into subsets of {\em closed} resp. {\em open} subsets
\[ \Edge(\Gamma) = \Edge_\black(\Gamma) \cup \Edge_\white(\Gamma), 
\quad  \Ver(\Gamma) = \Ver_\black(\Gamma) \cup \Ver_\white(\Gamma)  \] 
corresponding to  cylinders and strips, for edges of finite type,
and  holomorphic sphere and holomorphic disk components, for vertices; the leaves $\Edge_\rightarrow(\Gamma)$ are all of interior type and correspond to intersections with the Donaldson hypersurface $\DD$;  the open vertices $\Ver_\white(\Gamma)$
and open edges $\Edge_\white(\Gamma)$ are required to map to the dual complex
$\Lambda^\dual \subset B^\dual$ of \eqref{eq:lambdadual};
\item for each closed edge $e \in \Edge_\black(\Gamma)$ a direction $\cT(e)$  in the integral
lattice 
\[ \t_{P(e), \black} := \t_{P(e),\Z} \] 
and
\item for each open edge $e \in \Edge_\white(\Gamma)$ the corresponding direction $\cT(e)$ lies in the set $\t_{P(e),\white}$
of Definition \ref{def:reeborbit}.
\end{enumerate}
\end{definition} 

\begin{definition}  Let $\Gamma$ be an open-closed tropical graph in $B^\dual$.    A {\em broken map}
modelled on $\Gamma$ consists of the following data:
\begin{enumerate}
\item for each closed vertex $v \in \Ver_\black(\Gamma)$, a holomorphic sphere
\[ u_v: S_v \to \XX_{P(v)} \] 
with finite Hofer energy as defined in \cite[Section 3]{vw:trop} (so that the map $u_v$ extends to a map $\ol{u}_v$ to the compactification $\ol{\XX}_{P(v)}$ described in \cite[Section 3]{vw:trop});
\item for each open vertex $v \in 
\Ver_\white(\Gamma), $ a holomorphic disk $u_v$ in $\XX_{P(v)}$ with Lagrangian boundary condition 
in $\LL_{P(v)}$ with finite Hofer energy (so that $u_v$ extends to a
disk in $\ol{\XX}_{P(v)}$ with boundary condition $\ol{\LL}_{P(v)}$);  

\item at each node $w_e = (w^-_e, w^+_e)$ corresponding to a closed edge $e$ of the tropical graph, a {\em framing} given by an isomorphism
\begin{equation} \label{eq:framingeq} \fr_e : T_{w^+_e} S_{v_+}
  \otimes T_{w^-_e} S_{v_-} \to \C;
\end{equation}
\item at each node $w_e = (w^-_e, w^+_e)$ corresponding to an open edge $e$ of the tropical graph, a {\em framing} given by an isomorphism
\begin{equation} \label{eq:openframingeq} \fr_e : T_{w^+_e} S_{v_+}
  \otimes T_{w^-_e} S_{v_-} \to \H := \{ z \ge 0 | z \in \C \}
\end{equation}
 mapping the tangent space to the boundary to $\R$;  and
 \item for each leaf $e \in \Edge_\rightarrow(\Gamma)$, a  marking $z_e \in C_v$ mapping to the corresponding Donaldson hypersurface $\DD_P$ from \eqref{eq:ddp};
\end{enumerate}
satisfying the following matching conditions:  For any holomorphic
  coordinate
  \begin{equation}
    \label{eq:holcoord}
    z_\pm : (U_{w_e^\pm}, w_e^\pm) \to (\C,0)  
  \end{equation}
  in the neighborhood of nodal lifts $w_e^\pm$ that respect the framing, that is,
  \begin{equation}
    \label{eq:framingcoord}
    \d z_+(w^+_e) \otimes \d z_-(w^-_e)=\fr_e,
  \end{equation}
  the following matching condition is satisfied:
  \begin{equation}
    \label{eq:evalu}  
    \lim_{z_- \to 0} z_-^{-\cT(e)} 
    u_{v_-}(z_-) = 
    \lim_{z_+ \to 0} z_+^{\cT(e)} 
    u_{v_+}(z_+).
  \end{equation}
Here $z_\pm^{\pm \cT(e)} := \exp(\pm  \ln(z_\pm) \cT(e))$ using the exponential map $\t_{P(e)} \to T_{P(e)}$ and the partial translation action was defined in Definition \ref{def:partrans}.
\end{definition}

The possibilities for the directions meeting vertices  of rigid graphs corresponding to toric surfaces are particularly easy  to describe: They correspond to triples  satisfying the balancing condition described in  \eqref{eq:balance}; in general the description of possible vertices is more complicated.   A sample polyhedral degeneration  of the almost toric diagram of the degree four del Pezzo surface is shown in Figure  \ref{poly}.

\begin{figure}[ht]\begin{center} 
\scalebox{.5}{\includegraphics{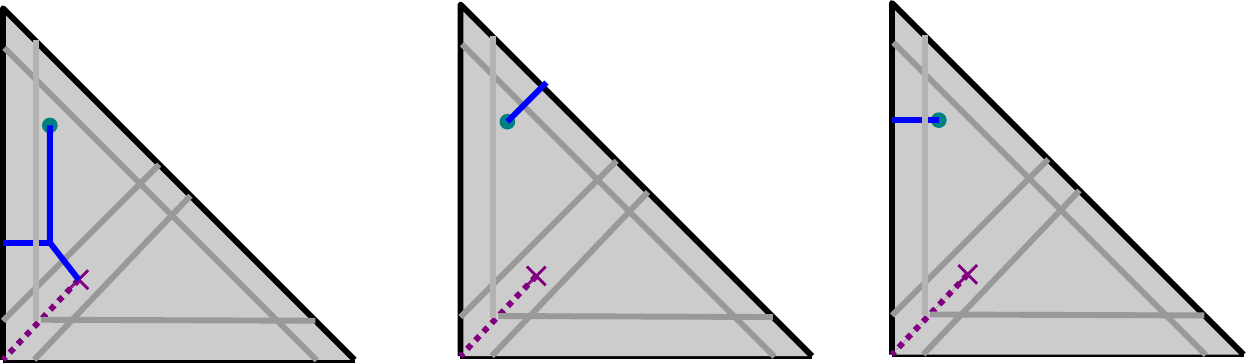}}
\end{center} 
\caption{A polyhedral decomposition of an almost toric diagram
and cartoon diagrams of three Maslov-two broken disks bounding a moment fiber}
\label{poly}
\end{figure}
\begin{definition}
An {\em open-closed map type} $\bGamma$ is an open-closed tropical graph $\Gamma$
with vertices $v \in \Ver(\Gamma)$ labelled by homotopy classes $[u_v]$ of disks $u_v$ (for open vertices $v$) or spheres  $[u_v]$ (for closed vertices $v$).   Denote by $\M_\bGamma(\XX,\LL)$ the moduli space of maps of type $\bGamma$.  
\end{definition}

    \begin{definition} \label{def:rigid} (c.f. \cite[Definition 8.48]{vw:trop})   A broken map of type $\bGamma$ is {\em rigid} if the tropical graph $\Gamma$ is rigid, the moduli space $\M_\bGamma(\XX,\LL)$ is expected dimension zero,  the stratum of the source moduli space $\M_\Gamma$ is of codimension zero in the moduli space of treed disks.  Let  $\M(\XX,\LL)$ denote the collection of rigid broken maps,  written as a union 
over rigid  types $\bGamma$
\[ \M(\XX,\LL) = \bigcup_{ \bGamma} \M_\bGamma(\XX,\LL) . \]
    \end{definition}

\begin{definition}
The area of a broken map $u = (u_v)$ of type $\bGamma$ is the sum
    \begin{equation} \label{eq:au} A(u) = \sum_{v \in \Ver(\bGamma)} A(u_v), \quad A(u_v) := A(\pi_{P(v)} \circ u_v) \end{equation}
    of the areas of the maps $u_v \circ \pi_{P(v)} $ to the symplectic manifolds $X_{P(v)}$ obtained
    by composing $u_v$ with the projection 
    \[ \pi_{P(v)}: {\XX}_{P(v)} \to {X}_{P(v)} \]
whose fibers are (either empty or) the complexified tori $T_{P(v),\C}$.
\end{definition}

\begin{definition} \label{def:index}
The {\em Maslov index} of a broken map $u = (u_v)$ of type $\bGamma$ is the sum
 of Maslov indices
    \begin{equation} \label{eq:iu}  I(u) = \sum_{v \in \Ver(\bGamma)} I(u_v), \end{equation} 
    where $I(u_v)$ is the Maslov index of the bundle pair defined as follows:
    Consider the compactification $\ol{S}_v$ obtained by taking the real blow-up of $S_v$ at the interior and boundary punctures; this has the effect of adding a circle at each interior puncture, and an interval at each boundary puncture.     The bundles $u_v^* T \XX_{P(v)}$ and
    $(\partial u_v)^* T \LL_{P(v)}$ have natural extensions to the compactification, given by adding the tangent space to $\XX_{P(v)}$ resp. $\LL_{P(v)}$ along corresponding Reeb orbit resp.  chord at infinity along the cylindrical or strip-like end.  The bundle $\XX_{P(v)}$ along any Reeb orbit $\gamma: S^1 \to \XX_{P(v)}$ is trivializable, with a trivialization given by choose a trivialization at any point any translating along the Reeb orbit.  Choose any Lagrangian subspace of $T_{\gamma(0)} \XX_{P(v)}$; by translation one obtains a Lagrangian sub-bundle along $\gamma$. 
    Denote by $(T \XX_{P(v)})_{\infty,\R}$ the Lagrangian sub-bundle over the circles at the limits of the cylindrical ends obtained in this way.
    Then the Maslov index of $u_v$ is defined by 
    \[ I(u_v) := I( u_v^* T \XX_{P(v)}, (\partial u_v)^* T \LL_{P(v)} \cup (T \XX_{P(v)})_{\infty,\R}) . \]
\end{definition}

\begin{remark} Associated to any broken map $u = (u_v)_{v \in \Ver(\Gamma)}$ is a homology class 
in $[u] \in H_2(X,\Z)$, defined by $[u] = [\ti{u}]$ where $\ti{u}: C \to X$ 
is the result of the pre-gluing procedure in Venugopalan-Woodward
\cite[Section 9.1]{vw:trop} to $(u_v)$, using any collection of gluing parameters to glue along cylindrical and strip-like ends.    Then $A(u) = A(\ti{u})$ and $I(u) = I(\ti{u})$, by standard gluing properties of area and index.  
\end{remark}

Associated to the leaves of the graph are various evaluation maps. 
 Denote by $\Edge_{\rightarrow}(\Gamma)$ the leaves of $\Gamma$.   
 Associated to each edge $e \in \Edge_{\rightarrow}(\Gamma) $ with zero direction $\cT(e) = 0$ is an evaluation map 
 \[ \ev_e: \M_\Gamma(\XX,\LL) \to \LL_{P(e)}  .\] 

\begin{definition}
    A {\em system of constraints} for a map type $\bGamma$ is a collection 
of oriented compact embedded submanifolds 
\[ \ul{\YY} = (\YY_e \subset \LL_{P(e)})_{e \in \Edge_\rightarrow(\bGamma)}  \]
  as follows:
  \begin{enumerate}
      \item If $e \in \Edge_\black(\Gamma)$ is a closed leaf then  $\YY_e$ is a compact, oriented, submanifold of the space of Reeb orbits $\RR_{\black}(P(e)) = \XX_{P(e)}$.
\item If $e \in \Edge_\white(\Gamma)$ is an open leaf, representing a strip-like end, then $\YY_e$ is a compact oriented submanifold of the space of
Reeb chords $\RR_{\white}(P(e))$ as in \eqref{eq:rc}.  
\end{enumerate} 
%
Denote by  $\M(\XX,\LL,\ul{\YY}) \subset \M(\XX,\LL)$ the moduli space with the given constraints.   
\end{definition}

\begin{definition} \label{def:cartoon} A {\em cartoon diagram} of a broken map $u = (u_v)$
is a collection of subsets 
\[ \fC = (\fC_v \subset P(v)) \] 
so that 
\begin{enumerate}
    \item each $\fC_v$ is the moment image $\Phi_{P(v)} \circ u'_v(C_v)$ for some smooth map $u_v'$ in the homotopy class of $u_v$, and 
    \item near any intersection of $\Phi_v$ with the polytope $P(e)$ corresponding to an adjacent edge $e \ni v$, $\Phi(v)$ is given by a path which meets $P(e)$ at a point with tangent given by the direction $\cT(e)$.
    \end{enumerate}
\end{definition}

\begin{remark}  The tropical graph can be constructed from the cartoon diagram, by reading off the directions of the edges at the vertices.
\end{remark}

\subsection{Tropical limit theorems}
\label{sec:limit}

The results of \cite{vw:trop} on the broken limit of holomorphic maps require minor modifications
in the definitions and results for Lagrangians extending over the neck.

\begin{theorem}\label{thm:cob} ( Venugopalan-Woodward
\cite{vw:trop} for moment fibers, Chanda \cite{chanda} for Lagrangians passing through the neck) 
Let $L \subset X$ be a breakable monotone Lagrangian submanifold for a polyhedral decomposition $\cP$.  There exists a compact one-dimensional cobordism from 
the moduli space $\M(X,L,\ul{Y})$  of rigid disks of Maslov-index-two passing through a generic point $\ul{Y} = ( \{ p \} \subset L)$ to the moduli space of rigid broken disks of Maslov-index-two passing through a generic point $\ul{Y} = \{ p \} \subset \LL $
\[ \M(X,L,\ul{Y})  \sim \bigcup_{\Gamma} \M_\Gamma(\XX,\LL,\ul{Y}) .\] 
\end{theorem}

\begin{proof}[Sketch of Proof] 
We describe
the other modifications of the results of Venugopalan-Woodward \cite{vw:trop} necessary to apply the degenerations in cases where the Lagrangians extend over the neck; unfortunately, the notation in \cite{vw:trop} was already so involved that we felt that this additional layer of complication would make the exposition there even longer.  

\begin{enumerate} 
\item  {\rm (Exponential Convergence)}  In the proof of  \cite[Theorem 8.2]{vw:trop}, to obtain exponential convergence along the ends we used the fact that the composition of a holomorphic map to $\XX_{P(v)}$ with 
projection to the base $X_{P(v)}$ satisfies removal of singularities.  For each $v$, consider the projection $\XX_{P(v)} \to X_{P(v)}$ with fibers
$T_{P(v),\C}$.  The various branches of the Lagrangian $\LL_{P(v)}$
project to a fixed Lagrangian $L_{P(v)}$ in $X_{P(v)}$.  Each fiber is a collection 
of copies of the Lie algebra $\t_{P(v)}.$  The necessary exponential convergence
results follow by considering the projection to the base.  There one has the usual annulus lemma.   Applying the arguments in the proof of \cite[Proposition 8.18]{vw:trop} to obtain the necessary exponential decay results for the difference 
between the given map and a constant cylinder.   
\item  {\rm (Directions of open edges)}  In \cite[Definition 4.25]{vw:trop}, there is no notion of primitive direction for edges corresponding to boundary nodes.  Instead, the boundary edges have only rational directions.
\item  {\rm (Monotonicity)}   The monotonicity \cite[Lemma 7.21]{vw:trop} for pseudoholomorphic maps  requires the version for Lagrangian boundary conditions, 
as in Cieliebak-Ekholm-Latschev \cite[Lemma 3.6]{cel:switch}.
\item {\rm (Removal of singularities)}  The argument in removal of singularities
\cite[Section 7.5]{vw:trop} requires modification to include the case of boundary punctures.  In this case, one uses the fact that the closure of the Lagrangian 
in the compactification of each piece has clean self-intersection, and removal of singularities for clean intersection is well-known; see for example Schm\"aschke \cite{schmaschke:clean}. 
\item {\rm (Gromov convergence)}  The definition of Gromov convergence of \cite[Section 8.1]{vw:trop} requires inclusion
of the case of strips passing through the neck;  cylinders in the definition are replaced by cylinders or strips in the definition.  
\item {\rm (Fredholm theory)}  In \cite[Section 6.4]{vw:trop} on Fredholm theory, one needs that the linearized Cauchy-Riemann operator on the given  weighted spaces with totally real boundary conditions is  Fredholm.  In the case that the almost complex structure is compatible, the tangential part of the linearized operator to the boundary at infinity is self-adjoint and the Fredholm property is standard as in \cite[Lemma 3.5]{ww:quilt}. The case that the almost complex structure is only tamed seems not to be treated in the literature, but the Fredholm property for the linearized Cauchy-Riemann operator seems to follow from Coriasco-Schrohe-Seiler \cite[Theorem 2.22]{schrohe}.  However, we only need the case that the almost complex structure is compatible. 
\end{enumerate}
The remaining changes in the proof are straightforward (replacing cylinders by cylinders or strips in the proofs, depending on the type of edges.)
\end{proof}

\begin{remark} In particular, holomorphic disks in the limit always correspond to tropical graphs, rather than tropical subvarieties of higher dimension as in \cite[Figure 3]{hicks:obs}.   
\end{remark}

\subsection{Relative maps}

We introduce the following terminology of {\em relative maps} similar to that of levels in symplectic field theory, so that a broken map is a collection of relative maps satisfying matching conditions. 

\begin{definition}  Given a tropical graph $\Gamma$ denote by $\Edge_\times(\Gamma) \subset \Edge(\Gamma)$ the set of edges $e$ with zero direction $\cT(e) = 0$.   Collapsing 
the zero-direction edges $\Edge_\times(\Gamma)$ in $\Gamma$ gives a graph $\Gamma_0$ equipped with a tropical structure so that every edge $e \in \Gamma_0$ has non-zero direction $\cT(e) \neq 0 $.For each vertex $v_0 \in \Ver(\Gamma_0)$, let $\Gamma_{v_0}$ denote its inverse image; the corresponding collection of components
$u_{v_0} := (u_{v}, v \in \Gamma_{v_0})$ will be called a {\em level} of the broken map $u$, and the edges collapsed under the morphism $\Gamma \to \Gamma_0$ correspond to nodes of $ u_{v_0}$.
\end{definition}

We view the levels of a broken map as {\em relative maps} to a cylindrical end manifold as follows.  
We give a utilitarian definition, since all of our cylindrical end manifolds arise from symplectic cutting and removing divisors. 

\begin{definition}
\begin{enumerate} 
\item A {\em cylindrical end symplectic manifold} is a component $X = \XX_P$ of a broken manifold  $\XX$ as defined in \cite{vw:at}, with $P$ some polytope in a polyhedral decomposition $\cP$.  
\item A {\em cylindrical end Lagrangian} is a component $L = \LL_P$  of a broken Lagrangian $\LL = (\LL_P)$ contained in $\XX_P$.    
\end{enumerate}
\end{definition}

\begin{remark}  In particular, any cylindrical end symplectic manifold $X$ has a map $\Phi$ to the interior $\on{int}(P)$ of a polytope $P \subset B$, with the property  that for any face $Q$ of $P$, there is an open neighborhood $U$ of $Q$ in $P$ and a symplectomorphism from $\Phinv(U)$ to  a manifold $\XX_Q = Z_Q \times \t_Q^\dual$ as in 
\eqref{eq:xxp}, where $Z_Q$ is a smooth manifold equipped with an action of $T_Q$ with only finite stabilizers.
The two-form on $\XX_Q$ is given by 
\[ \omega = \pi_1^* \omega_Q + \d \lan \alpha_Q, \pi_2^*  \ran  \] 
where  $\omega_Q$ is a closed two-form on $Z_Q$, $\alpha_Q$ is a connection form for the $T_Q$ action on $Z_Q$,
and $\pi_1,\pi_2$ are projections on the factors.

Any cylindrical end manifold $X$ has a natural orbifold compactification, obtained by adding back in the divisors $X_Q$ corresponding to faces of $P \times P^\dual$, as in \cite[Section 3]{vw:trop}.  The closure $\ol{L}$ of $L$ in $\ol{X}$ is a subset of a cleanly-self-intersecting immersed Lagrangian submanifold, as in \cite[Proposition 2.22]{bcsw:leg1}.  
\end{remark}

Our previous work \cite{vw:at} defines relative maps in $X$ with boundary in $L$  as components of broken maps.  
Moduli spaces of relative maps $\M_{\bGamma}(X,L)$ have additional evaluation maps 
at the strip-like and cylindrical ends given framings.   

\begin{definition}Given a subset $\Edge^-(\bGam)$
of the set of edges $\Edge(\bGam)$, a {\em relative map with partial framings} is a relative map with the additional data of, for each edge $e \in \Edge^-(\bGamma)$, a trivialization of the tangent space at the corresponding marked point $w_e$
\begin{equation} \label{eq:partialframingeq} \fr_e : T_{w_e} S_{w(e)}  \to \C.
\end{equation}
A local coordinate $z$ near the puncture $z_e$ is {\em compatible}
with the framing if the differential $\d z$
 is equal to $\fr_e$. 
 \end{definition}

\begin{remark} For any 
leaf $e \in \Edge^-(\bGamma)$ representing a boundary puncture we have an evaluation map 
\[ \ev_e: \ \M_\bGamma(X,L) \to \RR_\white(P(e)), \quad u \mapsto  (t \mapsto \lim_{s \to \infty}
\exp(-\cT(e)(s + it)) u_{v_-}(s + it)) \]
where $s + it = - \ln(z)$ and $z$ is compatible with the framing.
There are similar evaluation maps for cylindrical ends mapping to $ \XX_{P(e)}$. 
\end{remark}

\begin{definition}
    A {\em system of constraints} for a type $\bGamma$ of relative map is a collection 
of submanifolds for a subset of {\em constrained edges $\Edge^-(\bGamma) $}
\begin{equation} \label{eq:con} 
\ul{\YY} = (\YY_e)_{e \in \Edge^-(\bGamma)} \end{equation} 
where 
each $Y_e$ is the target of the evaluation map $\ev_e$, that is, either the space of Reeb orbits, chords, or the Lagrangian itself.  
\end{definition}

\begin{definition}
Given a system of constraints $\ul{\YY}$, a {\em constrained relative map}
is a relative map with framings at $\Edge^-(\bGam)$ satisfying the following matching conditions:  For any holomorphic
  coordinate
  \begin{equation}
    \label{eq:onesideholcoord}
    z : (U_{w_e}, w_e) \to (\C,0)  
  \end{equation}
that respect the framing, that is,
  \begin{equation}
    \label{eq:onesideframingcoord}
    \d z(w_e) =\fr_e.
  \end{equation}
  the following matching condition is satisfied:
  \begin{equation}
    \label{eq:onesideevalu}  
   \ev_e(u) \in \YY_e .
  \end{equation}
\end{definition}

In general, the counts of maps with constraints depend on the choice of almost complex structure, in the same way that counts of holomorphic disks depend on this choice in the closed case as well.  However, in the prime case these counts are invariant. 

\begin{definition} \label{def:prim} A stable map type $\bGamma$ is {\em prime} if $\bGamma$ does not admit a degeneration to different stable type $\bGamma'$.  That is, with respect to the natural partial order $\preceq$ on types corresponding to degeneration so that the boundary of $\M_\bGamma(\XX,\LL)$ is contained in the union of strata $\M_{\bGamma'}(\XX,\LL)$ for which $ \bGamma' \prec \bGamma$ there is no type $\bGamma'$ so that $\bGamma' \prec \bGamma$. 
\end{definition}

\begin{example} \label{ex:prime} Let $\bGamma$ be map type so $P(v)$ is top-dimensional and $A(u_v)$ is non-zero and minimal for each $v \in \Ver(\Gamma)$, and the edges $e \in \Edge(\bGamma)$ have the property that no two directions $\cT(e_1), \cT(e_2)$ lie in the same dual cone to a face of $P$.  Then $\bGamma$ is prime.
Indeed, only possible degeneration corresponds to the formation of zero-area components containing at least two such edges $e_1,e_2$.  Such a degeneration is impossible since $\cT(e_1), \cT(e_2)$ must lie in the dual cone to some face $Q$ of $P$; 
see \cite[(1.7)]{vw:trop}.
\end{example}

\begin{lemma} \label{lem:elem} \label{lem:prim}
    Let $X$ be a cylindrical end manifold and $L \subset X$ a cylindrical end Lagrangian,
    and let $\bGamma$ be a prime type of broken map to $X$ with boundary in $L$  with a single vertex. 
     Let  $ \ul{Y} = (Y_e)_{e \in \bGamma}$ be a collection of constraints so that the moduli space $\M(L,\ul{Y})$ of maps of type $\bGamma$ with constraints $\ul{Y}$ is expected dimension zero.    Then the count 
    $m(\bGamma)$ of maps of type $\bGamma$ is independent of the choice of domain-dependent almost complex structure $J_\Gamma$ and depends only on the isotopy class of the constraints
    $\ul{Y}$. 
    \end{lemma}

\begin{proof} The proof of the statement of the Lemma is a standard cobordism argument:  Given two such data $(J_b,  \ul{P}_b)$
for $b \in \{ 0,1 \}$ let $(J_t,  \ul{P}_t), t \in [0,1]$ be a generic homotopy between them.  Let 
$\ti{\M}_{\bGamma}(L,Y)$ be the parametrized moduli space for the family.   By assumption, degeneration of the domain is impossible in the parametrized moduli space.  The parameterized moduli space is therefore a compact oriented cobordism 
between the moduli spaces for $(J_b, \ul{P}_b)$
for $b \in \{ 0,1 \}$, and the signed count of the moduli spaces at the ends is equal.  The argument for variation of constraints $\ul{Y}$ is similar.  
\end{proof}

\begin{remark}   All of the combinatorial types described
in Definition  \ref{def:mv} are easily seen to be prime with the following exception:   Let $\LL_v \subset (\C^\times)^2$ be the diagonal, 
and $u_v$ a punctured holomorphic disk with boundary in $\LL_v$. 
Thus the tropical graph $\Lambda$ of the Lagrangian
$L$ has a single edge $e$ with direction $\cT(e) = (1,1)$, and the open edge $\eps$
of the tropical graph $\Gamma$ of $u_v$ has direction $\cT(\eps) = (1/2,1/2)$.
Suppose the point constraint $p \in \LL_v$ approaches the real part $\Delta_\R$ of the diagonal (thinking of the Lagrangian as the anti-holomorphic diagonal).  In the limit, the boundary of $u_v$
lies on $\Delta_\R$, and the Gromov limit of $u_v$ consists
of this map an a bubble containing both the open and closed ends.    The real diagonal $\Delta_\R \cong \R_+ \sqcup \R_-$ separates $L \cong \R \times S^1 \cong \C^\times$ into two parts, each isomorphic to $\R \times (-1,1)$. 
\begin{enumerate} 
\item If the point constraint $p$ lies in the first part, there is a holomorphic pant  $u_v$
with a closed end in the direction $(1,0)$.
\item If the point constraint $p$ lies in the second part, there is a holomorphic pant $u_v$ with a closed end with direction $(0,1)$. 
\end{enumerate}
The count therefore depends on which half of the Lagrangian $\LL_{P(v)}$ the point constraint is chosen to lie in, or equivalently, on the choice of almost complex structure, since there exists a symplectomorphism which maps a point in one half of $\LL_{P(v)}$ to the other. 
We choose the count to be $-1/2$ for each of these possibilities, 
so that the count is ``maximally symmetric''.
\end{remark}

\begin{remark} \label{rem:rel}
Given a choice of relative spin structure on $L$ and orientations on the constraints $\ul{Y}$, the moduli spaces $\M_{\bGamma}(X,L,Y)$
are {\em oriented} non-canonically, by choosing relative spin structures
on the extensions described in Definition \ref{def:index} and applying the constructions of Fukaya-Oh-Ohta-Ono \cite{fooo:part1}, see also 
Wehrheim-Woodward \cite{orient}.
\end{remark}

\subsection{Distribution of constraints}

 In this section we show that the moduli space of rigid broken maps to a broken symplectic four-manifold is, up to cobordism, a product of moduli spaces associated to the vertices, 
 after adding an appropriate set of constraints in the sense of \eqref{eq:con}.
 
 \begin{theorem}  \label{thm:split}  Let $\LL \subset \XX$ be a broken Lagrangian as above.  Let
 $\bGamma$ be a rigid type of broken map bounding $\LL$ so that the homotopy types of graphs $\bGamma(v) \subset \bGamma$ associated to the vertices $v$  are all prime.    Then there exists a collection of constraints
 \[ \ul{\YY} = (\ul{\YY}_v), \quad v \in \Ver(\bGamma) , \quad \ul{\YY}_v = (\YY_e \subset \XX_{P(e)})_{e \ni v} \] 
 where each $\YY_e$, if non-trivial, is a collection of point constraints so that the moduli space $\M_\Gamma(\LL)$ of unframed broken maps with underlying graph $\Gamma$ admits an oriented bijection with the product of partially framed moduli spaces with constraints $\ul{\YY}_v$:
\[ \M_\Gamma(\LL) \cong \prod_{v \in \Ver(\Gamma)} \M_{\bGamma(v)}(\LL_{P(v)},\ul{\YY}_v) .\]
      \end{theorem}

\begin{proof} The proof of the statement of the Theorem is an induction on the number of vertices in the graph.  For each vertex $v \in \Ver(\bGamma)$, the moduli space  $\M_{\bGamma(v)}( \LL_{P(v)})$ is non-negative dimension,  by regularity.

\vskip .1in \noindent 
{\em Case 1:  There exists a univalent closed vertex
$v \in \Ver_\black(\bGamma)$}.  The vertex $v$ is connected to a unique 
distinct vertex $v' \in \Ver_\black(\bGamma)$ by some edge $e$.    Cutting at the edge $e$ divides $\bGamma$ into two graphs $\bGamma(v)$ and $\bGamma'$,  and the moduli space $\M_{\bGamma'}(\LL_{P(v)})$ must be rigid. 
Since the matching condition 
cuts down the dimension by at most two,  $\M_{\bGamma(v)}(\LL_{P(v)})$ must 
be either dimension zero or two.   

\vskip .1in \noindent {\em Case 1a: The dimension of $\M_{\bGamma(v)}(\LL_{P(v)})$ is two.}  By the inductive hypothesis, there exist constraints $\ul{\YY}'$ so that 
$\M_{\bGamma'}(\LL')$ is a product of 
$\M_{\bGamma'(v')}(\LL_{P(v')},\ul{\YY}_{v'})$ over its vertices $v'.$
For each element of $\M_{\bGamma'}(\LL')$, choose a trivialization of the tangent space  $T_{w^+_e} S_{v_+}$, that is, a framing at the node.   Let 
$\ti{\YY}_v$  be the collection of points  that is the image of 
$\M_{\bGamma'}(\LL', \ul{\YY}')$ under the evaluation map at $w^+_e$.   The  forgetful map 
\[ \M_{\bGamma}(\LL) \to \M_{\bGamma(v)}(\LL_{P(v)},\ti{\YY}_{v}) , \quad (u_v) \mapsto u_v \] 
is a bijection.  Indeed, the inverse to the bijection is obtained by gluing on the corresponding element of 
$\M_{\bGamma'}(\LL', \ul{\YY}')$.    The claim follows from the inductive hypothesis.

\vskip .1in \noindent {\em Case 1b:  The dimension of $\M_{\bGamma(v)}(\LL_{P(v)})$ is zero.}  
  Let 
  \[ \ti{\YY}_e \subset \XX_{P(e)} \]
  denote the image of the moduli space $\M_{\bGamma(v)}$ 
under the evaluation map $\ev_e$, and $\ul{\YY}'$ the point constraint at the edge $e$, and the given constraints at the other edges.  Then 
$\M_{\bGamma'}(\LL',\ul{\YY}')$ is rigid.  The claim follows again by the inductive
hypothesis.

\vskip .1in \noindent {\em Case 2: There is a univalent open vertex
$v \in \Ver_\white(\Gamma)$}.  Let $\bGamma'$ denote the graph obtained from $\Gamma$ by removing $v$.  Since the matching condition 
cuts down the dimension by one, the sum of the dimensions 
$\M_{\bGamma(v)}(\LL_{P(v)})$ and $\M_{\bGamma'}(\LL')$ is equal to one. 
If $\M_{\bGamma(v)}(\LL_{P(v)})$ is dimension one, 
then $\M_{\bGamma'}(\LL')$ is rigid and is a product of moduli spaces 
associated to its vertices, up to cobordism, by the inductive hypothesis. 
Otherwise, $\M_{\bGamma(v)}(\LL_{P(v)})$ is rigid and 
$\M_{\bGamma'}(\LL',\ul{\YY}')$ is rigid
for some constraint $\ul{\YY}'$ given at $e$ by the output of $\M_{\bGamma(v)}(\LL_{P(v)})$. The discussion proceeds as in the closed case. 
\end{proof}

\begin{example} \label{ex:distrib} We give an example of distribution of constraints for a  Maslov-index-two disk with point constraint contributing to the potential for the degree seven del Pezzo, with a fixed point constraint $p \in \LL$. The broken disk $u = (u_v)$ has three components 
corresponding to vertices $v \in \{ v_0,v_1,v_2 \} = \Ver(\Gamma)$ as shown in the 
cartoon diagram in Figure \ref{fig:distribconstraints}, with point constraint in the piece in between the boundary and the trivalent vertex of the Lagrangian.   The component $u_{v_0}$ whose corresponding to the univalent vertex $v_0$ mapping to zero is rigid, and it's evaluation at the end corresponding to the only edge is then considered a constraint for the component $u_{v_1}$ whose vertex $v_1$ is a bivalent vertex with adjacent open and closed edges.  This component is rigid taking into account the point constraint at $p$, so that the number of point constraints on this component is two. Its evaluation at the end is considered as a point constraint for the final piece $u_{v_2}$ which is a punctured holomorphic sphere meeting the toric boundary, which is then rigid.
\end{example}

\begin{figure}[ht]\begin{center} 
\scalebox{.4}{\includegraphics{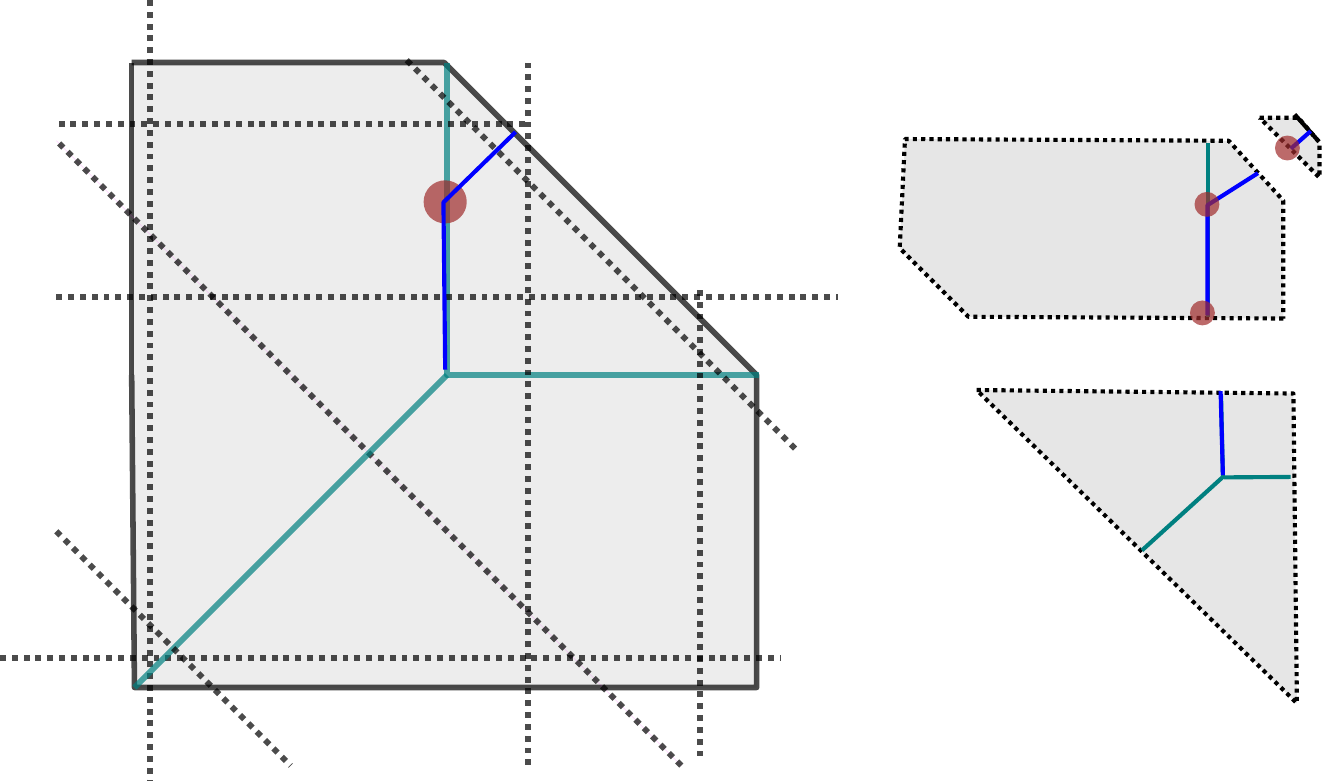}}
\end{center} 
\caption{The moduli space as a product of moduli spaces for its pieces} 
\label{fig:distribconstraints}
\end{figure}

\subsection{The definition of the disk counts}

We now give the weighted counts of disks used in, for example, the definition of the disk potential. 
Only in certain cases will these counts be invariant of the choice of almost complex structure and perturbations. 

\begin{definition}   \label{def:ml}   Let $u = (u_v)_{v \in \Ver(\Gamma)}$ be a broken map with tropical graph $\Gamma$.    Denote the number of interior leaves
    \[ d_\black(\Gamma) \in \Z_{\ge 0 } \]
    of the graph $\Gamma$.  Let
    \[ \gamma(u) \in \{\pm 1\} \]  
    denote the orientation sign defined using the relative spin structure.    The unbroken holomorphic disk count with constraints $\ul{Y}$ is given as the sum    over rigid  maps $u$
\begin{equation} \label{eq:wtwu}  m(L,\ul{Y}) = \sum_{\bGamma} \sum_{u \in \M_{\bGamma}(L,\ul{Y})} 
(d_\black(\Gamma)!)^{-1} 
\gamma(u)
    \in \Q
  \end{equation}
Given a polyhedral decomposition $\cP$, the definition of the broken count $m(\LL,\ul{\YY})$ is similar, and given by a sum over elements of  
 $\M_{\bGamma}(\LL,\ul{\YY})$.
\end{definition}

We emphasize that the numbers $m(v)$ are independent of the choice of almost complex structure only in the case that  $\bGamma(v)$ is prime, as in Lemma \ref{lem:elem}. Theorem \ref{thm:split} immediately implies:

\begin{corollary}  \label{cor:unmarked} 
Let $X$ be an almost toric surface and $L \subset X$ a monotone almost toric moment fiber.  The potential $W_L$ is given by a sum over graphs
$\bGamma$ resp. $\bGamma_0$ with resp. without interior markings 
\begin{eqnarray*} W_L &=&  
\sum_{\bGamma} \frac{1}{d_\black(\bGamma)!}
\prod_{v \in \Ver({\bGamma}) }  d_\black(\bGam(v))! m(v) \\ 
&=& \sum_{\bGamma_0} \frac{1}{\# \Aut(\bGamma_0)}
\prod_{v \in \Ver({\bGamma}_0) } m(v) \end{eqnarray*}
where the first resp. second sum is a sum over isomorphism classes of types of broken disks with resp. without semi-infinite edges of interior type.
\end{corollary}

\begin{proof}  The first equality in the statement of the Corollary is an immediate consequence of the cobordism in 
Theorem \ref{thm:split}.  The second equality is obtained by noting that for each type  $\bGamma$,
the numbers $m(v)$ are independent of the choice of ordering of interior edges associated to the Donaldson hypersurface, and the map $\bGamma \mapsto \bGamma_0$ which forgets the interior edges 
has $d_\black(\bGamma)/\Aut(\bGamma_0)$ points in its fiber. 
\end{proof}

\section{Holomorphic curves in the elementary pieces}
\label{sec:elem}

By an elementary piece, we mean a cylindrical-end almost toric four-manifold
equipped with a connected cylindrical-end tropical Lagrangian containing at most one 
focus-focus singularity or  vertex of the tropical graph of a tropical Lagrangian.  By the degeneration theory in the previous chapter, to count holomorphic disks bounding tropical 
Lagrangians it suffices to count the disks in these elementary pieces.  The general strategy is to apply some degeneration technique 
to the model problems considered in Sections \ref{sec:model1}, \ref{sec:model2} below.   Knowledge of all but one  of the open Gromov-Witten invariants involved gives the value the last invariant from the computation of the model problem.   The  multiplicities \eqref{mv:spheres}, \eqref{mv:pants}, \eqref{mv:bound}, \eqref{mv:multcov} of 
Definition \ref{def:mv} were described in Venugopalan-Woodward \cite{vw:at} in the current foundational scheme, although these formulas already appeared elsewhere in the literature.

\subsection{Disks in the product of projective lines}
\label{sec:model1}

We make two explicit computations of numbers of holomorphic disks in the product of projective lines. 

\begin{example} \label{ex:antidiag}   In the first example, 
we count disks with boundary in the  anti-diagonal two-sphere in the product of projective lines $X = \P^1 \times \P^1 $ with equal symplectic forms on the factors. That is, let 
  \[ L = \{ (z,\ol{z}) \} \cong \P^1  \subset X .\]
The Lagrangian $L$ is the fixed point set $X^\tau$ of the anti-symplectic involution 
\[  \tau: X \to X, \quad (z_1,z_2) \mapsto (\ol{z}_2, \ol{z}_1) .\]

\begin{lemma} \label{lem:anti-diag} \label{lem:p1p1} Let $L \subset X$ be the antidiagonal as above.  The number  $m(L,\ul{Y})$ of holomorphic disks $u: S \to X$ of Maslov index $4$ with boundary in $L$  and for some given generic $p_1,p_2,p \in \P^1$,
 \[ u(z_\black) = (p_1,p_2), \quad u(z_\white) = (p,\ol{p}) \] 
 for an interior point $z_\black \in S$ and a boundary point $z_\white$,  is equal to $-1$.
\end{lemma}

A tropical picture of the desired curve is shown in Figure \ref{fig:anti-diagc}.

\begin{figure}[ht]\begin{center} \scalebox{.5}{
\begingroup%
  \makeatletter%
  \providecommand\color[2][]{%
    \errmessage{(Inkscape) Color is used for the text in Inkscape, but the package 'color.sty' is not loaded}%
    \renewcommand\color[2][]{}%
  }%
  \providecommand\transparent[1]{%
    \errmessage{(Inkscape) Transparency is used (non-zero) for the text in Inkscape, but the package 'transparent.sty' is not loaded}%
    \renewcommand\transparent[1]{}%
  }%
  \providecommand\rotatebox[2]{#2}%
  \newcommand*\fsize{\dimexpr\f@size pt\relax}%
  \newcommand*\lineheight[1]{\fontsize{\fsize}{#1\fsize}\selectfont}%
  \ifx\svgwidth\undefined%
    \setlength{\unitlength}{160.35269309bp}%
    \ifx\svgscale\undefined%
      \relax%
    \else%
      \setlength{\unitlength}{\unitlength * \real{\svgscale}}%
    \fi%
  \else%
    \setlength{\unitlength}{\svgwidth}%
  \fi%
  \global\let\svgwidth\undefined%
  \global\let\svgscale\undefined%
  \makeatother%
  \begin{picture}(1,1.00874889)%
    \lineheight{1}%
    \setlength\tabcolsep{0pt}%
    \put(0,0){\includegraphics[width=\unitlength,page=1]{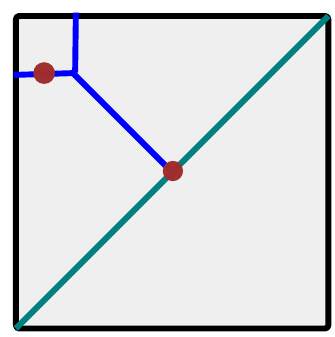}}%
    \put(0.5353573,0.42343767){\color[rgb]{0,0,0}\transparent{0.53157902}\makebox(0,0)[lt]{\lineheight{1.25}\smash{\begin{tabular}[t]{l}$(p,\ol{p})$\end{tabular}}}}%
    \put(0.08671281,0.69205876){\color[rgb]{0,0,0}\transparent{0.53157902}\makebox(0,0)[lt]{\lineheight{1.25}\smash{\begin{tabular}[t]{l}$(p_1,p_2)$\end{tabular}}}}%
  \end{picture}%
\endgroup%
}
\end{center} \caption{Curves with boundary in the  anti-diagonal in $\P^1 \times \P^1$}  \label{fig:anti-diagc}
\end{figure}

\begin{proof}   We apply the operation of ``removing a diagonal seam'' discussed in 
Wehrheim-Woodward \cite[Section 5.4]{orient} to reduce the count to one of holomorphic spheres.
Holomorphic disks in $X$ with boundary in $L$  correspond to holomorphic maps from $\P^1$ to $\P^1$ in the following sense:  Given a holomorphic disk $u: C \to X$ with components $(u_-,u_+)$ one obtains a holomorphic sphere 
\[ C' = \ol{C} \cup_{\partial C} C \] 
by gluing together the maps $u_-(\ol{z}), u_+(z), z \in C$ along the boundary.  This correspondence is an orientation-preserving diffeomorphism between the moduli space of disks and spheres by \cite[Proposition 5.4.6]{orient}.   The constraints on the disk translate to constraints on the sphere 
 \[ v(\ol{z}_\black) = p_1, \quad v(z_\black) = p_2, \quad v(z_\white) = p \] 
Complex conjugation on $\P^1$ is orientation reversing, and so 
maps the dual class $[p_1]$ to $p_1$ to minus the dual class  $- [\ol{p_1}]$ of $\ol{p_1}$ in $H^2(\P^1)$.

We claim that the count of holomorphic spheres $v$ with these three point constraints is equal to $-1$.  Indeed, each such map corresponds to an automorphism $\varphi: \P^1 \to \P^1$ with three given point constraints, and any such automorphism is uniquely specified by three such point constraints.  
Since one of the constraints has the opposite of the standard orientation, the orientation sign is negative.  
 \end{proof}
\end{example} 

\begin{example} \label{ex:lines} Consider the problem of counting affine linear maps from the affine line to itself. 

\begin{lemma}  \label{lem:lines} 
The number  $m(L,\ul{Y})$ of holomorphic maps from $S = \C$ to $\XX = \C$ passing through two fixed generic points in $\XX$ is equal to one.  Similarly, the number of holomorphic maps from the half-space $S = \H$ to $\XX = \C$ with boundary in $\R$ passing through
 points $p_1 \in \LL$ and asymptotic to a fixed Reeb chord at infinity is equal to $- 1$. 
\end{lemma}

\begin{proof}  The first statement of the Lemma is equivalent to the existence of a unique line between any two generic points.  A holomorphic map $u: \C \to \XX$ of the form 
\[ u(z) = az + b, a,b \in \C, a \neq 0  \] 
has finite Hofer energy.  Given $p_1 \neq p_2 \in \XX$ and $z_1 \neq z_2 \in \C$, there are unique values of $a,b$ so that $u(z_k) = p_k, k = 1,2$. The relative Gromov-Witten invariant in this case is one for the standard complex structure, and the answer is independent of the choice of almost complex structure by Lemma \ref{lem:elem}.  

In the real case, maps to $\C \times \C$ with boundary in the  diagonal
glue to holomorphic degree one maps $\C \to \C$
of the form $\ti{u}(z) = az + b$, by taking the first resp. second component to be the restriction of $\ti{u}$ to $\Im(z) \ge 0$ resp. its restriction to the complement.  
 \end{proof}
  Note, for later use, that the first resp. second component of $u$ has a zero away from the the boundary $\R \subset \C$ if $\Im(-b/a)$ is positive resp. negative.
\end{example} 

\begin{remark} \label{rem:choose} 
We choose the relative spin structure for the moduli space of relative broken maps so that the count has a positive sign.
 \end{remark}

\subsection{Disks in the blow-up of the product of projective lines}
\label{sec:model2}

We make three explicit computations of numbers of holomorphic disks in the blow-up of the product of projective lines.

\begin{example} \label{ex:blowup} We count disks in a blow-up. 
Let $X = \Bl(\P^1 \times \P^1)$ be the blow-up of the product
of projective lines and $L \cong S^2$ the inverse image of the anti-diagonal in Lemma 
\ref{lem:p1p1} under the blow-down map. Let $E \cong \P^1 \subset X$ denote the exceptional locus.

\begin{lemma}\label{lem:blowup} Let $L \subset X$ be as above.  The number  $m(L,\ul{Y})$ of holomorphic disks in $X$ with boundary in $L$ , intersecting $E$ with multiplicity one and passing through a generic point in $X$ and in $L$ is equal to minus one. 
\end{lemma}

\begin{proof} We reduce to the case of disks in the product of projective lines
as follows. Let $\bGamma$ be the type of holomorphic disk in $X = \Bl(\P^1 \times \P^1 )$ with boundary in $L$  meeting the 
exceptional curve $E$ once transversally.  Under the blow-down map $\pi: X \to \P^1 \times \P^1$, these are identified
with disks $\ol{u}$ in $\P^1 \times \P^1$ meeting the anti-diagonal and passing through the blow-up point $\pi(E)$, by removal of singularities for pseudoholomorphic disks; see \cite[Theorem 6.1]{vwx}. By Lemma
\ref{lem:anti-diag} these have count $-1$.
\end{proof}
\end{example}

\begin{example}  \label{ex:triv} 
Consider disks in the blow-up of the product of projective lines with boundary in the  tropical Lagrangian whose graph has a single trivalent vertex.  
Let $X$ be the blow-up of $\P^1 \times \P^1$ at a point and $L$ the embedded
Lagrangian sphere in $X$ considered in Examples \ref{ex:hind}, \ref{ex:hind2}.

\begin{lemma} \label{lem:hind3} 
The count of Maslov-index-two disks with boundary in the  tropical Lagrangian $L \subset X = \Bl(\P^1 \times \P^1)$ whose graph $\Gamma$ has a single trivalent vertex
$v \in \Ver(\Gamma)$ and three edges $e_1,e_2,e_3 \in \Edge(\Gamma)$ in the blow-up $X$ of $\P^1 \times \P^1$ at a torus-fixed-point is equal to $-1$.
\end{lemma} 

\begin{proof}   The previous Lemma \ref{lem:blowup} covered the case of the inverse image of the diagonal.
Example \ref{ex:hind2} described an  the isotopy of the given tropical Lagrangian to the inverse image of the
diagonal. The statement of the Lemma now follows by invariance of the disk counts under Hamiltonian isotopy.
Compare with Hicks \cite[Figure 2]{hicks:realizability}. 
\end{proof}

\begin{lemma} \label{lem:maslovfour} The count $m(L,\ul{Y})$ of Maslov-index-four disks bounding
the Lagrangian $L$ in Examples \ref{ex:hind}, \ref{ex:hind2} and passing through three generic points $p_1,p_2,p_3 \in L$ is equal to $1$.
 \end{lemma}

A tropical picture of the disks in question is shown on the right in Figure \ref{fig:b2p2tri} below.
The left shows that tropical picture for the lift of the diagonal.   Note that the intersection number of the disks with the anticanonical divisor is three, for the disks on the right; this reflects
the fact that the Maslov index of a broken disk $u = (u_v)$ with boundary in $L$  is {\em not} in general equal to the number $\# \Gamma \cap \partial B^\dual$ of intersections of the tropical graph $\Gamma$ with the boundary $\partial B^\dual$, but rather has a correction from trivalent vertices
$v \in \Ver(\Gamma), |v| = 3$ explained further in Remark \ref{rem:arearem}.

\begin{figure}[ht]\begin{center} \scalebox{.3}{\includegraphics{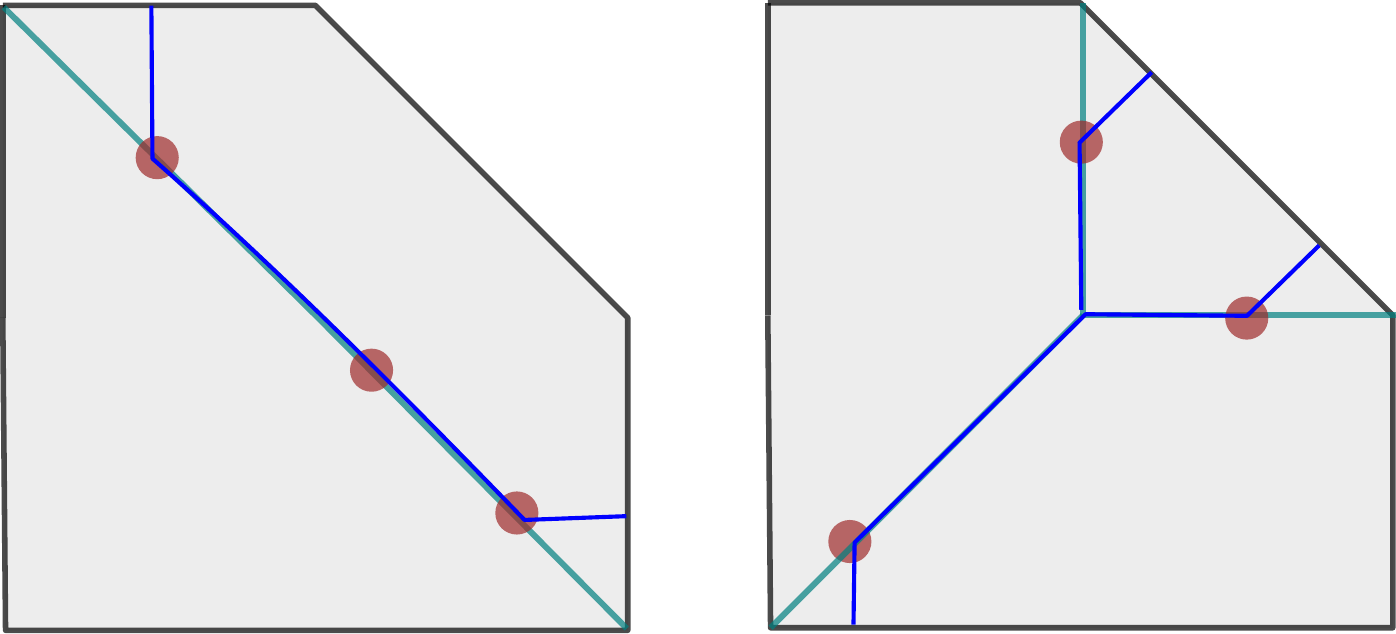}}
\end{center} \caption{Cartoon diagrams of Maslov-index-four disks bounding tropical Lagrangians}  \label{fig:b2p2tri}
\end{figure}

 \begin{proof}
 An isotopy of the given tropical Lagrangian to the inverse image of the diagonal was described in Example \ref{ex:hind2}.
 We first consider the possibilities for the intersection number with the exceptional divisor.  Let  $\ul{Y} = ( p_1,p_2,p_3)$ with $p_1,p_2,p_3 \in L$ generic. The moduli space $\M(L,\ul{Y})$ consists of rigi Maslov-index-four disks $u$  with boundary in $L$  with constrains $\ul{Y}$, and breaks up into a disjoint union   of subsets $\M(L,\ul{Y})_d$ consisting of maps $u$ whose intersection multiplicity $[u] . [E]$ with the 
 exceptional divisor $E \subset X$ is equal to $d \in \Z_{\ge 0}$. 
 The manifold $X$ is monotone and we may assume that the monotonicity constant is one, 
 so that the symplectic class $[\omega]$ is represented by the boundary divisor $Y = \Phinv(\partial \Phi(X))$.
 The immersed submanifold $Y $ has 
 the exceptional curve $E$ as an irreducible component. 
  Since the area  $A(u)$ of $u$ is two, the intersection number $d$
 must equal $0,1$ or $2$.  Since non-transverse intersections with $E$ are codimension two, we may assume after choosing a generic domain-dependent almost complex structure that the $d$ intersection points are disjoint; see \cite[Lemma 5.27]{vwx}. 
 Composition with the projection $\pi: \Bl(\P^1 \times \P^1) \to \P^1 \times \P^1$ gives a holomorphic disk $u' = \pi \circ u$ in $\P^1 \times \P^1$ passing through the blow-up point $\pi(E)$ $d$ times,   with Maslov index $4$, $6$, or $8$.

The projected Lagrangian is the fixed point set of an anti-symplectic involution and the usual correspondence between disks and spheres in this case may be used to count the disks.    Since the number of such disk $u$ is independent of the choice of almost complex structure, we may assume that the almost complex structure is the standard one.  In particular, the almost complex structure is anti-invariant under the standard anti-symplectic involution $(z_1,z_2) \mapsto (\ol{z}_2 , \ol{z}_1)$.  By the doubling construction,  any such disk $u$ corresponds to a holomorphic sphere $\ti{u}$ in 
 $\P^1 \times \P^1$ passing through the original points $p_1,p_2,p_3$.   In addition, 
 the spheres $\ti{u}$ obtained by doubling pass through the  blow-up point $\pi(E)$  $2d$ times.     The Chern number of such a sphere $\ti{u}$ is $2,3$ or $4$, and the bidegree of $\ti{u}$ is 
  $(e,e)$ with either $ e = 1$ or $e = 2$. However, in the $e=2$ case, the 
   disk $u$ meets the anticanonical divisor in $X$ {\em only} at two points in $E$.  Its projection $u'$ meets the anti-canonical divisor in $\pi(Y) \subset \P^1 \times \P^1$
  only at $(0,0)$ and $(\infty,\infty)$.   Such a curve $u'$ is necessarily a double cover
  of a curve of degree $(1,1)$, and the corresponding disk $u$ in $X$ is therefore a double cover of a disk of Maslov-index-two.  Such curves $u$ are disjoint from a choice of three generic points on the Lagrangian; the same holds after perturbation.  It follows that the corresponding count vanishes. 
 \end{proof}
\end{example}

\subsection{Disks ending on cylinders}

We now prove the various cases of Theorem \ref{thm:mult} in Definition \ref{def:mv}.
The general strategy is to deduce these relative counts from the absolute counts in the previous two sections.   The first count is that of holomorphic disks whose tropical graphs meet the graph of the Lagrangian perpendicularly:  

\begin{lemma}  \label{lem:mvhitsproof} Let $\XX_{P(v)}$ be a cylindrical end manifold  obtained from some compact toric orbifold 
 $\ol{\XX}_{P(v)}$ be removing the inverse image of the boundary of the moment polytope.    Let $\LL_{P(v)}  \subset \XX_{P(v)}$ be a Lagrangian realization of a tropical graph $\Lambda$ with a single edge 
$\Edge(\Lambda) = \{ \eps \}$ in $P(v)$.   The count of punctured disks $u_v: C_v \to \XX_{P(v)}$ bounding $\LL_{P(v)}$ 
with a single cylindrical end and no strip-like ends with a single point constraint $Y_e = (p_e \in \LL_{P(v)}) , \Edge_\white(\Gamma) = \{e \}$
  is equal to $m(v)  = -1$, as in Definition \ref{def:mv}  \eqref{mv:cylhit}. 
  \end{lemma}

  \begin{proof}  We first consider the case that the target is a quotient of the product of projective lines $\ol{\XX}_{P(v)} \cong (\P^1 \times \P^1)/\Z_2$ and consider the arbitrary toric case afterwards.   We decompose the moduli space of Maslov index four disks with boundary in the  anti-diagonal with a point constraint as follows.  
  Let 
  \[ L = \{ (z,\ol{z}) \} \cong \P^1  \subset X:= (\P^1)^2 \]
  denote the Lagrangian anti-diagonal. Let 
  \[ \pt_X  = (0,\infty) \in \P^1 \times \P^1. \] 
  The point $\pt_X$ is a torus-fixed-point    not on the Lagrangian $L$.  The count of holomorphic Maslov-index-four disks with boundary in $L$   passing through $\pt_X$ equal to $-1$ by Lemma \ref{lem:anti-diag}.  
Consider a decomposition $\cP = \{ P \} $ of the moment polytope of $\ol{\XX}_{P(v)}$,  consider a broken map  $u = (u_0,\ldots, u_k)$ with pieces mapping to $\XX_{P_0}, \ldots, \XX_{P_k}$ of the following form:  The piece $\ol{\XX}_{P_0}$ contains the constraint $Y_{e}$ has polytope a non-regular hexagon, and is connected by a sequence of pieces 
whose polytopes are rectangles to a piece $\ol{\XX}_{P_k}$ containing the Lagrangian $\LL_{P_k} \cong \R \times S^1$.    Each piece  not containing the constraint or Lagrangian is a $\Z_2$-quotient
of $\P^1 \times \P^1$
\[ \ol{\XX}_{P_i} \cong ( \P^1 \times \P^1)/\Z_2, \quad i = 1,\ldots, k-1 . \] 
Let $P_1$ denote the polytope containing the image of $\pt_X$.

Broken maps in the degeneration are classified as follows.
The only rigid broken maps $u = (u_v)$ are those meeting 
only $\XX_{P_0}$ and $\XX_{P_k}$ consisting of two levels given by disks $u_{v_0}$ in $\XX_{P_0}$
and spheres $u_{v_k}$ in $\XX_{P_1}$. Indeed,  $L$ is the fixed point set of an anti-symplectic involution, and the decomposition into pieces $\XX_P$ is compatible with the involution.   We may view $u$ as half of a broken map $\ti{u}$ of Chern number four, with tropical graph $\ti{\bGamma}$ obtained by ``doubling'' $\bGamma$.     All other univalent vertices must lie in the intersections of $\ti{\bGamma}$
with the boundary of the moment polytope $\partial \Phi(X)$.  It follows that these vertices correspond to intersections of $\ti{u}$ with the anti-canonical divisor.  Since the Chern number of $\ti{u}$ is $2$,
the only intersections with the boundary are $(0,\infty) , (\infty,0) \in \P^1 \times \P^1$.   Hence
the only possible univalent vertex of $\bGamma$ is at the image of $(0,\infty)$.
Thus there are no trivalent vertices $v \in \Ver(\bGamma)$ in the type $\bGamma$ of any rigid map.
There must be two univalent vertices $v_0, v_k \in \Ver(\bGamma)$ mapping to $P_1^\dual$ and the dual $P_k^\dual$.    The contribution $m(v_k)$ of the sphere $u_{v_k}$ in $\XX_{P_1}$ is equal to one, since $u_{v_1}$ is transversally cut out and orientation signs for regular holomorphic spheres are always positive.  It follows that the number of holomorphic disks in $\XX_{P_0}$ with boundary in $L$  with a point constraint on the boundary and with a single
cylindrical end is equal to 
\[ m(v_0) = -1 . \] 
Now that we have proved the formula for the multiplicity for the case that the target is 
$(\P^1 \times \P^1)/\Z_2$, the formula for arbitrary toric targets follows from Theorem \ref{thm:mult} case Definition \ref{def:mv}  \eqref{mv:spheres} and the
decomposition formula in Theorem \ref{thm:mult}, by taking a decomposition in which 
the piece $\XX_{P_0}$ containing the intersection with the Lagrangian has polytope $P_0$ such a rectangle.
\end{proof}

\begin{figure}[ht]\begin{center} \scalebox{.5}{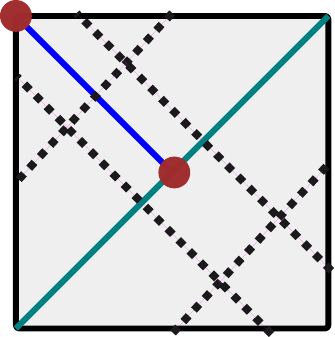}
\end{center} \caption{Curves with boundary in the  anti-diagonal in $\P^1 \times \P^1$}  \label{fig:anti-diag}
\end{figure}

\begin{proof}[Proof of  Theorem \ref{thm:mult} case \eqref{mv:cylhit}]
The case of primitive directions was proved in Lemma \ref{lem:mvhitsproof}.   To prove the claim for arbitrary directions, we consider the case of holomorphic disks $u$ in $X = \Bl(\P^1 \times \P^1)$ with a tangency at order $d \in \Z_{\ge 0}$ at a point $z \in C$ mapping to the exceptional divisor $E$.  These are naturally in bijection with disks $u'$ with tangency of order $1$  as follows: 
Given $u'$, a map $u$ is obtained composing with the $d$-fold covering $z \mapsto z^d$.  The proof that this gives a bijection is similar to the arguments in the proof of Lemma  \ref{lem:mvhitsproof}.  Namely, any such disk $u$ corresponds to a sphere $\ti{u}$ in  $\P^1 \times \P^1$ passing through $(0,0)$ and $(\infty,\infty)$, and vanishing to order $d$ at those points.  Such a sphere $\ti{u}$ is locally, hence globally, a $d$-fold ramified cover of its image.   The relative spin structure for such a cover $\ti{u}$ differs from the trivial relative spin structure by the $d$-th power of the corresponding map for the 
corresponding embedding $u$.  The relative spin structure on 
$(\partial \ti{u})^* TL$ is a $d$-fold twist of the 
relative spin structure for $(\partial u)^* TL$. 
By \cite[Proposition 8.1.16]{fooo:part2}  the orientation of the moduli space at $\ti{u}$ is $(-1)^d$ times the standard orientation of a point.  
\end{proof}

\begin{example}  \label{ex:b7p2} 
We give an example where non-primitive directions occur in Theorem \ref{thm:mult} case Definition \ref{def:mv}  \eqref{mv:cylhit}, that is, the case $d > 1$ occurs.  Figure \ref{fig:b7p2hit} shows an almost toric diagram for the del Pezzo of degree two, as in Vianna \cite{vianna:dp}.

\begin{figure}[ht]\begin{center} \scalebox{.5}{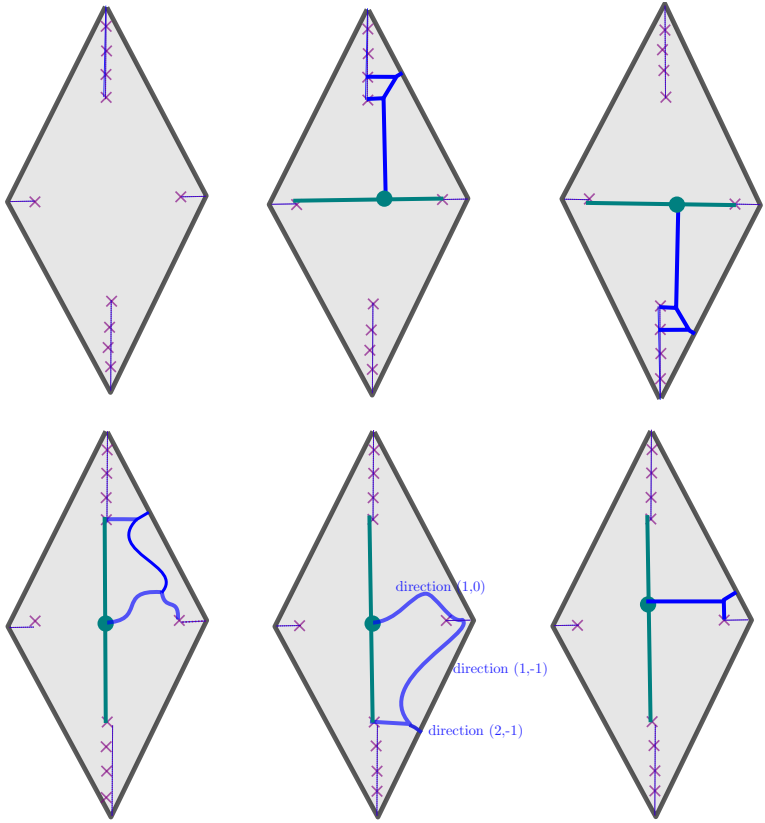}
\end{center} \caption{Cartoon diagrams for broken disks in the del Pezzo of degree two}  \label{fig:b7p2hit}
\end{figure}

In our previous paper 
\cite{vw:at} we computed the potential of the monotone torus $L \subset X$ for this diagram and found that the potential $W_L: \C^{\times,2} \to \C$ had critical values 
\[ \Crit(W_L) = \{ -12, 52 . \} \]
Consider the two Lagrangian spheres $L_1,L_2$ whose possible tropical graphs $\Lambda_1,\Lambda_2$ consist of a long vertical or horizontal segment $\eps_1, \eps_2$.
\begin{itemize}
\item   The tropical graphs $\Gamma_1$ corresponding to Maslov-index-two disks bounding $L_1$ have initial direction $(0, \pm 1)$ and interact with $2$ out of the $4$ possible focus-focus singularities, for a total of 
\[ W_{L_1} = 2(-1)6 = -12 . \]
\item
For the tropical Lagrangians $L_2$ with vertical graph $\Lambda_2$ the possible initial directions
of the tropical graphs $\Gamma_2', \Gamma_2''$ corresponding to disks bounding $L_2$ are $(\pm 1,0)$
and $(\pm 2,0)$.  The tropical graphs of holomorphic disks of the latter type $\Gamma_2''$ have a single vertex with multiplicity $2$, for a total contribution of $2(2) = 4$.  The first type of tropical graphs $\Gamma_2'$ have two vertices $v_1,v_2$ of multiplicity $m(v_1) = m(v_2) = 1$, and interact with one of the eight focus-focus singularities
along the vertical axis, for a total contribution of $2(-8) = -16$.  Thus the potential of these tropical Lagrangians is 
\[ W_{L_2} = 4 - 16 = -12 . \]
\end{itemize}
Thus these Lagrangians $L_1,L_2$ have the same potential, that is,
lie in the same Fukaya eigencategory $\Fuk_{-12}(X)$ of $X$.   The same fact can be seen from the simpler observation that the Lagrangians $L_1,L_2$ have a single intersection point (the unique focus-focus point in the intersection of their preimages).  
It follows that the Floer boundary operator on $CF(L_1,L_2)$ has vanishing square, since the operator has odd degree and $CF(L_1,L_2)$ has rank one.  On the other hand, 
Oh's results in \cite{oh:floer} imply that the square of the Floer differential is the difference in potentials $W_{L_2} - W_{L_1}$.
\end{example}

\subsection{The holomorphic pant}

Next we investigate holomorphic curves with one cylindrical and one strip-like end, 
that is, the holomorphic pant.  

\begin{theorem}  Suppose $\XX_{P(v)}$ is a component of $\XX$ not containing a focus-focus singularity $x \in X^{\foc}$, so that the compactification $\ol{\XX}_{P(v)}$ is a compact toric orbifold with polytope the convex hull of $\{ (\pm 1,0), (0, \pm 1) \}$.   Let $\LL_{P(v)}  \subset \XX_{P(v)}$ be a Lagrangian realization of a tropical graph $\Lambda$ with a singe edge with direction $(1,1)$.    For a particular almost complex structure, The count of punctured holomorphic disks in $\XX_{P(v)}$ bounding 
  $\LL_{P(v)}$ with  a single strip-like end $e_\white \in \Edge_\white(\bGamma)$ asymptotic to a fixed Reeb chord 
 $\gamma_\white$ with direction $\cT(e_\white)$ and single cylindrical-end $e_\black \in \Edge_\black(\bGamma)$ at the interior puncture 
asymptotic to a fixed Reeb orbit $\gamma_\black$ at the puncture with direction $\cT(e_\black)$
  is equal to 
\[ m(v)  = 
(-1)^{ | \det( \cT(e_\black) \cT(e_\white)) | } | \det( \cT(e_\black) \cT(e_\white)) |  \] 
as in
  Theorem \ref{thm:mult} case Definition \ref{def:mv} \eqref{mv:pant}.  
   \label{thm:pantone}
\end{theorem}

We first show the following special case: 

\begin{lemma}  \label{lem:pantone}
Let $ \LL_{P(v)} \subset \XX_{P(v)}:= (\C^\times)^2$ denote the antidiagonal Lagrangian.  The count of punctured holomorphic disks in $\XX_{P(v)}$ bounding 
  $\LL_{P(v)}$ with  strip-like end with direction $(-1,-1)$
  and one cylindrical-end with direction either $(1,0)$ or $(0,1)$, and asymptotic to a fixed Reeb orbit
  \[ Y_{e_1} = \{ \gamma_{e_1} \}, \{  e_1 \}  =\Edge_\white(\Gamma) \] 
  at the interior puncture 
and point
\[ Y_{e_2}  \in \LL_{P(v)} \] 
is equal to 
\[ m(v)  = - 1 .\] 
  \end{lemma}

\begin{proof}
The model problem of counting real lines was considered in Lemma \ref{lem:lines}.  The count of such lines was shown to be equal to $-1/2$ for a fixed Reeb orbit at infinity and a point constraint on the boundary.   From the count of lines we obtain a count of relative maps as follows. Consider the symplectic cut of $X = \C^2$ along a line perpendicular to the direction of the moment image $\Lambda$ of the Lagrangian and two lines parallel to $\Lambda$  as in Figure \ref{fig:antidiag2}. 
The only rigid tropical graph $\Gamma$ is the one shown in Figure \ref{fig:antidiag2}, with a single bivalent vertex $v \in \Ver(\Gamma)$.  Indeed, any additional vertex  vertex $v' \in\Ver(\Gamma)$ would have deformable position, and so the graph $\Gamma$ and broken map $u$ would not be rigid in the sense of Definition \ref{def:rigid}, a contradiction.  Since the count of holomorphic spheres in $\XX_Q$ is equal to one, there are two disks for each sphere.  The orientation at one of the constraints is reversed, so the count of curves in $\XX_P$ must be $m(v) =- 1/2 $, as in Theorem \ref{thm:mult} case Definition \ref{def:mv}  \eqref{mv:pant}.
The answer is independent of the choice of almost complex structure by Lemma \ref{lem:elem}.
\end{proof}

\begin{figure}[ht]\begin{center} \scalebox{.9}{\includegraphics{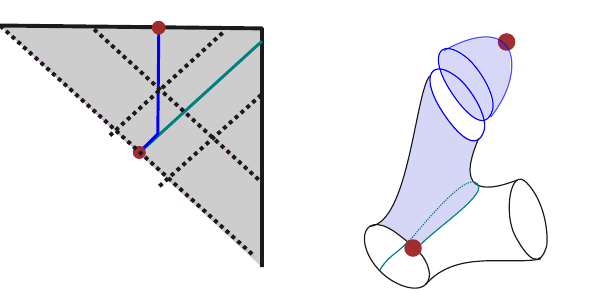}}
\end{center} \caption{The holomorphic pant and a capping disk}  \label{fig:antidiag2}
\end{figure}

\begin{remark}  \label{rem:explicit}
A holomorphic disk with boundary in the  Lagrangian pair of pants
as in Theorem \ref{thm:mult} case Definition \ref{def:mv}  \eqref{mv:pant} can be described explicitly as follows.   Consider a holomorphic pair of pants 
$u: C = \P^1 \to \P^1 \times \P^1 $.  Define  
\[ u(z) = (z, 1-z) \] 
of so that $u$ has bidegree $(1,1)$.   
 Let $u^{-1}(\Delta)$ denote the inverse image of the anti-diagonal 
 \[ L = \{ z_2 =\ol{z}_1 \} .\]
We have 
\[ u^{-1}(L) = \{ z = 1  - \ol{z} \} .\] 
This inverse image is a curve separating $\P^1 - \{ 0,1 , \infty \}$ into 
two disks that are the images of maps $u_-,u_+$, as in Figure \ref{fig:antidiag2}.
\end{remark}

To compute the multiplicities of graphs with arbitrary directions at a trivalent vertex, we consider a finite cover operation as in Mikhalkin \cite[Lemma 8.21]{mikhalkin}.  Let $X = (\C^\times)^2$ with the reversed almost complex structure on the first factor, viewed as the interior of the toric orbifold whose moment polytope $\Phi(X)$ has
vertices $(\pm 1,0), (0,\pm 1)$.  Let 
\[ L = \{ (z,\ol{z}), z \in \C^\times \} \subset X \] 
be the diagonal Lagrangian.  Note that the diagonal is the fixed point set of the antisymplectic involution 
\[ \sigma: (z_1,z_2) \mapsto (\ol{z}_2,\ol{z}_1) . \] 
It has two cylindrical ends corresponding to the punctures at $z = 0$ and $z = \infty$.  For integers $n_-,n_+$ consider the 
holomorphic map 
\[ \tau_{n_-,n_+}: X \to X, \quad (w_-,w_+) \mapsto ( w_-^{n_-} w_+^{n_+}, 
w_-^{n_+} w_+^{n_-} )  \] 
is $\sigma$-equivariant, and preserves the fixed point set $L$.   
Define 
\begin{equation} \label{eq:matrixT} T_{n_-,n_+} = \left[ \begin{array}{cc} n_- & n_+ \\ 
n_+ & n_- \end{array}  \right] .\end{equation}
The following is immediate from the definition:

\begin{lemma} \label{lem:coverdir}
Let $u : C \to X$ be a holomorphic disk with boundary in the  diagonal $L$ intersecting
with a single cylindrical end $e$
with direction $\cT(e) \in \Z^2 .$
Then $\ti{u}$ has a cylindrical end $\ti{e}$ with direction  $\cT(\ti{e}) =  T_{n_-,n_+} \cT(e) $.
\end{lemma}

We claim that all curves with such an end arise this way:

\begin{lemma} \label{lem:ramified} 
\begin{enumerate}
    \item
Any punctured disk $\ti{u}: S \to X $ with a unique cylindrical end
with direction $(n_-,n_+) $ and with boundary in the  diagonal is the composition of 
$\tau_{n_-,n_+}$ with some disk $u:S \to X$.  
\item
The number of such lifts $u$ for a given map $\ti{u}$ is  equal to 
$n_+ + n_-$.
\item The map $u$ is regular if and only if $\ti{u}$ is regular.  
\end{enumerate}
\end{lemma}

\begin{proof}     The two  claims follow from standard lifting properties of holomorphic covering maps, with $u$ being a lift of $\ti{u}$ under $\tau_{n_-,n_+}$.  Namely, the group $\Z_{n_- + n_+} $ of $n_- + n_+$-roots of unity acts naturally on $X$ by 
\[ \gamma(z_-,z_+) = (\gamma z_-, \ol{\gamma} z_+) \]
and preserves the fixed point set $L$.  Any two lifts 
differ by an action of some $\gamma \in \Z_{n_- + n_+}$ acting diagonally, since the other elements 
of the group of deck transformations of $\tau_{n_-,n_+}$
do not preserve the boundary condition.   For any pair of points 
in $L$, there are $(n_- + n_+)^2$ lifts under $\tau_{n_-,n_+}$,
and two lifts have the same projection if they differ by an element $\gamma \in \Z_{n_- + n_+}$.  Since there is a unique lift for each pair of given constraints, the number of 
curves $\ti{u}$ with the given contraints is $ n_- + n_+$. (This is the "real version" of the proof of  Mikhalkin's \cite[Lemma 8.21]{mikhalkin}.)

The claim on regularity follows from an identification of the linearized operators $D_u$ and $D_{\ti{u}}$ (see \cite[(6.39)]{vw:trop}) give as follows:
The identification $\ti{u}^* TX \to u^* TX$ given by the differential of $\tau_{n_-,n_+}$
induces an isomorphism of real boundary conditions and Cauchy-Riemann operators.
\end{proof}

\begin{figure}[ht]\begin{center} \scalebox{.5}{\includegraphics{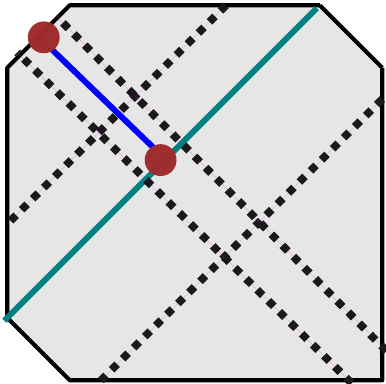}}
\end{center} \caption{Decomposing into rectangles}  \label{fig:antidiagsplit}
\end{figure}

\begin{proof}[Proof of Theorem \ref{thm:pantone}]
We use Lemma \ref{lem:ramified} to reduce to the case of directions 
$(1,0), (1,0)$ covered in  Theorem \ref{thm:pantone}.  The number of lifts $\ti{u}$ of 
a given disk $u$ is equal to the number of points in each fiber of $\tau_{n_-,n_+}$,
which is the determinant of the matrix in \eqref{eq:matrixT}.   On the other hand, the count of disks in Theorem \ref{lem:pantone} is, for an appropriate choice of relative spin structure as in Remark \ref{rem:rel}
\[ m(v_0) = -1/2 .\]
The boundary conditions $\ti{u}^* TL$ and $u^* TL$ are naturally identified, 
but the relative spin structures are defined by  capping procedure, 
detailed in \cite{orient}.   The two relative spin structures on the extensions of 
$\ti{u}^* TL$ and $u^* TL$  differ in relative spin structure by a sign 
\[ (-1)^{n_+ + n_- - 1} = (-1)^{n_+ + n_- -1}
= (-1)^{n_+^2 - n_-^2 - 1} = (-1)^{\left| \det (T_{n_-,n_+}) \right|  - 1}
.\] 
(The number of half turns of $\ti{u}$ at the open puncture is $n_+ + n_-$, while the number of half-turns in $u$ is $1$.)   Since the number of such disks, $1$, is half the determinant $2$, the statement of the Theorem now follows from the identities
\begin{eqnarray*} m(v) &=& m(v_0) \# \{ \ti{u} \mapsto u \} \\
&=& -\hh (-1)^{ 
| n_+ + n_- - 1 |} \left| \det (T_{n_-,n_+}) \right|   \\
&=& -\hh 
(-1)^{ |\det(\mu_\black r(\mu_\black))|  }
|\det(\mu_\black r(\mu_\black))|  . \end{eqnarray*}
\end{proof}

\subsection{Curves with boundary in the  pair of pants}

Next, we consider curves with boundary in the  Lagrangian pair of pants
from  Example \ref{ex:locmod}  \eqref{item:lpants}.   We use the isotopy 
of the previous section to compute the number of holomorphic disks $u: C \to X$ with boundary in the  pair of pants $L$  of various topological types.  

\begin{lemma} \label{lem:singleend} Let $X = (\C^\times)^2$.  For a judicious choice of relative spin structure as in Definition \ref{def:index}, the number  $m(L,\ul{Y})$
of holomorphic disks $u: C \to X$ with boundary in  Lagrangian pair of pants 
$L \cong \R^2 - \{ 0, 1 \} $ from  Example \ref{ex:locmod}  \eqref{item:lpants} so that $C$ has a single strip-like end and no cylindrical ends is equal to $1$, as in Definition \ref{def:mv} \eqref{mv:pantseam}
\end{lemma}

\begin{proof} The number of disks with boundary in the  compactification of the Lagrangian pair of pants
was computed in Lemma \ref{lem:hind3}:  Let $X_0 = \Bl(\P^1 \times \P^1)$ be the blow-up at  point  $  \P^1 \times \P^1$.   Consider the Lagrangian $L_0$ whose image $\Phi(L_0)$ under the moment map $\Phi$ is a trivalent graph $\Lambda$ shown in Figure \ref{fig:pants1} centered at $0$ in the moment polytope $\Phi(X)$, so that $\Lambda$ intersects the image of the exceptional divisor $E$ twice.  The number of disks $u: C \to X_0$  bounding $L_0$ of Maslov index $I(u) = 2$
is equal  to $-1$, by Lemma \ref{lem:hind3}.

We deduce from the previous paragraph the number of holomorphic disks with boundary in the  Lagrangian pair of pants as in  the Lemma.   Consider the polyhedral decomposition  $\cP = \{ P \}$ shown in Figure \ref{fig:pants11}.  Let $u = (u_v)$ be a broken Maslov-index-two disk in the resulting broken symplectic manifold $\XX$
passing through the given point $p \in L$ as shown, with tropical graph $\Gamma$.

\vskip .1in \noindent {\em Claim:  There is no vertex of $v$
mapping to the polytope $P(v)$ containing the vertex of $\Lambda$; that is, the tropical graph $\Gamma$
maps to a single leg of $\Lambda$.} 
Indeed, since the intersection number of $[u]$ with the class of the divisor $[D]$ is $1$, 
$u_{v_i}:C_{v_i} \to \XX_{P(v_i)} $ in the pieces $\XX_{P(v_i)}$ corresponding to the other legs of the trivalent graph $\Lambda$ of $L$ in $u$.  Indeed, removing any such vertex $v_0$ would 
break $\Gamma$ into pieces $\Gamma_i$
the univalent vertices of $\Gamma_i$ other than $v_0$ represent other intersection point with the divisor $D$.  Thus each piece $\Gamma_i$ contributes at least one to the intersection number
with $[D]$, and the graph $\Gamma$ of the broken map $u$ must be as shown.  

\vskip .1in
By the already proved parts of Theorem \ref{thm:mult} (namely the two multiplicities appearing in Figure \ref{fig:antidiag2}) the count of holomorphic spheres $u_v : C_v \to \XX_{P(v)}$ in the lower left in Figure \ref{fig:pants1} is equal to $-1$. So the number of holomorphic disks 
 $u_{v'} : C_{v'}\to \XX_{P(v')}$ 
 with boundary in the  Lagrangian $\LL_{P(v')}$ with the given point constraint
is equal to $1$.  This justifies Theorem \ref{thm:mult} case Definition \ref{def:mv}  \eqref{mv:pant}.
\end{proof}

\begin{figure}[ht] \begin{center} \scalebox{.3}{\includegraphics{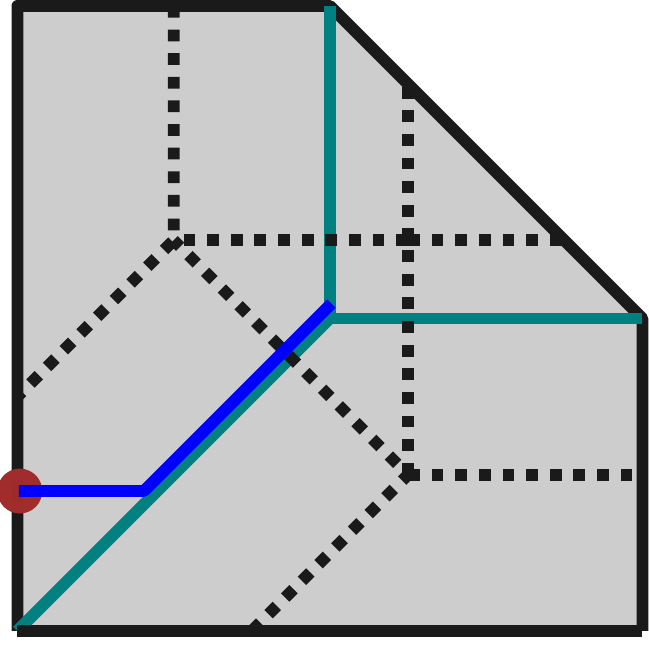}}
\end{center} \caption{Polyhedral decomposition computing the disks bounding a Lagrangian pair of pants} \label{fig:pants11} \end{figure}

\begin{remark}  For the standard complex structure, the
Lagrangian pair of pants in Lemma \ref{lem:singleend} above can be described explicitly as follows.   Consider the holomorphic pair of pants $H = \{ 1 + z_1 + z_2 = 0 \}$ in $\C^2$, and let $L \subset \R^4$ be its Lagrangian hyper-K\"ahler rotation.  The diagonal $\Delta = \{ (z,z), z \in \C \}$ is divided into two pieces by its intersection with $L$, 
as in Hicks \cite[Figure 2]{hicks:realizability}.  Each half of the holomorphic curve $\Delta$ is the image of a holomorphic disk, whose tropical graph $\Gamma$ is shown on the left in Figure 
\ref{fig:pants1}, in blue. Note that  the canonical trivialization of $u^* TL$ at the strip-like end, together with the obvious trivialization of $u^ TL$ along the boundary into tangential and normal parts, do not combine to a trivialization of the extended boundary value problem described in Definition \ref{def:index}.  Therefore, the definition of relative spin structure on the extension.    \end{remark}

\begin{figure}[ht] \begin{center} \scalebox{.3}{\includegraphics{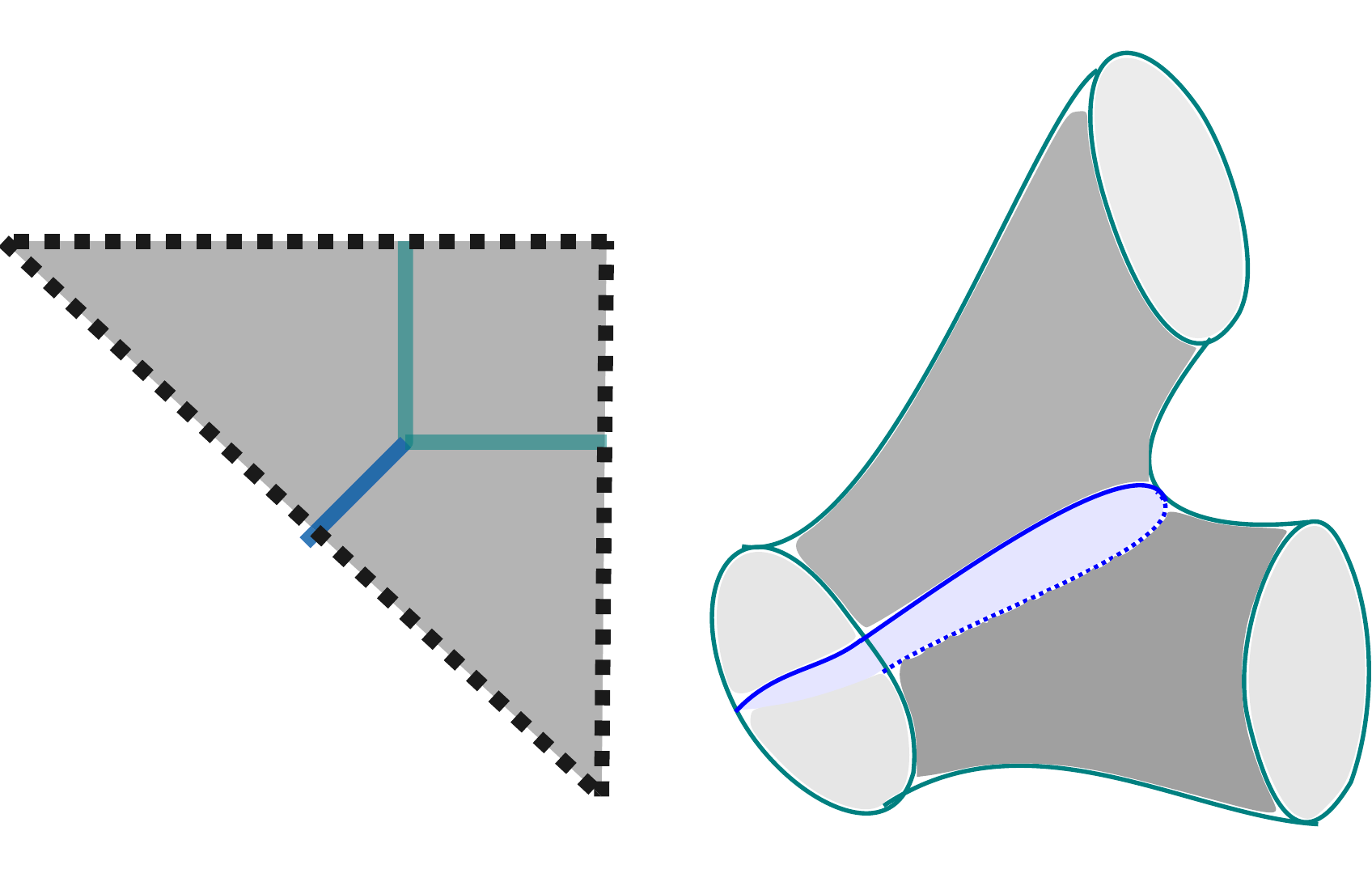}}
\end{center} \caption{A disk bounding a Lagrangian pair of pants} \label{fig:pants1} \end{figure}

Next, we compute the number of holomorphic disks with boundary in the  Lagrangian pair of pants with two
or three strip-like ends.

\begin{lemma} \label{lem:stripsinpants}
The number $m(v)$ of holomorphic disks in $\XX_{P(v)} = (\C^\times)^2 \subset
\Bl^3(\P^2)$ with boundary in the Lagrangian pair of pants $\LL_{P(v)} \cong \R^2 - \{ 0 , 1 \}$ with two or
three strip-like ends and a point constraint $p \in \LL_{P(v)}$,
is equal to $1$.
\end{lemma}

\begin{proof}   We recall the model problem in Lemma \ref{lem:maslovfour}:
For the trivalent graph shown in Figure \ref{pants2} the number of disks $u: C \to X$ with boundary in $L$ 
with three point constraints 
\[ u(z_1) = p_1, u(z_2) = p_2, u(z_3) = p_3, \quad z_1,z_2,z_3 \in \partial C \] 
is equal to $1$, independent of the choice of almost complex structure.   
We take a polyhedral decomposition as in Figure \ref{pants2}, into a central piece $P_0$ and pieces $P_i$ near the vertices.  The number $m(v_i)$ of lines in the pieces near the vertices
are equal to $1$, for a suitable extension of relative spin structures. For this choice, we must also have that 
\[ m(v_0) = 1 \] 
where $m(v_0)$  is the number of disks $u: C \to X^\circ$ with boundary in the  Lagrangian pair of pants $L$ in the complement $X^\circ$ of the toric boundary
in $\P^2$ with either three fixed Reeb chords $\gamma_{e_1},\gamma_{e_2}, \gamma_{e_3}$ at infinity, or two Reeb chords $\gamma_{e_1}, \gamma_{e_2}$ and one boundary constraint $p_3 \in L_0$, or with three boundary 
constraints $p_1,p_2,p_3 \in L_0$.  In the last case, there are  two possible orderings of the boundary constraints; we may assume that the ordering is compatible with the orientation around the boundary.  This justifies 
Theorem \ref{thm:mult} case Definition \ref{def:mv}  \eqref{mv:twoend} and \eqref{mv:threeend}.
\end{proof}

\begin{figure}[ht] \begin{center} \scalebox{.3}{\includegraphics{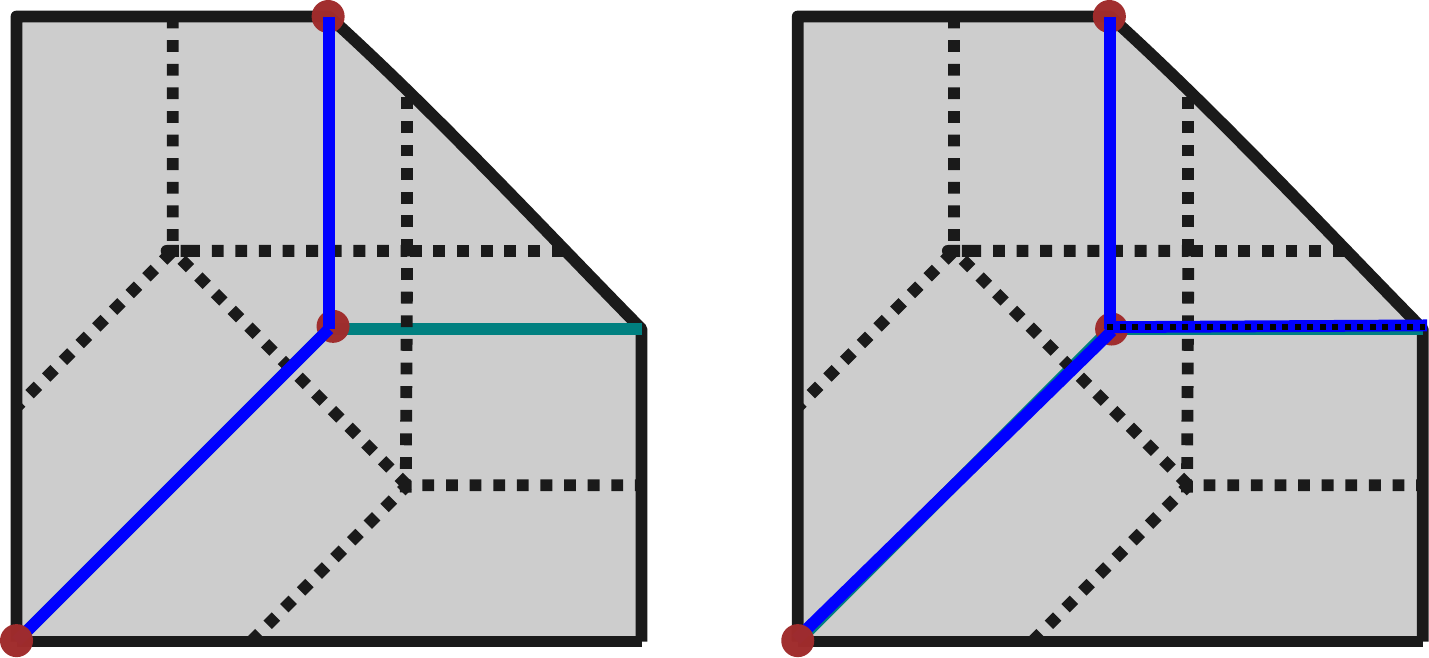}}
\end{center} \caption{Cartoon diagrams for disks bounding a Lagrangian sphere in $\Bl^2 \P^2$} \label{pants2} \end{figure}

\begin{remark}  \label{rem:arearem} \label{rem:index} We make some comments on the combinatorial formulas for area and index. The area 
\[ A(u) = \sum_{v \in \Ver(\bGamma)} A(u_v) \] 
of a broken map $u = (u_v)$ may be computed as a sum of the areas of its pieces, attached
to the vertices of $\Gamma$; this quantity is not related to the lattice length of the graph.
For example, the areas 
$ A(u_1),A(u_2) \in \R_{\ge 0}$
of the curves $u_1,u_2$ corresponding to the tropical graphs $\Gamma_1,\Gamma_2$ shown in the left and right 
of Figure \ref{pants2} are equal, since they both correspond to Maslov index four 
disks with boundary in the  anti-diagonal.  

To further explain this point, note that the realizability theorem
for tropical Lagrangians in Theorem \ref{thm:realize} does not produce Lagrangians $L$ that are translationally invariant over the edges $\eps \in \Edge(\Lambda)$, but rather Lagrangians $L$ that are asymptotic to codimension one tori $ \{a \} \times T^{n-1}$ at the ends of the edges $\eps \in \Edge(\Lambda)$.  The inverse images
of the interiors of  edges in the Lagrangian are not translationally invariant, and the area $A(u)$ of the holomorphic strips $u: C \to X$ over the edges $\eps \in \Edge(\Lambda)$  is not proportional to the length of $\eps \in \Edge(\Lambda)$.  

On the other hand, 
the Maslov index may be read off from the 
pieces via the formula 
\[ I(u) = \sum_{v \in \Ver(\bGamma)} I(u_v) .\] 
Since the dimension of the relative moduli space of disks resp. spheres containing $u_v$ is $2+ I(u_v) - \dim(\aut(C_v))$, 
resp. $4 + I(u_v) - \dim(\aut(C_v))$ we see that 
\[ I(u_v) = \begin{cases} 2  & \text{for the closed univalent vertices}  \\
 0 & \text{for bivalent vertices and open univalent vertices and}  \\
 -2 & \text{for the trivalent vertices} \end{cases} \] 
 For example,  closed univalent vertices have automorphism group of dimension two and a dimension two moduli space, 
 and so must be Maslov index two.   Similarly, for the disks bounding the pairs of pants, 
 the moduli space computed in Lemma \ref{lem:stripsinpants} is dimension zero, so by Riemann-Roch $\dim(L) + I(u_v)  =0 $
 so $I(u_v) = -2$.   In particular, the Maslov number of a broken disk for an elementary decomposition is 
{\em not} twice the number of times the graph intersects the boundary of the dual complex. 
\end{remark}

\subsection{Collisions on the boundary}

We want to show that for a good choice of cutting data, all of the graphs 
of Maslov-index-two broken disks with valence at least two have vertices in the interior of the dual complex; that is, the collisions occur in the interior in the
sense of Definition \ref{def:ci}.  We suppose that the polyhedral decomposition $\cP$ of $B$ is the intersection of polyhedral decompositions $\cP_0$
and $\cP_1$, where $\cP_0$ is the collection of polytopes obtained by a sequence of cuts parallel to the facets of $\Phi(X)$ as in Figure \ref{fig:2p}.

\begin{figure}[ht]
  \begin{center}
    \scalebox{.8}{
\begingroup%
  \makeatletter%
  \providecommand\color[2][]{%
    \errmessage{(Inkscape) Color is used for the text in Inkscape, but the package 'color.sty' is not loaded}%
    \renewcommand\color[2][]{}%
  }%
  \providecommand\transparent[1]{%
    \errmessage{(Inkscape) Transparency is used (non-zero) for the text in Inkscape, but the package 'transparent.sty' is not loaded}%
    \renewcommand\transparent[1]{}%
  }%
  \providecommand\rotatebox[2]{#2}%
  \newcommand*\fsize{\dimexpr\f@size pt\relax}%
  \newcommand*\lineheight[1]{\fontsize{\fsize}{#1\fsize}\selectfont}%
  \ifx\svgwidth\undefined%
    \setlength{\unitlength}{515.22336464bp}%
    \ifx\svgscale\undefined%
      \relax%
    \else%
      \setlength{\unitlength}{\unitlength * \real{\svgscale}}%
    \fi%
  \else%
    \setlength{\unitlength}{\svgwidth}%
  \fi%
  \global\let\svgwidth\undefined%
  \global\let\svgscale\undefined%
  \makeatother%
  \begin{picture}(1,0.26360103)%
    \lineheight{1}%
    \setlength\tabcolsep{0pt}%
    \put(0,0){\includegraphics[width=\unitlength,page=1]{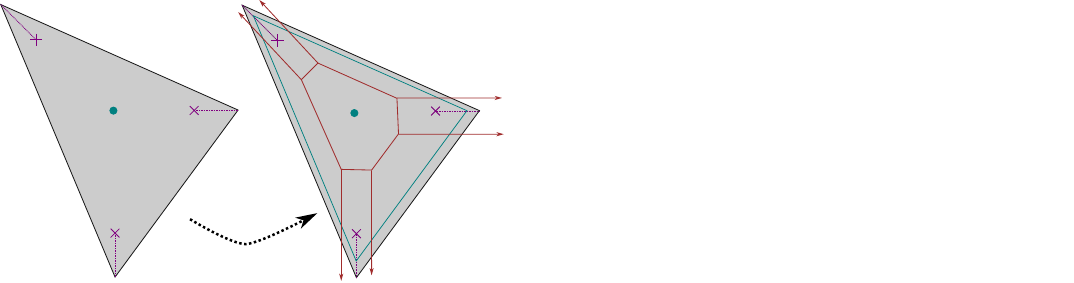}}%
    \put(0.21222367,0.01212206){\color[rgb]{0,0,0}\makebox(0,0)[lt]{\lineheight{1.25}\smash{\begin{tabular}[t]{l}Cut\end{tabular}}}}%
    \put(0,0){\includegraphics[width=\unitlength,page=2]{2p.pdf}}%
    \put(0.65995227,0.03892623){\color[rgb]{0.78431373,0.21568627,0.21568627}\makebox(0,0)[lt]{\lineheight{1.25}\smash{\begin{tabular}[t]{l}$\PP_{\on{in}}$\end{tabular}}}}%
    \put(0.39115157,0.0499392){\color[rgb]{0.10196078,0.10196078,0.10196078}\makebox(0,0)[lt]{\lineheight{1.25}\smash{\begin{tabular}[t]{l}$\PP$\end{tabular}}}}%
    \put(0,0){\includegraphics[width=\unitlength,page=3]{2p.pdf}}%
    \put(0.87411168,0.04203492){\color[rgb]{0,0.50196078,0.50196078}\makebox(0,0)[lt]{\lineheight{1.25}\smash{\begin{tabular}[t]{l}$\PP_{\on{ann}}$\end{tabular}}}}%
  \end{picture}%
\endgroup%
}
  \end{center}
  \caption{A polyhedral decomposition $\PP$ on $\Bl_8\P^2$ defined as 
     the intersection of $\PP_{\on{ann}}$ and $\PP_{\on{in}}$.}
  \label{fig:2p}
\end{figure}

Given a family of cutting datum for $\cP_0$ and $\cP_1$ we obtained a family of cutting data for $\cP$ as explained in \cite[Lemma 5.5]{vw:at}.  Multiplying the cutting data for $\cP_0$
by a parameter $\rho$ gives a family of dual complexes $B^\dual_\rho$ depending on the 
parameter $\rho$. 

\begin{lemma} \label{lem:univbound}  For any $\rho_0 > 0$, there exists a real number $T > 0 $ so that the number of tropical graphs  $\Gamma$ for gluing data $\rho > \rho_0$ of broken Maslov-index-two disks $u = (u_v)$ is at most $T$. 
\end{lemma}

\begin{proof}   By monotonicity, the set of broken Maslov-index-two map has energy, hence area, bounded by some number $E_0$.  On the other hand, suppose that the broken map $u$ is non-constant.  If $P \in \PP$ is top dimensional polytope, then any component $u_v$ with $P(v) = P $ has positive area, since it is non-constant.  Hence the number of vertices of $\Gamma$ mapping to 
vertices of $B^\dual$ that $\Gamma$ is universally bounded.  It follows in particular that there exists a number $T_0$
so that the number of focus-focus values $b^\dual \in B^\dual, b \in B^{\foc}$ that $\Gamma$ intersects is less than $T_0$, for any graph $\Gamma$.  In particular, the number of univalent vertices 
$v \in \Ver(\Gamma) = 1$ is universally bounded for tropical graphs $\Gamma$ of Maslov-index-two broken disks.  The bound on the number of combinatorial types follows, since each type has underlying graph that is a tree. 
\end{proof}

We show that any rigid tropical graph consists of a collection of rigid graphs in the interior of the dual complex a graph in the boundary of the dual complex:  Any tropical graph $\Gamma$ in $B^\dual$  admits a decomposition 
\[ \Gamma = \on{int}(\Gamma) \cup \partial \Gamma \] 
where $\on{int} \Gamma $ resp. $\partial \Gamma $ maps to $\on{int}(B^\dual)$ resp. $\partial B^\dual$.

\begin{theorem} \label{thm:trunc} Let $X \to B$ be a compact almost toric manifold without elliptic singularities and $\cP  = \{ P \}$
an elementary polyhedral decomposition.  There exists $\rho_0 > 0 $ so that for any $\rho> \rho_0$, the tropical graph $\Gamma$ of any broken Maslov-index-two disk with boundary in $L$  consists of a connected piece $\on{int}(\Gamma)$ mapping to the interior $\on{int}(B^\dual)$
and a connected piece $\partial \Gamma$ mapping to the boundary $\partial B^\dual$, joined by a single edge $e$ whose direction $\cT(e)$ is normal to one of the facets of $\Phi(X)$.
\end{theorem}
   
We begin with some remarks on holomorphic spheres in Hirzebruch surfaces.  By a Hirzebruch surface we mean the projectivization 
\[ X = \P(\mO(d) \oplus \mO(0)), d \in \Z \] 
of a line bundle over $\P^1$.    The space $X$ has a natural torus action, arising from the action of scalar multiplication on first factor, and a lift of the circle action on the base, for which the moment polytope can be taken to be the convex hull of the points $(0,0), (0,1), (1,1), (1,1+d).$
By an {\em orbifold Hirzebruch surface}, we mean any 
projective toric surface $X$ whose moment polytope $\Phi(X)$ is a trapezoid.   We may view such a surface as the projectivization of an orbifold bundle over the orbifold projective line with two orbifold points. 

\begin{lemma} \label{lem:boundver} Suppose that $P \in \cP$ is such that either
\begin{enumerate}
\item  {\rm (Boundary edge piece)} $P \cap \Phi(X)$ is a trapezoid that 
intersects $\partial \Phi(X)$ in a single face of $P \cap \Phi(X)$,  and $X_P $ is the corresponding (possibly orbifold) Hirzebruch surface and $Y$  is the union of fibers over $0,\infty \in \P^1$ and the section at infinity, isomorphic to $\P^1$; or
\item  {\rm (Boundary vertex piece)}  $P \cap \Phi(X)$ is a  rectangle intersecting $\partial \Phi(X)$ in two faces
of $P \cap \Phi(X)$, in which case, $X_P = \P^1 \times \P^1$ and $Y = (\P^1 \times \{ \infty \})  \cup (\{ \infty \} \times \P^1)$ is the divisor correponding to two adjacent faces of  $P \cap \Phi(X)$.
\end{enumerate}
Let $\bGamma$ be a rigid type with a vertex $v$ whose adjacent edges 
$e_1,\ldots, e_k$ have the property that at least one of which, say $e_1$, has $P(e_1)$ equal to one of the faces 
not meeting $\partial \Phi(X)$.   Let $\bGamma(v)$ denote the graph consisting of the vertex $v$ and adjacent edges, and $\M_{\bGamma(v)}(\XX_{P(v)})$ the moduli space  of relative 
maps to $\XX_{P(v)}$ without constraints.  Let $\bGamma_1,\ldots,\bGamma_k$
denote the graphs obtained from $\bGamma$ by breaking the edges $e_1,\ldots, e_k$, in addition to $\Gamma(v)$, and $\M_{\bGamma_i}(\XX)$ the corresponding moduli spaces. 
Then $\M_{\bGamma(v)}(\XX_{P(v)})$ is dimension at least $2k$, and 
the moduli spaces $\M_{\bGamma_i}(\XX)$ are dimension zero. 
Furthermore, in the case of a (Boundary vertex piece), the vertex $v$
is adjacent to a single edge $e \in \Edge(\Gamma)$.
\end{lemma}

\begin{proof}  Consider the first case of a (Boundary edge piece).   Suppose first that $d$ is integral, so that $X_P$ is a smooth Hirzebruch surface.   Each edge corresponds to an intersection of $u_v$ with the boundary divisor $Y$, and the union of $Y$ with the zero section $Z$ is anti-canonical divisor of $X$:
\begin{equation} c_1(X) = [Y]^\dual + [Z]^\dual . \end{equation}
Riemann-Roch implies that the dimension of the moduli space of 
relative maps to $\XX_{P(v)}$ of type $\bGamma(v)$ is
\begin{equation} \label{eq:rr}
\dim(\M_{\bGamma(v)}(\XX_{P(v)})) = \dim(X) - 6 + 2k +  c_1(u) - \sum_{i=1}^k \codim P(e_i) . \end{equation}
The first Chern number
$c_1(u)$ is the sum of the intersection multiplicities with $Y$ and $Z$.  
After removing the intersection points with $Z$, we view
$u$ as a map to the interior of the toric variety $\ol{\XX}_{P(v)}$ with polytope $\Delta := P(v) \cap \Phi(X)$.
Each puncture maps to the boundary $Y$, and represents an intersection of multiplicity at least two if it maps to a corner.  Hence 
\begin{equation} \label{eq:Ydual} [Y]^\dual . [u_v] - \sum_{i=1}^k \codim P(e_i)  \ge 0 . \end{equation}
Let 
\[ \ti{\mu}_i \in \R^2 , i = 1,\ldots, \ti{k} \] 
denote the directions at the punctures.   Thus in particular we may assume $\ti{\mu}_i = \cT(e_1)$ has 
non-trivial second component.   The balancing condition \eqref{eq:balance} reads
\[ \sum_{i=1}^{\ti{k}} \ti{\mu_i} = 0 . \] 
This implies that at least one of the other vectors $\ti{\mu}_i, i > 1$ has positive second component.  We take the normal directions corresponding to the boundary facets of $\Delta$  to be $(0,1),(0,-1),(1,0),(-1,-d)$. Any edge direction $\ti{\mu}_i$ is a combination of the normal vectors to adjacent facets to the polytope $P(e_i)$.    Since the only facet of $P(v)$
whose normal direction has positive first component is bottom facet, the image of the zero section $Z$, the intersection number of $u$ with $Z$ is positive.  It follows that $2[u].[Z]$ is at least two.  Hence 
\begin{equation} \label{eq:Zpos}
\dim(X) - 6 + 2[u].[Z] \ge 0 . \end{equation}
Combining  \eqref{eq:rr} , \eqref{eq:Zpos} and \eqref{eq:Ydual} gives
\[ \dim  \M_{\bGamma(v)}(\XX_{P(v)}) \ge 2k . \] 
Rigidity of the moduli space
of maps of type $\Gamma$ implies that the 
moduli spaces for the graphs on the other side of the edges $e_1,\ldots, e_k$ are rigid
\[ \dim \M_{\bGamma_i}(\XX) = 0  \]
since the fiber product of  $\M_{\bGamma(v)}(\XX_{P(v)})$
with the spaces $\M_{\bGamma_i}(\XX)$ has expected dimension zero.

The general case follows from resolution of singularities.  Let 
$\ti{\XX}_{P(v)}$ be a smooth toric resolution of singularities of $\ol{\XX}_{P(v)}$ obtained by repeated blow-ups at the inverse images of the vertices of the moment polytope, corresponding to a polytope $Q$.  For suitable choices of blow-ups, the directions $\mu_i$ of each edge  are normal to some facet.  As before, there must be at least one 
edge either meeting the bottom facet, or mapping to one of the facets that is collapsed to the bottom vertices
under the blow-down map. The resulting moduli space is contained in the codimension two locus of 
relative maps whose evaluation at $e_i$ is the inverse image of vertex of $\Delta$, and the same argument holds.  

The case of a (Boundary vertex piece) is similar.  The dimension of the moduli space of genus zero 
maps to $\P^1 \times \P^1$ of bidegree $(d_1,d_2) \in \Z_{\ge 0}^2$ is 
\[ \dim \M(\P^1 \times \P^1, (d_1,d_2))  = 4 + 4(d_1 + d_2) - 6 = 4(d_1 + d_2) - 2 . \] 
Such maps cannot be made rigid by adding $2(d_1 + d_2)$ constraints (such as fixing the evaluation map, or requiring additional tangency conditions), unless $(d_1,d_2) \in \{ (1,0), (0,1) \}$.
 \end{proof}

It remains to identify the kind of graphs that can appear in the boundary of the dual complex.   We already identified one type of such graph, namely a cylinder that hits the boundary in a perpendicular direction to one on the facets of the moment polytope.

\begin{lemma}
    \label{lem:edge}  
 Suppose $P \cap \Phi(X)$ is a trapezoid 
intersects $\partial \Phi(X)$ in a single face of $P \cap \Phi(X)$,  and $X_P = \P(\mO(d) \oplus \mO(0))$ is a Hirzebruch surface with $d \ge 0$ and  $Y = (\P^1 \times \{ \infty \}) \cup (\{ 0 \} \times \P^1 ) \cup (\{ \infty \} \times \P^1)$.  Let $\Gamma$ be a graph with a single vertex $v$, and exactly two adjacent edges  $e_1,e_2$ with $e_1$ mapping to the corner disjoint from $\partial \Phi(X)$ so that $\t_{P(e_1)} \cong \R^2$.   Then the second edge $e_2$ maps to a face
$P(e_2)$ of $P$ with $\t_{P(e_2)} \cong \R$, and the direction ${\cT}(e_2)$
is the projection of ${\cT}(e_1) \in \R^2$ onto the first factor via the canonical map $\t_{P(e_1)} / \R \nu\cong \t_{P(e_2)} $, where $\nu = (0,1)$ is the normal to the facet of the moment polytope that is not in $Y$, i.e. not a relative divisor.  
\end{lemma}

A picture of the situation is shown in Figure \ref{fig:hitbound}.

\begin{figure}[ht] \begin{center} \scalebox{.6}{
\begingroup%
  \makeatletter%
  \providecommand\color[2][]{%
    \errmessage{(Inkscape) Color is used for the text in Inkscape, but the package 'color.sty' is not loaded}%
    \renewcommand\color[2][]{}%
  }%
  \providecommand\transparent[1]{%
    \errmessage{(Inkscape) Transparency is used (non-zero) for the text in Inkscape, but the package 'transparent.sty' is not loaded}%
    \renewcommand\transparent[1]{}%
  }%
  \providecommand\rotatebox[2]{#2}%
  \newcommand*\fsize{\dimexpr\f@size pt\relax}%
  \newcommand*\lineheight[1]{\fontsize{\fsize}{#1\fsize}\selectfont}%
  \ifx\svgwidth\undefined%
    \setlength{\unitlength}{293.92465426bp}%
    \ifx\svgscale\undefined%
      \relax%
    \else%
      \setlength{\unitlength}{\unitlength * \real{\svgscale}}%
    \fi%
  \else%
    \setlength{\unitlength}{\svgwidth}%
  \fi%
  \global\let\svgwidth\undefined%
  \global\let\svgscale\undefined%
  \makeatother%
  \begin{picture}(1,0.4563334)%
    \lineheight{1}%
    \setlength\tabcolsep{0pt}%
    \put(0,0){\includegraphics[width=\unitlength,page=1]{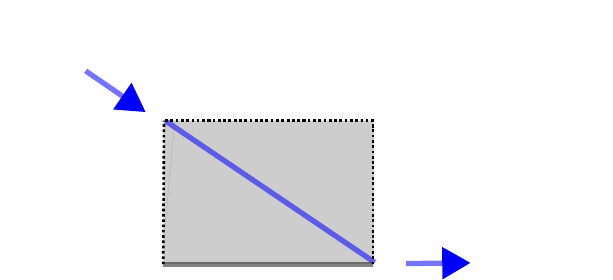}}%
    \put(-0.00447043,0.41087921){\color[rgb]{0,0,1}\transparent{0.53157902}\makebox(0,0)[lt]{\lineheight{1.25}\smash{\begin{tabular}[t]{l}$\cT(e_1)$ \end{tabular}}}}%
    \put(0.66869788,0.08653662){\color[rgb]{0,0,1}\transparent{0.53157902}\makebox(0,0)[lt]{\lineheight{1.25}\smash{\begin{tabular}[t]{l}$\cT(e_2)$ \end{tabular}}}}%
  \end{picture}%
\endgroup%
}
\end{center} \caption{Curves hitting the boundary and continuing in a projected direction} \label{fig:hitbound} \end{figure}

\begin{proof}   The proof is a computation with intersection multiplicities.  
    Let $u:C \to X$ be a map with such a graph $\Gamma$.   Let $X' = \P^2$ equipped with a toric 
    birational equivalence with $X$, and $u': C \to X'$ the map given by composing with the birational equivalence.  The map  $u'$ intersects the boundary of $X'$ exactly twice.  The image is therefore a line intersecting a torus fixed-point.  Such a curve given by a one-parameter subgroup to the interior $(\C^\times)^2$. 
    Suppose the direction of the first edge is   ${\cT}(e_1) = (\mu_1,\mu_2) \in \Z^2$, in the standard coordinates for which the moment polytope is the trapezoid with vertices
    at $(0,0), (0,1), (1,1), (0, 1+d).$  Without loss of generality $\mu_1 < 0 $ and the intersection numbers of $u$ with $\{ \infty \} \times \P^1$ resp. $  \P^1 \times \{ \infty\}$ are $-\mu_1 $ resp. $-\mu_2$.
    The other intersection of $u$ with the toric boundary is at $(0,0)$,
    with intersection multiplicities $-\mu_1,-\mu_2 + d\mu_1$.  Now $(0,0)$
    is in the pre-image of the hyperplane $P(e_2)$ with normal vector $(1,0)$, and the intersection multiplicity of $u$ with $\{ 0 \} \times \P^1$ is the projection $-\mu_1$ of ${\cT}(e_2)$ onto the first factor $\R$, as claimed.
    \end{proof}

   \begin{proof}[Proof of Theorem \ref{thm:trunc}]
Suppose that $\Gamma$ is a tropical graph  in $B^\dual$ containing edges 
$e_1,\ldots, e_k$ that meet in the boundary $\partial B^\dual$, as shown in Figure  \ref{fig:allinterior}.  We may choose coordinates so that 
\[ \partial \ti{B}^{\dual} = \ti{B}^{\dual} \cap ( \R \times \{ -\rho \} ) .\]
In particular the tangent space to the boundary has direction $(1,0)$.
Suppose without loss of generality that the edges $e_1,\ldots, e_\ell$
have directions $\cT(e_1),\ldots, \cT(e_\ell)$ with positive projection onto $T\partial \ti{B}^{\dual} $,
and $e_{\ell+1},\ldots, e_k$ have directions with negative projection.
Choose $\rho_0$ so that for $\rho > \rho_0$, the intersections of 
$\R \times \{ -\rho \} $ with $e_1,\dots, e_\ell$ are to the right of the
intersections of $\R \times \{ - \rho\}$ with $e_{\ell+1},\ldots, e_{k}$, as in Figure \ref{fig:allinterior}.
\begin{figure}[ht]\begin{center} 
\scalebox{.8}{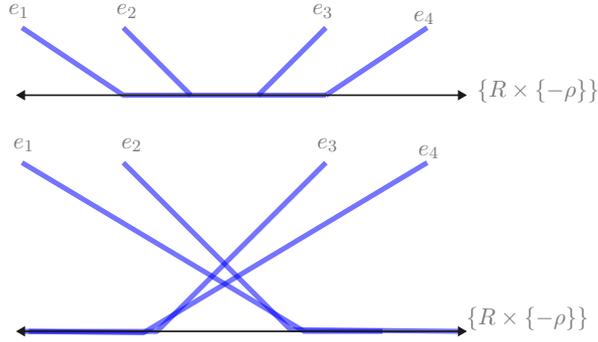}
\end{center} 
\caption{Removing collisions on the interior by increasing $\rho$} 
\label{fig:allinterior}
\end{figure}
Any vertices involve collisions between the edges $e_1,\ldots, e_\ell$ or between 
$e_{\ell+1}, \ldots, e_k$.    Any edge direction resulting from such a collision has positive resp. negative 
projection onto the first component.   Since the two types of edges never intersect, there is no connected graph which involves combinations of these edges, which is a contradiction.   Repeating the process for each graph $\Gamma$ (without tropical structure) produces a constant $\rho_0$ so that for $\rho > \rho_0$, all collisions are in the interior. 
\end{proof}

%

\begin{example} Continuing Example \ref{ex:invim}, we give an example where
bivalent vertices map to the boundary of the dual complex.  
Let $p_1$ be a constraint with moment image $\Phi(p_1) = (-\eps, -\eps), \eps > 0$. 
In order to compute the multiplicities in this case, we need to orient the Lagrangian so that the 
direction of the edge containing this point is $(-1,1)$.

There are four tropical graphs $\Gamma_1,\Gamma_2,\Gamma_3,\Gamma_4$ (shown in the lower part of Figure \ref{fig:b2p2wdual}) contributing to the count: There are two tropical graphs $\Gamma_1,\Gamma_2$ with directions $(1,1)$ and $(-1,-1)$, respectively.
Each  of $\Gamma_1,\Gamma_2$ has two vertices, with multiplicities $m(v_1) = -1, m(v_2) = 1$.
The total contribution of each graph is therefore 
\[ m(\Gamma_k) = m(v_1)m(v_2) = -1 \] 
to the count.  On the other hand, there is a graph $\Gamma_3$ with
a bivalent vertex at $\Phi(p_1)$ and edges $e_1,e_2$ with directions
$1$ (from the Lagrangian) and $(0,1)$.  The determinant of the matrix formed by these directions is one and so 
\[ m(\Gamma_3 ) = m(v_1)m(v_2)m(v_3) = 1 (-1) 1 /2 = -1/2 . \]
The final graph $\Gamma_4$ has a bivalent vertex at $\Phi(p_1)$ and edges $e_1,e_2$ with directions
$1$ (from the Lagrangian) and $(2,1)$.   The contribution from the graph is 
\[ m(\Gamma_4) =  m(v_1)m(v_2)m(v_3) = 1 (3) 1 /2 = 3/2 \]

The total disk count is 
therefore 
\[ W_L = (-1) + (-1) + (1/2)  + (-3/2)  = -1 . \]
\end{example}

\begin{figure}[ht]
    \centering
    \scalebox{.5}{
\begingroup%
  \makeatletter%
  \providecommand\color[2][]{%
    \errmessage{(Inkscape) Color is used for the text in Inkscape, but the package 'color.sty' is not loaded}%
    \renewcommand\color[2][]{}%
  }%
  \providecommand\transparent[1]{%
    \errmessage{(Inkscape) Transparency is used (non-zero) for the text in Inkscape, but the package 'transparent.sty' is not loaded}%
    \renewcommand\transparent[1]{}%
  }%
  \providecommand\rotatebox[2]{#2}%
  \newcommand*\fsize{\dimexpr\f@size pt\relax}%
  \newcommand*\lineheight[1]{\fontsize{\fsize}{#1\fsize}\selectfont}%
  \ifx\svgwidth\undefined%
    \setlength{\unitlength}{834.30867833bp}%
    \ifx\svgscale\undefined%
      \relax%
    \else%
      \setlength{\unitlength}{\unitlength * \real{\svgscale}}%
    \fi%
  \else%
    \setlength{\unitlength}{\svgwidth}%
  \fi%
  \global\let\svgwidth\undefined%
  \global\let\svgscale\undefined%
  \makeatother%
  \begin{picture}(1,0.21082365)%
    \lineheight{1}%
    \setlength\tabcolsep{0pt}%
    \put(0,0){\includegraphics[width=\unitlength,page=1]{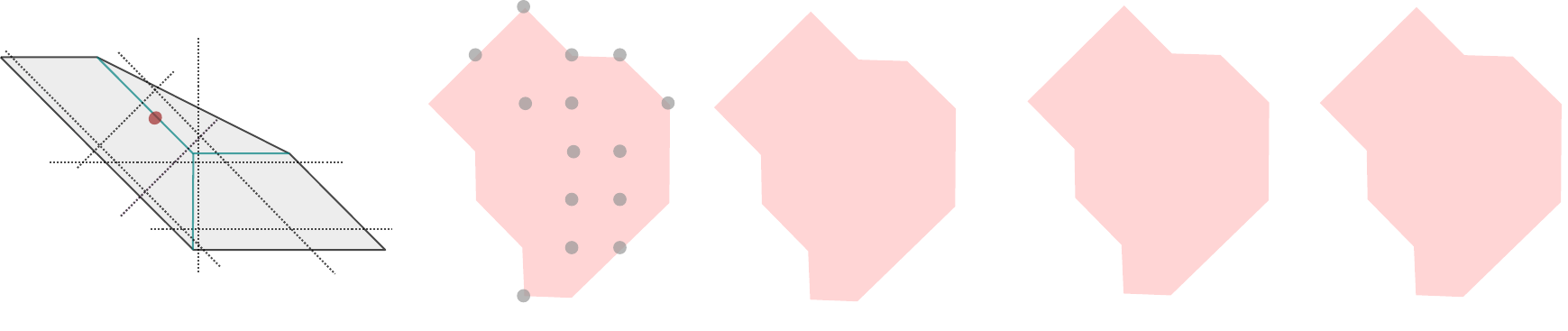}}%
    \put(0.07242946,0.12069424){\color[rgb]{0.87058824,0.52941176,0.52941176}\makebox(0,0)[lt]{\lineheight{1.25}\smash{\begin{tabular}[t]{l}$p_3$\end{tabular}}}}%
    \put(0.34979755,0.00164766){\color[rgb]{0,0,0}\makebox(0,0)[lt]{\lineheight{1.25}\smash{\begin{tabular}[t]{l}$w = (-1) + (-1) + (-1/2) + (3/2) = -1 $\end{tabular}}}}%
    \put(0,0){\includegraphics[width=\unitlength,page=2]{b2p2wdual2.pdf}}%
  \end{picture}%
\endgroup%
}
    \caption{Cartoon diagrams of Maslov-index-two disks in the degree seven del Pezzo}
    \label{fig:b2p2wdual}
\end{figure}

\begin{remark}  \label{rem:htrid}  In Definition \ref{def:mv} there is one additional vertex which is generic.  This vertex has two adjacent open edges and one adjacent closed edge.  Consider trivalent vertices $v \in \Ver_\white(\Gamma)$ mapping to an  edge $\eps$ of $L$ with open edges  $e_\white , e_\white' \subset \eps$ and closed edge $ e_\black$ with directions satisfying the balancing condition 
\[ \cT(e_\white') + \cT(e_\white) + \cT(e_\black)  + \cT(r(e_\black)) = 0 . \]

\begin{figure}[ht]
\begin{center} 
\scalebox{1.1}{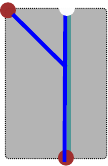}
\end{center}
    \caption{The half-trident graph}
    \label{fig:htrid}
\end{figure}

The tropical graph $\Gamma$  is  half of a 
tropical graph $\ti{\Gamma}$ with a four-valent vertex $\ti{v} \in \Ver(\ti{\Gamma)}$ that corresponds to a holomorphic sphere.   As explained in \cite{vw:trop}, for holomorphic sphere counts, four-valent vertices $v \in \Ver(\Gamma), |v| = 4$ can be be avoided by choosing
the constraints generically, so after perturbation two three-valent vertices $v_1,v_2 \in \Ver(\Gamma)$ appear rather than a
single four-valent vertex.  However, in the case of open-closed graphs $\Gamma$, the position of the open edges $e \in \Edge_\white(\Gamma)$ cannot be chosen generically because  the Lagrangian boundary condition cannot in general be perturbed.     As a result, the diagram in Figure 
\ref{fig:htrid} occurs even generically.   However, this vertex does not include any of our examples, as a little thought shows that the contribution of such a vertex to the Maslov index is zero, and since there are three punctures, the Maslov index of the corresponding holomorphic disk must be at least four.  
\end{remark}

\section{Potentials of Lagrangian spheres in del Pezzo surfaces}

In this section we complete the proof of Theorem \ref{thm:dp}, in the sense that we show
that every integer eigenvalue in Table \ref{tab:eigen} with non-maximal modulus is the potential of a Lagrangian sphere.   We assume the completeness of Table \ref{tab:eigen}, which we plan to show in the sequel to this paper on the surjectivity of the corresponding open-closed maps.   The reader will note that, in each case, there is at most one such eigenvalue, except for the del
Pezzo of degree six whose associated Manin system
is semisimple but not simple, in which case there are two. 
 
\begin{proof}[Proof of Theorem \ref{thm:dp}]
Let $X$ be a del Pezzo surface equipped with its monotone symplectic form. We deal with the various cases in descending degree.  The case that $X$ is the product of projective lines is trivial:  The minimal Chern number of $X$ is two, and the minimal Maslov number is four.  It follows that the potential of the anti-diagonal $L$ is
$W_L  = 0$, as there are no Maslov-index-two disks.     The other spectral values 
\[ \pm 4 \in \Spec(c_1 \star)  \] 
have maximal modulus $| \pm 4 | = 4$ among all spectral values.
%
The case of the del Pezzo of degree eight, that is, the blow-up of the projective plane, is also trivial, as there are no integer spectral values.
The case of the del Pezzo of degree seven, that is, the blow-up of $\P^2$ at two points, was treated in Example \ref{ex:invim}.  
%
The degree six  case was partly treated in 
Example \ref{ex:3p2}.    Figure \ref{fig:3p2pot} shows that three Lagrangians (with trivial local systems) giving rise to the spectral values 
\[ 6,-2,-3 \in \Spec(c_1 \star) \] 
respectively.  The monotone torus $L_1$
with its trivial local system gives rise to the critical value of the potential $W_{L_1}(1) = 6$, with the 
six tropical graphs corresponding to the Cho-Oh disks shown on the left in Figure \ref{fig:3p2pot}.
A Lagrangian sphere $L_2$ with $W_{L_2} = -2$ is shown in the middle picture in 
Figure \ref{fig:3p2pot}.  The tropical graph $\Lambda$ is a single segment with two univalent vertices; this Lagrangian $L_2$ together with the Lagrangians $L_2',L_2''$ for the rotations of the graph $\Lambda$ by $2\pi/3$
and $4\pi/3$ form the $A_2$ part of a Manin system of Lagrangians. 
\begin{figure}[ht]
    \centering
    \scalebox{.3}{
\begingroup%
  \makeatletter%
  \providecommand\color[2][]{%
    \errmessage{(Inkscape) Color is used for the text in Inkscape, but the package 'color.sty' is not loaded}%
    \renewcommand\color[2][]{}%
  }%
  \providecommand\transparent[1]{%
    \errmessage{(Inkscape) Transparency is used (non-zero) for the text in Inkscape, but the package 'transparent.sty' is not loaded}%
    \renewcommand\transparent[1]{}%
  }%
  \providecommand\rotatebox[2]{#2}%
  \newcommand*\fsize{\dimexpr\f@size pt\relax}%
  \newcommand*\lineheight[1]{\fontsize{\fsize}{#1\fsize}\selectfont}%
  \ifx\svgwidth\undefined%
    \setlength{\unitlength}{1220.28433228bp}%
    \ifx\svgscale\undefined%
      \relax%
    \else%
      \setlength{\unitlength}{\unitlength * \real{\svgscale}}%
    \fi%
  \else%
    \setlength{\unitlength}{\svgwidth}%
  \fi%
  \global\let\svgwidth\undefined%
  \global\let\svgscale\undefined%
  \makeatother%
  \begin{picture}(1,0.46234698)%
    \lineheight{1}%
    \setlength\tabcolsep{0pt}%
    \put(0,0){\includegraphics[width=\unitlength,page=1]{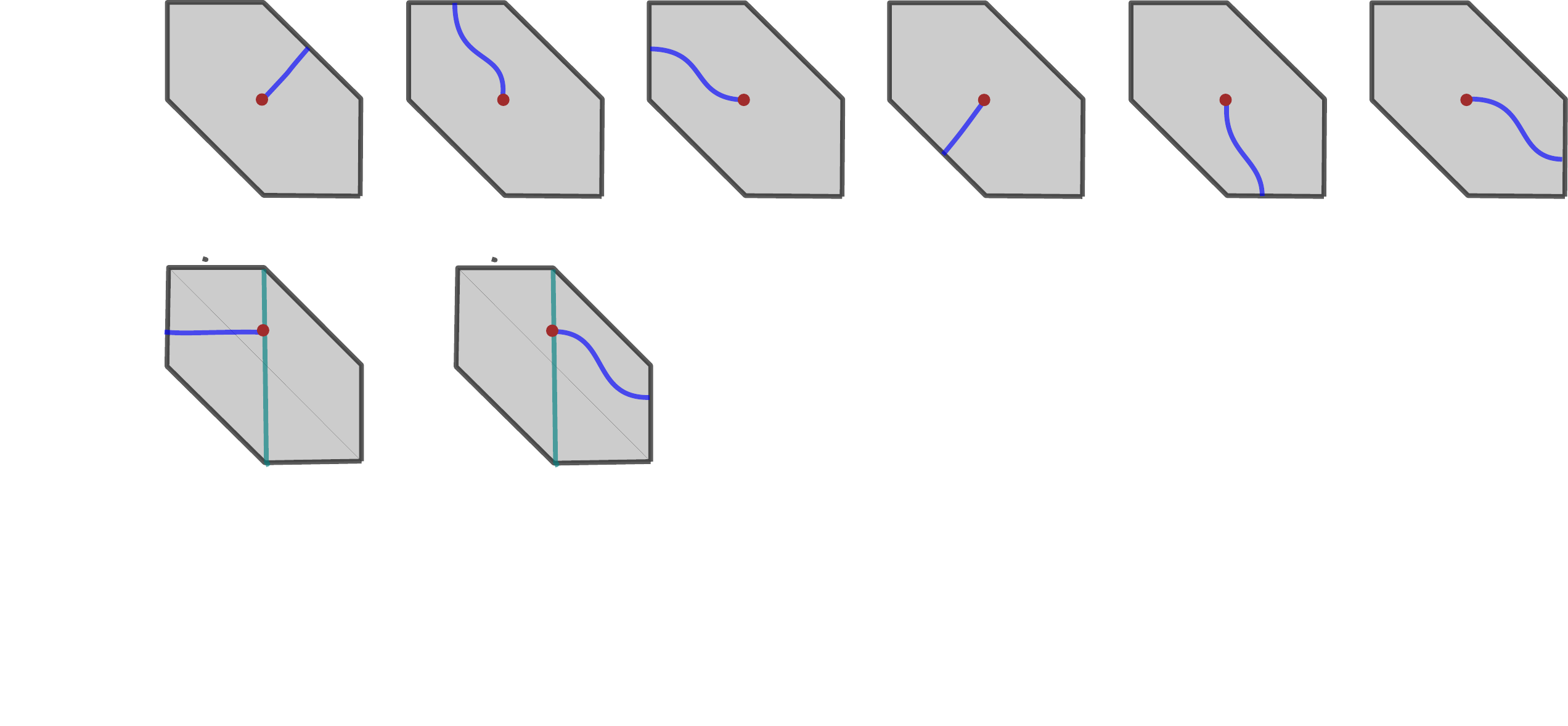}}%
    \put(0.01474814,0.41456461){\makebox(0,0)[lt]{\lineheight{1.25}\smash{\begin{tabular}[t]{l}$W_L(1) = 6$\end{tabular}}}}%
    \put(-0.00074521,0.2611531){\makebox(0,0)[lt]{\lineheight{1.25}\smash{\begin{tabular}[t]{l}$W_L = -2$\end{tabular}}}}%
    \put(0.01462006,0.09213511){\makebox(0,0)[lt]{\lineheight{1.25}\smash{\begin{tabular}[t]{l}$W_L = -3$\end{tabular}}}}%
    \put(0,0){\includegraphics[width=\unitlength,page=2]{b3p2_pot.pdf}}%
    \put(0.71659678,0.0386591){\color[rgb]{0.10196078,0.10196078,0.10196078}\transparent{0.00980392}\makebox(0,0)[lt]{\lineheight{1.25}\smash{\begin{tabular}[t]{l}w\\w\end{tabular}}}}%
    \put(0,0){\includegraphics[width=\unitlength,page=3]{b3p2_pot.pdf}}%
  \end{picture}%
\endgroup%
}
      \caption{Cartoon diagrams of Maslov-index-two disks in the degree six del Pezzo bounding (top) the monotone Lagrangian torus (middle) a Lagrangian sphere meeting two fixed points (bottom) the Lagrangian sphere meeting three fixed points}
    \label{fig:3p2pot}
\end{figure}
%
The del Pezzo of degree four was treated in Example \ref{ex:b5p2}.
The del Pezzo of degree two was treated in Example \ref{ex:b7p2}.
%
We show the computation for the del Pezzo of degree one in Figure \ref{fig:b8p2}. The figure shows 
four different types of tropical graphs of broken holomorphic disks of Maslov-index-two bounding a tropical Lagrangian sphere whose total count is $W_L = -60$.  Note that the graph of the Lagrangian lies along one of the cut loci, 
so the affine structures on either side of the Lagrangian are related by a shear.  However, 
the only graphs contributing are those with intial direction perpendicular to the Lagrangian, which are invariant under the shear, and so this complication may be overlooked. 
\begin{itemize}
    \item The first three types $\Gamma_1,\Gamma_2,\Gamma_3$ corresponding to disks $u$ of Maslov index $I(u) = 2$ have
initial direction $\cT(e_1) = (-1,-1)$, and interact with two of the focus-focus values as shown.
\item 
The last type $\Gamma_4$ has initial direction $\cT(e_1) = (1,1)$, and interacts with two of the focus-focus values near the bottom vertex of the moment polytope $\Phi(X)$ shown in the Figure.   
\end{itemize}
There are also two other types $\Gamma_5, \Gamma_6$ of tropical graphs of Maslov-index-two disks, each intersecting a single 
focus-focus value near the bottom vertex of $\Phi(X)$, whose contributions $m(\Gamma_5) = - m(\Gamma_6)$ cancel as explained in 
Proposition 6.6 of \cite{vw:at}.  In particular, the proof of that proposition explains
that graphs $\bGamma_5$ with two edges ending at the same focus-focus value with directions $(1,0)$ cancel with those  types $\bGamma_6$ with a single edge ending at the same focus-focus value with direction $(2,0)$.
%
We leave the cases of del Pezzos of degree five and three to the reader.
%
\end{proof}

\begin{figure}[ht] \begin{center} \scalebox{.4}{\includegraphics{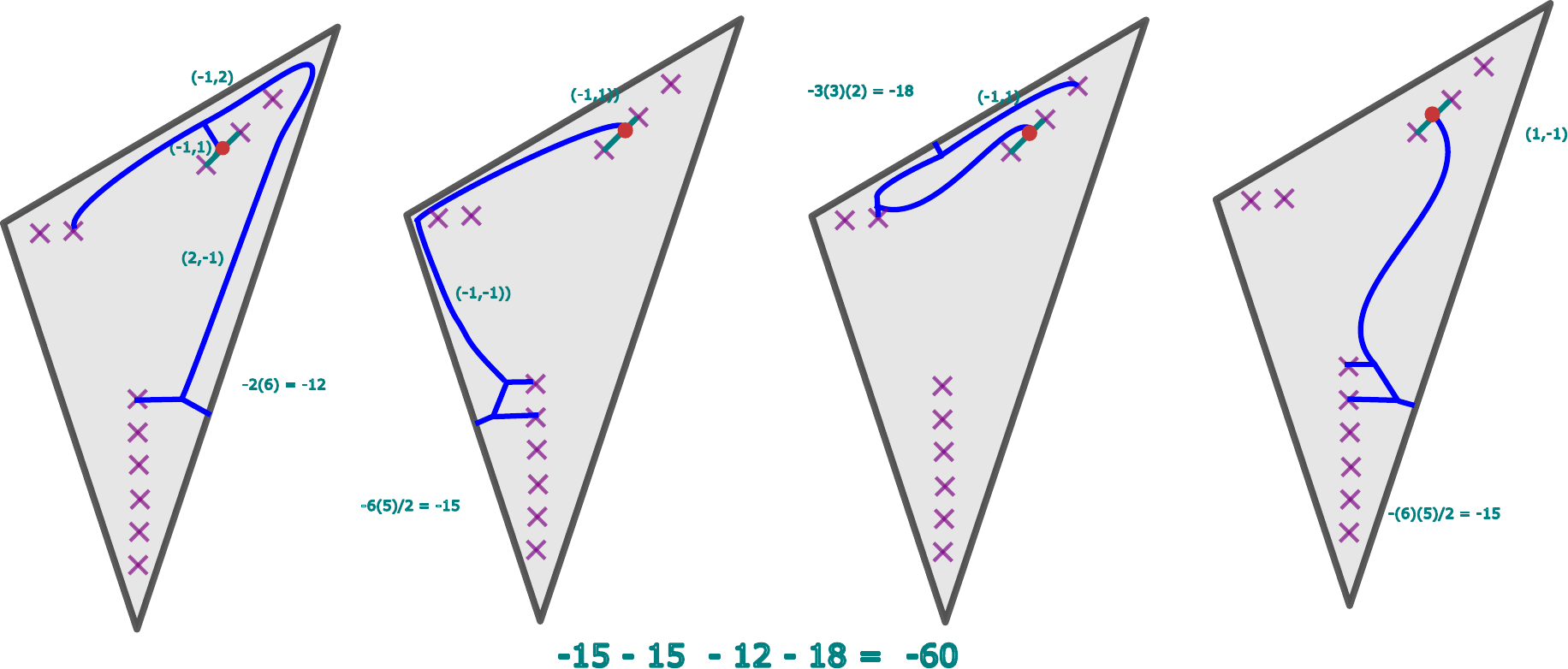}}
\end{center} \caption{Cartoon diagrams for disks bounding a Lagrangian sphere in $\Bl^8 \P^2$} \label{fig:b8p2} \end{figure}

\bibliography{tlag}{}
\bibliographystyle{plain}
\end{document}